\documentclass[12pt]{article}
\usepackage{amsmath}
\usepackage{graphicx,psfrag,epsf}
\usepackage{enumerate}
\usepackage{url} 
\usepackage{float}
\usepackage[table,xcdraw,dvipsnames]{xcolor}

\usepackage{cite}  
\newcommand{\blind}{0}

\usepackage{multirow}
\usepackage{lipsum}
\usepackage{amsfonts}
\usepackage{xr}
\usepackage{cite}      
                  
\usepackage{graphicx}   

\usepackage{amsthm}
\RequirePackage[colorlinks,citecolor=blue,urlcolor=blue]{hyperref}
\usepackage{amssymb,calc}
\usepackage{thmtools}
\usepackage{mathrsfs}
\usepackage{mathtools}
\usepackage{bm}
\usepackage{bbm}

\usepackage{enumerate}
\usepackage{rotating}
\RequirePackage{cleveref}
\usepackage{hyphenat}
 \usepackage[numbers]{natbib}
\usepackage{caption,subcaption}
\usepackage{tikz}
\usepackage{pgfplots}

\usepackage[T1]{fontenc}
\usepackage{textcomp} 
\usepackage{lmodern}

\usepackage[noend]{algpseudocode}
\usepackage{algorithm}
\makeatletter
\renewcommand{\ALG@name}{Table}
\makeatother

\usepackage{xr}
\pgfplotsset{width=10cm}

\tikzset{declare function={gamma(\x)=sqrt(2*pi)*\x^(\x-0.5)*exp(-\x)*exp(1/(12*\x));}}
\tikzset{declare function={tpdf(\x,\nu)=gamma(0.5*(\nu+1))/(sqrt(pi*\nu)*gamma(\nu/2))*(1+\x^2/\nu)^(-(\nu+1)/2);}}
\tikzset{declare function={invgampdf(\x,\a,\b)=(\b/\x)^\a/\x/gamma(\a)*exp(-\b/\x);}}

\newcommand{\nhphantom}[1]{\ifmmode\settowidth{\dimen0}{$#1$}\else\settowidth{\dimen0}{#1}\fi\hspace*{-\dimen0}}

\usetikzlibrary{patterns}
\tikzset{
	hatch distance/.store in=\hatchdistance,
	hatch distance=5pt,
	hatch thickness/.store in=\hatchthickness,
	hatch thickness=0.5pt,
}

\newcommand{\wh}[1]{\widehat{#1}}

\definecolor{pink}{rgb}{0.9, 0.17, 0.31}

\def\C {\,|\:}

\newcommand\mP{\mathcal P}
\newcommand\mX{\mathcal X}

\newcommand\Xm{\wt X^{(m)}}
\renewcommand\d{\mathrm d}
\newcommand\bG{\Gamma}
\renewcommand\P{\mathbb P}

\newcommand\mA{\mathcal A}
\newcommand\Xn{X^{(n)}}
\newcommand\R{\mathbb R}
\renewcommand\b{\bm{\beta}}
\newcommand\e{\mathrm e}
\newcommand\N{\mathbb N}
\renewcommand\bG{\mathbb G}

\newcommand{\wt}[1]{\widetilde{#1}}

\newtheorem{assumption}{Assumption}

\renewcommand{\nhphantom}[1]{\ifmmode\settowidth{\dimen0}{$#1$}\else\settowidth{\dimen0}{#1}\fi\hspace*{-\dimen0}}

\numberwithin{equation}{section}

\newtheorem{thm}{Theorem}[section]

\newtheorem{lem}[thm]{Lemma}
\newtheorem{coro}[thm]{Corollary}

\newtheorem{exa}{Example}

\newtheorem{asm}{Assumption}

\newtheorem{rem}{Remark}

\crefname{thm}{Theorem}{Theorems}
\crefname{prop}{Proposition}{Propositions}
\crefname{lem}{Lemma}{Lemmas}
\crefname{coro}{Corollary}{Corollaries}
\crefname{add}{Addendum}{Addendums}
\crefname{asm}{Assumption}{Assumptions}
\crefname{alg}{Algorithm}{Algorithms}
\crefname{proc}{Procedure}{Procedures}
\crefname{exe}{Exercise}{Exercises}
\crefname{exa}{Example}{Examples}
\crefname{prob}{Problem}{Problems}
\crefname{section}{Section}{Sections}
\crefname{subsection}{Section}{Sections}
\crefname{appendix}{Appendix}{Appendices}

\DeclareMathOperator*{\conv}{\mathchoice{%
	\,\longrightarrow\,}{
	\rightarrow}{
	\rightarrow}{
	\rightarrow}
}

\allowdisplaybreaks[3]

\begin{document}

\def\spacingset#1{\renewcommand{\baselinestretch}%
{#1}\small\normalsize} \spacingset{1}


	\title{\sf Metropolis-Hastings via Classification}
	\author{
		Tetsuya Kaji\footnote{
			Assistant Professor in Econometrics and Statistics; Liew Family Junior Faculty Fellow and Richard N. Rosett Faculty Fellow at the  {\sl \small Booth School of Business, University of Chicago}}  \,\, and    Veronika Ro\v{c}kov\'{a}\footnote{  Associate Professor in Econometrics and Statistics and James S. Kemper Faculty Scholar at the {\sl \small Booth School of Business, University of Chicago}.} 
	} 
	\maketitle

\bigskip
\begin{abstract}
This paper develops a Bayesian computational platform at the interface between posterior sampling and optimization in models whose marginal likelihoods are difficult to evaluate.
Inspired by  {contrastive learning} and Generative Adversarial Networks  (GAN) \cite{gan}, we reframe the likelihood function estimation problem as a classification problem. 
Pitting a Generator, who simulates fake data,  against a Classifier, who tries to distinguish them from the real data, one obtains likelihood (ratio) estimators which can be plugged into the 
Metropolis-Hastings algorithm. The resulting Markov chains generate, at a steady state, samples from an approximate posterior whose asymptotic properties we characterize.
Drawing upon connections with empirical Bayes and Bayesian mis-specification,  we quantify the convergence rate in terms of the contraction speed of the actual posterior and the convergence rate of the Classifier. Asymptotic normality results are also provided which justify the 
inferential potential of our approach. We illustrate the usefulness of our approach on examples which have challenged for existing Bayesian likelihood-free approaches.
\end{abstract}

\noindent%
{\bf Keywords:} {\em Approximate Bayesian Computation, Classification, Generative Adversarial Networks, Likelihood\hyp{}free Inference, Metropolis\hyp{}Hastings Algorithm.}

\spacingset{1.4} 

\vspace{-0.5cm}
\section{Introduction}

Many contemporary statistical applications require inference for models  which are easy to simulate from but whose likelihoods are impossible to evaluate.
This includes  implicit (simulator-based) models  \citep{diggle1984monte}, defined through an underlying generating mechanism,
or models prescribed through intractable likelihood functions. 

Statistical inference for intractable models has traditionally relied on some form of likelihood  approximation (see \cite{gutmann2016bayesian} for a recent excellent survey).
For example, \cite{diggle1984monte} propose kernel log-likelihood estimates  obtained from simulated realizations of an implicit model.
Approximate Bayesian Computation (ABC)  \citep{beaumont2002approximate,pritchard1999population,tavare1997inferring} is another simulation-based approach which obviates the need for likelihood evaluations by (1)  generating fake data $\wt X_\theta$ for parameter values $\theta$ sampled from a prior, and (2)   weeding out those pairs $(\wt X_\theta,\theta)$ for which $\wt X_\theta$ has  low fidelity to  observed data. 
 The discrepancy between observed and fake data is evaluated by first reducing the two datasets  to a vector of summary statistics and then measuring the distance between them.   
 Both the distance function and the summary statistics are critical for inferential success.
While eliciting suitable summary statistics often requires expert knowledge, automated approaches have emerged \citep{blum,gdkc2018,wasser}.
Notably, \cite{fearnhead2011constructing} proposed a 
semi-automated approach that approximates the posterior mean (a summary statistic that guarantees first-order accuracy) using a linear model regressing parameter samples onto simulated data.
Subsequently, \cite{jiang2017learning} elaborated on  this strategy  using deep neural networks which are expected to yield better approximations to the posterior mean. Beyond subtleties associated with summary statistics elicitation,
ABC has to be deployed with caution for Bayesian model choice \citep{robert2011lack,abc_model_choice}.
 Synthetic likelihood (SL) \cite{wood2010statistical,BSL} is another approach for carrying out inference in intractable models by constructing a proxy Gaussian likelihood for a vector of summary statistics. Implicit in the success of both ABC and SL is the assumption that the generating process can produce simulated summary statistics that adequately represent the observed ones. If this compatibility is not satisfied (e.g. in misspecified models), both SL \cite{robust_SL} and ABC \cite{ABC_misspec} can provide unreliable estimates.
Besides SL, a wide range parametric surrogate likelihood models have been suggested including normalising flows, Gaussian processes or neural networks \citep{gutmann2016bayesian,blum2,durkan,papa16}.
Avoiding the need for summary statistics,  \cite{gdkc2018} proposed to use discriminability  of the observed and simulated data as a discrepancy measure in ABC. Their accepting/rejecting mechanism separates samples based on a discriminator's ability to tell the real and fake data apart.
Similarly as their work, our paper is motivated by the observation that distinguishing two data sets is usually easier if they were simulated with very different parameter values. However, instead of deploying this strategy inside ABC, we embed it directly inside the Metropolis-Hastings algorithm using likelihood approximations obtained from classification.


The Metropolis-Hastings (MH) method generates ergodic Markov chains  through an accept-reject mechanism which depends in part on likelihood ratios comparing proposed candidate moves and current states.  For many latent variable models, the marginal likelihood is {\em not} available in closed form, making direct application of MH impossible (see  \citep{ddp2018} for examples). The pseudo-marginal likelihood method \citep{andrieu2009pseudo} offers a remedy by replacing  likelihood evaluations with their  (unbiased) estimates.  Many variants of this approach have been proposed including the inexact MCWM method (described in \cite{neil} and \cite{andrieu2009pseudo}) and its elaborations that  correct for  bias \citep{neil}, reduce the variance of the likelihood ratio estimator \citep{ddp2018} or  make sure that the resulting chain produces samples from the actual (not only approximate) posterior \citep{beaumont2003estimation}.  The idea of using likelihood approximations within MH  dates back to at least \citep{neil} and has been implemented in a series of works  (see e.g \citep{o1998microsatellite} and \citep{beaumont2003estimation}  and references therein). 
Our approach is fundamentally different from many typical pseudo-marginal MH algorithms since it {\em does not} require a hierarchical model where likelihood estimates are obtained through simulation from conditionals of latent data.
 Our method can be thus applied in a wide range of generative models (where forward simulation is possible) and other scenarios (such as diffusion processes \cite{heston}) where PM methods would be cumbersome or time-consuming to implement (as will be seen later in our examples).

{Inspired by  contrastive learning (CL)} \citep{htf09,gh_12} we reframe the likelihood (ratio) estimation problem as a classification problem using the `likelihood-ratio trick' \citep{durkan,cranmer16,gutman2,density_ratio}. Similarly as with generative adversarial networks (GANs) \cite{gan}, we pit two agents (a Generator and a Classifier) against one another. Assessing the similitude between the fake data, outputted by the Generator,   and observed data,  the Classifier provides likelihood estimators which can be deployed inside MH. The resulting algorithm provides samples from an approximate posterior. 

Our contributions are both methodological and theoretical. We develop a personification of  Metropolis-Hastings (MH) algorithm  for intractable likelihoods based on Classification, further referred to as MHC. 
We consider two variants: (1) a fixed generator design which may yield biased samples, and (2) a random generator design which may yield unbiased samples with increased variance. 
We then describe how and when the two can be combined in order to provide posterior samples with an asymptotically correct location and spread. 
 {Contrastive learning has been suggested in the context of posterior simulation before \citep{hermans, pham}.  Our approach differs in the choice of the contrasting density and, in addition, we  develop theory which was previously unavailable.} Our theoretical analysis consists of 
new  convergence rate results  for a posterior residual (an approximation error) associated with the Classifier. 
These rates are then shown to affect the rate of convergence of the stationary distribution, in a similar way as the ABC tolerance level affects the convergence rate of ABC posteriors \cite{abc_theory}. 
{Theoretical developments for related pseudo-marginal (PM) methods have been  concentrating on convergence properties of the Markov chain such as mixing rates} \citep{andrieu2009pseudo, ddp2018}.
  Here, we provide a rigorous asymptotic study of the stationary distribution including convergence rates (drawing upon connections to empirical Bayes and Bayesian misspecification), asymptotic normality results and, in addition, polynomial mixing time characterizations of the Markov chain.
  
 To illustrate that our MHC procedure can be deployed in situations when sampling from   conditionals of latent data ({often needed} for PM) is not practical or feasible, we consider two examples.
The first one entails discretizations of continuous-time processes which are popular in finance \citep{heston, cir}.  The second one is  a population-evolution generative model   where PM is  not straightforward and where   ABC methods need strong informative priors and high-quality summaries. In both examples, we demonstrate that MHC offers a reliable practical inferential alternative which is straightforward to implement.  We also show very good performance on a  Bayesian model choice example (where ABC falls short) and on the famous Ricker model (Section \ref{sec:ricker} in the Appendix) \cite{ricker1954stock} analyzed  by multiple authors \cite{gutmann2016bayesian,wood2010statistical,fearnhead2011constructing}.

The paper is structured as follows. Section \ref{sec:GAN} and \ref{sec:MH} introduce the classification-based likelihood ratio estimator and the MHC sampling algorithm. Section \ref{sec:properties} then describes the asymptotic properties of the stationary distribution. Section \ref{sec:simul}  shows demonstrations on simulated data and, finally, Section \ref{sec:discuss} wraps up with a discussion.

\smallskip

\vspace{-0.5cm}
\section{Likelihood Estimation with a Classifier}\label{sec:GAN}
Our framework consists of a series of i.i.d. observations $\{X_i\}_{i=1}^n\in\mathcal X$  realized from a probability measure $P_{\theta_0}$  indexed by a parameter  $\theta_0\in\Theta$ which is endowed with a prior $\Pi_n(\cdot)$.
We assume that $P_{\theta}$, for each $\theta\in\Theta$, admits a density $p_\theta$. 
Our objective is to draw observations from the posterior density given $\Xn=(X_1,\dots, X_n)'$ defined through
\begin{equation}\label{eq:posterior}
	\pi_n(\theta\mid \Xn)=\frac{p_\theta^{(n)}(\Xn) \pi(\theta)}{\int_\Theta p_\vartheta^{(n)} (\Xn)\ d\Pi(\vartheta)},
\end{equation}
where $p_\theta^{(n)}=\prod_{i=1}^np_\theta(X_i)$.
Our focus is on situations where the likelihood $p_\theta^{(n)}$ is  too costly to evaluate but can be readily sampled from. 

We develop a Bayesian computational platform at the interface between sampling and optimization {inspired by contrastive learning (CL) \citep{htf09,gh_12}} and Generative Adversarial Networks  (GAN) \cite{gan}.
The premise of GANs is to discover rich distributions over complex objects arising in artificial intelligence applications through simulation. The learning procedure consists of two 
entities pitted against one another.
A Generator  aims to deceive an Adversary by simulating samples that resemble the observed data while, at the same time, the Adversary learns to tell the fake and real data apart.
This process iterates until the generated data are indistinguishable by the Adversary.
While GANs have found their usefulness in simulating from distributions over images, here we forge new connections to Bayesian posterior simulation.

Similarly as with GANs, we assume a Generator transforming a set of latent variables $\wt X\in\wt{\mathcal{X}}$  to collect samples from $P_\theta$ through a known deterministic mapping $T_\theta:\wt{\mathcal{X}}\to\mathcal{X}$, i.e. $T_\theta(\wt{X})\sim P_\theta$ for $\wt{X}\sim\wt{P}$ for some distribution $\wt P$ on $\wt\mX$.
This implies that we can draw a single set of $m$ observations $\wt X^{(m)}$ and then filter them through $T_{\theta}$ to obtain a sample  $\wt X^{(m)}_\theta=T_\theta(\wt X^{(m)})$ from $P_\theta$ for any $\theta\in\Theta$.
Being able to easily draw samples from the model suggests the intriguing  possibility of learning density ratios    `by-comparison'  \citep{mohamed}.
Indeed, the fact that density ratios can be computed by building a classifier that compares two data sets \citep{durkan,cranmer16} has  lead to an emergence of a rich ecosystem of   algorithms for model-free inference \citep{gutman2, hermans,papa16, pham}. 
Many of these machine learning procedures are based on variants of the `likelihood ratio trick' (LRT) which builds a surrogate classification model for the likelihood ratio.
{Similarly as \cite{pham} and \cite{hermans}, we embed  the LRT within a classical Bayesian sampling algorithm  and furnish our procedure  with rigorous frequentist-Bayesian inferential theory.}


Our approach relies on the simple fact that a  cross\hyp{}entropy classifier  can be deployed to obtain an estimator of the likelihood ratio \cite{htf09,gh_12,gutman2}.
Recall that the classification problem with the empirical cross\hyp{}entropy loss  is defined through
\begin{equation}\label{eq:discriminate}
	\max_{D\in\mathcal{D}}\left[\,\frac{1}{n}\sum_{i=1}^n\log D(X_i)+\frac{1}{m}\sum_{i=1}^m\log(1-D(X_i^\theta))\right],
\end{equation}
where $\mathcal{D}$ is a set of measurable classification functions $D:\mathcal{X}\to(0,1)$ ($1$ for `real' and $0$ for `fake' data)  and where  $X_i^\theta=T_\theta(\wt X_i)$ for $\wt X_i\sim \wt P$ for $i=1,\dots, m$ are the `fake' data outputted by the Generator. If an oracle were to furnish the true model $p_{\theta_0}$, it is known that the population solution to \eqref{eq:discriminate} is the `Bayes classifier'  {(see Section 14.2.4 in \cite{htf09} and Proposition 1 in \citep{gan})}
\begin{equation}\label{eq:oracle_dis}
	D_\theta(X)\vcentcolon=\frac{p_{\theta_0}(X)}{p_{\theta_0}(X)+p_\theta(X)}\quad\text{for $X\in\mX$}.
\end{equation}
Reorganizing the terms in \eqref{eq:oracle_dis}, the likelihood can be written (see e.g. \cite{gutman2}) in terms of the  discriminator function  as
\begin{equation}\label{eq:likelihood}
{p_\theta^{(n)}(\Xn)}=p_{\theta_0}^{(n)}(\Xn)\exp\left(\sum_{i=1}^n\log\frac{1-D_\theta(X_i)}{D_\theta(X_i)}\right).
\end{equation}
The oracle discriminator $D_\theta(\cdot)$ depends on $p_{\theta_0}$ but can be estimated by simulation. Indeed, one can deploy the Generator to simulate the fake data $\wt X^{(m)}_\theta=T_\theta(\wt X^{(m)})$ and train a Classifier to distinguish them from $\Xn$. The Classifier outputs an estimator $\hat{D}_{n,m}^\theta$, for which we will see examples, and  which can be plugged into \eqref{eq:likelihood} to obtain 
the following likelihood estimator {$\wh p_\theta^{(n)}(\Xn)=p_{\theta_0}^{(n)}(\Xn)\exp\left(\sum_{i=1}^n\log\frac{1-\hat{D}_{n,m}^\theta(X_i)}{\hat{D}_{n,m}^\theta(X_i)}\right)$, i.e.
\begin{equation}\label{eq:estimated_likelihood}
\wh p_\theta^{(n)}(\Xn)=
p_\theta^{(n)}(\Xn)\e^{u_\theta(\Xn)},
\end{equation}
}where
\begin{equation}\label{eq:u}
	u_\theta(\Xn)\vcentcolon=\sum_{i=1}^n\biggl(\log\frac{1-\hat{D}_{n,m}^\theta}{1-D_\theta}-\log\frac{\hat{D}_{n,m}^\theta}{D_\theta}\biggr)
\end{equation}
will be further referred to as the log-posterior residual.
 In other words, \eqref{eq:estimated_likelihood}   is a deterministic functional of auxiliary random variables $\wt X^{(m)}$ and the observed data $\Xn$, and can be computed (up to a norming constant) from 
 $\hat D^\theta_{n,m}$. The posterior density  $\pi_n(\theta\mid \Xn)$ can be then estimated    by replacing $D_\theta$ with $\hat{D}_{n,m}^\theta$ in the likelihood expression  to obtain
\begin{equation}\label{eq:post_approx}
	\wh{\pi}_{n,m}(\theta\mid \Xn)\vcentcolon=\exp\biggl(\sum_{i=1}^n\log\frac{1-\hat{D}_{n,m}^\theta(X_i)}{\hat{D}_{n,m}^\theta(X_i)}\biggr)\pi(\theta)\propto\pi_n(\theta\mid \Xn)\e^{u_\theta(\Xn)}.
\end{equation}

Two observations ought to be made. First, the estimator \eqref{eq:post_approx} targets the posterior density only up to a norming constant. This will not be an issue in Bayesian algorithms involving  posterior density ratios (such as the Metropolis-Hastings algorithm considered here). Second, the estimator \eqref{eq:post_approx}  performs  exponential tilting of the original posterior, where the quality of the approximation crucially depends on the statistical properties of $u_\theta(\Xn)$. Note that $u_\theta(\Xn)$ depends also on the latent data $\wt X_\theta^{(m)}$. We devote the entire Section \ref{sec:residual} to statistical properties of $u_\theta(\Xn)$. The idea of  estimating likelihood ratios via discriminative classifiers has emerged in various contexts including hypothesis testing \citep{cranmer16} and posterior density estimation 
 \cite{gutman2}. {An important distinguishing feature of our approach is that we contrast observed and fake data, using the truth $\theta_0$ as a fixed reference point. This is different from the marginal approach in \cite{gutman2} which contrasts two fake datasets generated from the marginal and conditional likelihoods.  We highlight the connections in Section \ref{sec:variants} in the Supplement.}


\vspace{-0.5cm}
\section{Metropolis Hastings via Classification}\label{sec:MH}
The Metropolis-Hastings  (MH) algorithm is one of the mainstays of Bayesian computation.
The deployment of unbiased likelihood estimators within MH has shown great promise in models whose likelihoods are not available \citep{beaumont2002approximate,andrieu2009pseudo,av2015}. In the previous section, we have suggested how classification may be deployed to obtain estimates of likelihood ratios. This
suggests a compelling question: {\em  Can we deploy these classification-based estimators within MH?} This section explores this intriguing possibility and formalizes an MH variant that we further refer to as MHC, Metropolis Hastings via Classification.



Our objective is to simulate values from an (approximate) posterior distribution $\Pi_n(\cdot\C\Xn)$ with a density  $\pi_n(\theta\C\Xn)\propto p_\theta^{(n)}(\Xn)\pi(\theta)$ over $(\Theta,\mathscr{B})$ using the MH routine. Recall that MH simulates a Markov chain according to the transition kernel
$
K(\theta,\theta')\vcentcolon=\rho(\theta,\theta')q(\theta'\mid\theta)+\delta_\theta(\theta')\int_\Theta(1-\rho(\theta,\tilde{\theta}))q(\tilde\theta\mid\theta)\d\tilde\theta,
$
where
\begin{equation}\label{eq:accept_ratio}
\rho(\theta,\theta')\vcentcolon=\min\biggl\{\frac{p_{\theta'}^{(n)}(\Xn)\pi(\theta')}{p_{\theta}^{(n)}(\Xn)\pi(\theta)}\frac{q(\theta\mid\theta')}{q(\theta'\mid\theta)},1\biggr\}.
\end{equation}
and
where $q(\cdot\mid\theta)$  is a proposal  density generating candidate values $\theta'$ for the next move. 

It is often the case in practice that  we cannot directly evaluate $p_\theta^{(n)}(\Xn)$ but have access to its (unbiased) estimator (see \cite{doucet_etal14} for a recent overview).
In Bayesian contexts, an unbiased likelihood estimator can be constructed using importance sampling \citep{beaumont2003estimation} or particle filters \cite{andrieu_particle, andrieu2009pseudo}
via data augmentation through the introduction of auxiliary latent variables, say  $\wt X^{(m)}_\theta$. 
{The perhaps simplest variant of such strategies is the Monte Carlo Within Metropolis (MCWM) algorithm \cite{neil,andrieu2009pseudo}, which requires independently simulating}  $m$ replicates of the auxiliary data for each likelihood evaluation at each iteration. Other, {so called pseudo-marginal \citep{andrieu2009pseudo}}, variants have been suggested with latent data  recycled from the previous iterations (Grouped Independence MH (GIMH) described in \cite{beaumont2002approximate}) or with correlated latent variables  for the numerator and the denominator of the acceptance ratio \cite{ddp2018}.   
In this work, we propose replacing $p_\theta^{(n)}$ in the acceptance ratio \eqref{eq:accept_ratio} with the classification-based likelihood estimator \eqref{eq:estimated_likelihood} outlined in Section \ref{sec:GAN}. 
This estimator, similarly as with  pseudo-marginal (PM) methods, also relies on the introduction of latent variables $\wt X^{(m)}_\theta$. 
However, unlike with {related MH methods \citep{neil,andrieu2009pseudo}}, we {\em do not} require  an explicit hierarchical model where sampling from the conditional distribution of the latent data is feasible. 
Later in Section \ref{sec:lotka} we show an example of a generative model, where our approach fares {favorably}  while the PM-style approaches are not straightforward, if at all possible.
As we have seen earlier, our likelihood estimator can be rewritten in terms of the estimated discriminator as
\begin{equation}\label{eq:estimator_lik}
\wh p_\theta^{(n)}(\Xn)\propto \exp\left(\sum_{i=1}^n\log \frac{1-\hat D_{n,m}^\theta(X_i)}{\hat D_{n,m}^\theta(X_i)}\right).
\end{equation}
The evaluation of $\wh p_\theta^{(n)}(\Xn)$ can be carried out by merely computing  $\hat D^\theta_{n,m}(X_i)$   where $\hat D^{\theta}_{n,m}$ is a trained  classifier distinguishing $\Xn$ from $\wt X^{(m)}_\theta$. Putting the pieces together, one can replace the intractable likelihood ratio in the acceptance probability \eqref{eq:accept_ratio} with  
\begin{equation}\label{eq:ratio2}
	\rho_u(\theta,\theta')\vcentcolon=\min\left\{\frac{\wh p_{\theta'}^{(n)}(\Xn)\pi(\theta')}{\wh p_\theta^{(n)}(\Xn)\pi(\theta)}\frac{q(\theta\mid\theta')}{q(\theta'\mid\theta)},1\right\}.
\end{equation}
 Note that the proportionality constant in the likelihood expression \eqref{eq:estimator_lik}  cancels out in \eqref{eq:ratio2}, allowing $\rho_u(\theta,\theta')$ to be directly computable. We consider two variants.
The first one, called a {\em fixed generator design},  assumes  that the randomness of $\hat D^\theta_{n,m}$, for each given $\theta$ and $\Xn$, is determined by latent variables $\wt X^{(m)}$ {\em shared} by all steps of the algorithm. This corresponds to the case  when $m$ auxiliary data points $\wt X_\theta^{(m)}=\{\wt X_i^\theta\}_{i=1}^m$ are obtained through a {\em deterministic} mapping $\wt X_i^\theta=T_\theta(\wt X_i)$ for some  $\wt X_i\sim\wt P$ that are not changed throughout the algorithm. 
The second version, called a {\em random generator design}, 
assumes that  the underlying latent variables   variables $\wt X^{(m)}=\{\wt X_i\}_{i=1}^m$  are refreshed at each step.
While the difference between these two versions is somewhat subtle, we will see important bias-variance implications (discussed in more detail below).
While technically our MHC sampling procedure follows the footsteps of a standard MH algorithm, we still find it useful to summarize the computations in an algorithm box (see Table \ref{alg:MHC}).



\begin{table}[!t]
\centering
\vspace{-1.5cm}
	\spacingset{1.1}
	\small
	\scalebox{0.9}{
		\begin{tabular}{l l}
			\hline 	\hline
			\multicolumn{2}{c}{\bf INPUT \cellcolor[gray]{0.6} }\\	
				\hline
			\multicolumn{2}{c}{Draw $\wt X=\{\wt X_i\}_{i=1}^m\sim\wt P$}\\
			\multicolumn{2}{c}{Initialize $\theta^{(0)}$ and generate $\wt X_{\theta^{(0)}}=\{\wt X_i^{\theta^{(0)}}\}_{i=1}^m$ according to $\wt X^{\theta^{(0)}}_i=T_{\theta^{(0)}}(\wt X_i)$ .} \\
				\hline
			\multicolumn{2}{c}{\bf LOOP \cellcolor[gray]{0.6}}\\			
	\hline
			\multicolumn{2}{c}{For $t=1,\dots, T$ repeat steps C(1)-(3), R and U.} \\
			\multicolumn{2}{c}{\bf  \cellcolor[gray]{0.9}Algorithm 1: Fixed Generator}\\
			C(1): Given $\theta^{(t)}$, generate $\theta'\sim q(\cdot\mid\theta^{(t)})$.  &\\
			C(2): Generate $\wt X_{\theta'}=\{\wt X^{\theta'}_i\}_{i=1}^m$ according to $\wt X^{\theta'}_i=T_{\theta'}(\wt X_i)$. &\\
			C(3): Compute $\hat D_{n,m}^{\theta'}$ from $\Xn$ and $\wt X_{\theta'}$ and compute $\wh p_\theta(\Xn)$ in \eqref{eq:estimator_lik}. &\\
			C(4)  With $\rho_u(\cdot\,,\,\cdot)$   in \eqref{eq:ratio2},  set
	$
		\theta^{(t+1)}=\begin{cases} \theta' \quad &\text{with probability $\rho_u(\theta^{(t)},\theta')$}, \\ \theta^{(t)} \quad &\text{with probability $1-\rho_u(\theta^{(t)},\theta')$}.\end{cases}
	$
	&\\
			\multicolumn{2}{c}{\bf \cellcolor[gray]{0.9} Algorithm 2: Random Generator}\\
			C(1): Given $\theta^{(t)}$, generate $\theta'\sim q(\cdot\mid\theta^{(t)})$ and $\wt X'\sim \wt q(\wt X'\C \wt X^{(t)})$ &\\
			C(2): Generate $\wt X_{\theta'}=\{\wt X^{\theta'}_i\}_{i=1}^m$ according to $\wt X^{\theta'}_i=T_{\theta'}(\wt X_i')$ &\\
			C(3): Compute $\hat D_{n,m}^{\theta'}$ from $\Xn$ and $\wt X_{\theta'}$ and compute   $\wh p_\theta(\Xn)$ defined in \eqref{eq:estimator_lik}. &\\
			C(4): With  $\rho_u(\cdot\,,\,\cdot)$   in \eqref{eq:accept_ratio2},  set  
				$
		(\theta^{(t+1)},\wt X^{(t+1)})=\begin{cases} (\theta',\wt X') & \text{with probability $\wt\rho_u(\theta,\wt X;\theta',\wt X')$}, \\ 
		(\theta^{(t)},\wt X^{(t)}) & \text{with probability $1-\wt\rho_u(\theta,\wt X;\theta',\wt X')$}. \end{cases}
	$
	&\\
	\hline
						\multicolumn{2}{c}{\bf OUTPUT \cellcolor[gray]{0.6}}\\			
\hline

		\multicolumn{2}{c}{Samples $\theta^{(1)},\dots, \theta^{(T)}$ }\\
			\hline 	\hline
	\end{tabular}}
	\caption{\em  Metropolis-Hastings via Classification.}\label{alg:MHC}
\end{table}

\vspace{-0.5cm}
\subsection{Fixed Generator MHC}\label{sec:fg}

We  inquire whether and how the likelihood approximation affects the stationary distribution of the resulting Markov chain.
Due to the exponential tilt $\e^{u_\theta(\Xn)}$ in the likelihood approximation \eqref{eq:estimated_likelihood}, Algorithm 1 (Table \ref{alg:MHC}) does {\em not} yield the correct posterior $\pi_n(\theta\C\Xn)$ at its steady state.
Indeed, under standard assumptions  (see Section 7.3.1 of \cite{rc2004}), the  stationary distribution of the Markov chain, conditional on $\wt X^{(m)}$, writes as (see e.g.  Theorem 7.2 in \cite{rc2004})
\begin{equation}\label{eq:posterior_density}
\pi^\star_n(\theta\C X^{(n)})=\frac{p^{(n)}_\theta(X^{(n)})\times \e^{u_\theta(X^{(n)})}\times\pi(\theta)}{\int_\Theta p^{(n)}_\theta(\Xn)\times \e^{u_\theta(X^{(n)})}\times\pi(\theta)\d\theta}.
\end{equation}
We do not view this property as unsurmountable. {Other approximate MH algorithms
(e.g the MCWM method) may also not yield $\pi_n(\theta\C X^{(n)})$ as their stationary distribution, provided that it in fact  exists  \citep{andrieu2009pseudo}.}
However, the samples generated by Algorithm 1 will be  distributed according  an approximate posterior \eqref{eq:posterior_density} whose statistical properties we describe in detail in Section \ref{sec:properties}. 
In Section \ref{sec:speed}, we further quantify the speed of MHC convergence  in large samples under the assumption of asymptotic normality. 
 As will be seen in Section \ref{sec:properties}, the exponential tilt induces certain bias where the pseudo-posterior \eqref{eq:posterior_density} concentrates around a projection of the true parameter $\theta_0$. Despite the bias,  
 the {limiting} curvature of the approximate posterior can be shown to match the {limiting}  curvature of the actual posterior (under differentiability assumptions in Section \ref{sec:residual}).  
 The random generator version, introduced in the next section, works the other way around.
{Under some assumptions, it can lead to a correct location (no bias) but, possibly, at the expense of an enlarged variance.}
\vspace{-0.5cm}
\subsection{Random Generator MHC}\label{sec:rg}
The  random generator variant proceeds as Algorithm 1 but  refreshes $\wt X^{(m)}\sim\wt P$ at each step  before computing the acceptance ratio. {We denote the density associated with $\wt P$ by $\wt \pi$.} For simplicity, we have dropped the subscript $m$ in $\wt X^{(m)}$ while describing the algorithm in Table \ref{alg:MHC}.
The acceptance probability now also involves $\wt X$ and writes as
\begin{equation}\label{eq:accept_ratio2}
	\wt\rho_u(\theta,\wt X;\theta',\wt X')=\min\biggl\{\frac{\wh p_{\theta'}^{(n)}(\Xn)\pi(\theta')\wt\pi(\wt X')}{\wh p_\theta^{(n)}(\Xn)\pi(\theta)\wt\pi(\wt X)}\frac{q(\theta\mid\theta')}{q(\theta' \mid\theta)}\frac{\wt q(\wt X\mid\wt X')}{\wt q(\wt X'\mid\wt X)},1\biggr\}.
\end{equation}
To glean more insights into this variant, it is helpful to regard $(\theta^{(t)},\wt X^{(t)})$ jointly as a Markov chain with  an augmented proposal density 
$q(\theta',\wt X'\mid\theta,\wt X)=q(\theta'\mid\theta)\wt q(\wt X'\C \wt X)$ where $\wt q(\wt X'\C \wt X)$ possibly depends on $\wt X$. 
In order to make the dependence on $\wt X$ in  $u_\theta(\Xn)$ more transparent, we will denote the posterior residual defined in \eqref{eq:u} with $u_\theta(\Xn,\wt X)$ going forward. 
It can be seen that the marginal stationary distribution of the augmented Markov chain under Algorithm 2 equals
\begin{equation}\label{eq:stationary2}
	\wt\pi^\star_n(\theta\C\Xn)\vcentcolon=\int \pi^\star_n(\theta\C\Xn)\d \wt P(\wt X),
\end{equation}
where $\pi^\star_n(\theta\C\Xn)$ was defined earlier in \eqref{eq:posterior_density} and depends on $\wt X$ through $u_\theta(\Xn,\wt X)$.
The following characterization will be useful for establishing statistical properties of  $\wt\pi^\star_n(\theta\C\Xn)$ later in Section \ref{sec:properties}. From \eqref{eq:posterior_density}, we can write
\begin{equation}\label{eq:stationary_random}
\wt\pi^\star_n(\theta\C\Xn)\propto p_\theta^{(n)}(\Xn)\times \e^{\wt u_\theta(\Xn)}\times\pi(\theta)
\end{equation}
where
\begin{equation}\label{eq:u2}
\wt u_\theta(\Xn)=\log \int\e^{ u_{\theta}(\Xn,\wt X)}\d \wt P(\wt X).
\end{equation}
Assuming almost-sure positivity of the joint proposal density $q(\theta',\wt X'\mid\theta,\wt X)$, 
it can be verified (e.g. from Corollary 4.1 in \cite{tha2017}) that the marginal distribution of $\theta^{(t)}$ after $t$ steps of Algorithm 1  converges  in total variation to   $\wt\pi^\star_n(\theta\C\Xn)$ as {$t\rightarrow\infty$.} 
Interestingly, from \eqref{eq:stationary_random} we can see that the stationary distribution  \eqref{eq:stationary2}  from Algorithm 2 has {\em the same functional form} as the stationary distribution \eqref{eq:posterior_density} from  Algorithm 1. The only difference is replacing $u_\theta(\Xn)$ with an averaged-out version $\wt u_\theta(\Xn)$ in   \eqref{eq:u2}. Integration may inflate the stationary distribution \eqref{eq:stationary2} by making it more spread-out compared to the fixed generator sampler. However, the exponential tilting factor $\wt u_\theta(\Xn)$ is averaged out.
 While  $u_\theta(\Xn)$ in \eqref{eq:u} is {\em fixed} in $\wt X$ (creating a non-vanishing bias term), $u_\theta(\Xn)$ in \eqref{eq:u2} can average out to $0$ (depending on $\wt q(\cdot\C\cdot)$), erasing the bias and yielding the actual posterior as the stationary distribution.
\vspace{-0.5cm}
\subsection{Debiasing}\label{sec:alg_three}
Algorithm 1 and 2 can be combined to produce a more realistic representation of the true posterior.
We mentioned that Algorithm 1, under the differentiability assumptions, has the same {asymptotic  curvature (as $n\rightarrow\infty$)} as the actual posterior but has a non-vanishing shift.
Algorithm 2, on the other hand, has a reduced bias due to the averaging aspect in \eqref{eq:u2}. We can thus diminish the bias of the fixed generator design by shifting the location towards the 
 mean of samples obtained with the random generator. This leads to a hybrid procedure summarized in Table \ref{alg:three}.
While Algorithms 1 and 2  can be deployed as a standalone, the de-biasing variant might increase the quality of the samples.
Note that if $\wt u_\theta(\Xn)=0$, Algorithm 2 will be {exact}, yielding the actual posterior as its stationary distribution. {If inexact, in Section \ref{sec:residual} we provide sufficient conditions under which Algorithm 3 yields samples from an object which, at least, has the same limit as the actual posterior.}
\begin{table}[!t]
\vspace{-1.5cm}
\centering
	\small
	\scalebox{0.9}{
		\begin{tabular}{| l l|}
\hline\hline
			 \multicolumn{2}{c}{\sl   \cellcolor[gray]{0.8} Algorithm 3: Bias Correction}\\
			\hline
			 (1) &{Generate a sample $\{\theta_1^{(t)}\}_{t=1}^T$ using Algorithm 1}\\
			 (2) &{Generate a sample $\{\theta_2^{(t)}\}_{t=1}^T$ using Algorithm 2} \\
			 (3) &{Debias $\{\theta_1^{(t)}\}$ using $\{\theta_2^{(t)}\}$,  i.e.  construct a sample $\{\theta^{(t)}\}$ by}\\
			       &$\theta^{(t)}\vcentcolon=\theta_1^{(t)}-\frac{1}{T}\sum_{s=1}^T\theta_1^{(s)}+\frac{1}{T}\sum_{s=1}^T\theta_2^{(s)}.$ \\
			 \hline\hline
	\end{tabular}}
	\caption{\em   Bias Correction with Algorithm 3.}\label{alg:three}
\end{table}
\vspace{-0.5cm}

\vspace{-0.2cm}
\section{Theory for MHC}\label{sec:properties}
We now shift attention from the computational aspects of MHC to its potential as a statistical inference procedure.
To understand the qualitative properties of the MHC scheme, we provide an asymptotic study of its stationary distribution (convergence rates in Section \ref{sec:conc_rate} and asymptotic normality in Section \ref{sec:bvm}), drawing upon its connections to empirical Bayes methods (Section \ref{sec:EB}) and Bayesian misspecification (Section \ref{sec:misspec}). Before delving into the stationary distribution, however, we first derive rates of convergence for the posterior residual $u_\theta(\Xn)$  in \eqref{eq:u} which plays a fundamental role. For additional theory showing fast mixing of our Markov chains (i.e. polynomial mixing times) see Section \ref{sec:speed} in the Appendix.

\vspace{-0.5cm}
\subsection{Convergence of the Posterior Residual}\label{sec:residual}
We denote the sample objective function in \eqref{eq:discriminate} with $\mathbb{M}_{n,m}^\theta(D)\vcentcolon=\mathbb{P}_n\log D+\mathbb{P}_m^\theta\log(1-D)$, where 
we employed the operator notation for expectation, e.g., $\mathbb{P}_n f=\frac{1}{m}\sum_{i=1}^m f(X_i)$ and $\mathbb{P}_m^\theta f=\frac{1}{m}\sum_{i=1}^m f(X_i^\theta)$ (see the notation Section \ref{sec:notation} in the Appendix for further details).
Throughout this section, we will use a simplified notation $u_\theta$ instead of $u_\theta(\Xn)$ and similarly for $p_\theta$ and $p_\theta^{(n)}$. We denote by $P$ the probability measure that encompasses all randomness, e.g., as $O_P(1)$.%
\footnote{We may think of this $P$ as the ``canonical representation'' \citep[Problem 1.3.4]{vw1996}.}
The estimated Classifier  is seen to satisfy
$$
\hat D_{n,m}^{\theta}:=\max_{D\in\mathcal D_n}\mathbb{M}_{n,m}^\theta(D)
$$
where $\mathcal{D}_n$ constitutes a sieve  of classifiers that expands with the sample size and that is not too rich (as measured by the bracketing entropy  $N_{[]}(\varepsilon,\mathcal{F},d)$).
In practice, the estimator $\hat{D}_{n,m}^\theta$ can be obtained by deploying  a variety of classifiers ranging from logistic regression to deep learning (see Assumption 3 in \cite{kmp2020} for a sieve construction using neural network classifiers). The discrepancy between two  classifiers will be measured by a Hellinger\hyp{}type distance (see \cite{kmp2020} and \cite{patilea} for more discussion)
$
	d_\theta(D_1,D_2)\vcentcolon=\sqrt{h_\theta(D_1,D_2)^2+h_\theta(1-D,1-D_\theta)^2},
$
where $h_\theta(D_1,D_2)=\sqrt{(P_{\theta_0}+P_\theta)(\sqrt{D_1}-\sqrt{D_2})^2}$.
The rate of convergence of the Classifier was previously established by \cite{kmp2020} under assumptions reviewed below.
In the following, we denote with $\mathcal{D}_{n,\delta}^\theta\vcentcolon=\{D\in\mathcal{D}_n:d_\theta(D,D_\theta)\leq\delta\}$ the neighborhood of the oracle classifier within the sieve. 

\begin{asm} \label{asm:1}
Assume that $n/m$ converges and that an estimator $\hat{D}_{n,m}^\theta$ exists that satisfies $\mathbb{M}_{n,m}^\theta(\hat{D}_{n,m}^\theta)\geq\mathbb{M}_{n,m}^\theta(D_\theta)-O_P(\delta_n^2)$ for a nonnegative sequence $\delta_n$. Moreover, assume that the bracketing entropy integral\footnote{See the notation Section \ref{sec:notation} in the Appendix.} satisfies
$J_{[]}(\delta_n,\mathcal{D}_{n,\delta_n}^\theta,d_\theta)\lesssim\delta_n^2\sqrt{n}$ and that there exists $\alpha<2$ such that $J_{[]}(\delta,\mathcal{D}_{n,\delta}^\theta,d_\theta)/\delta^\alpha$ \textcolor{blue}{has a majorant} decreasing in $\delta$.
\end{asm}

{
The assumption requires that the synthetic sample size $m$ is at least as large as the actual sample size $n$, including the case when $n/m$ converges to $0$.
The second assumption requires that the training algorithm for the discriminator can find a sufficiently good approximate maximizer.
The third assumption requires that the entropy of the sieve is not too large in order to avoid overfitting. For example, the bracketing entropy of a neural network sieve was shown to be bounded  \citep[Lemma 2]{kmp2020}.
}

Under Assumption \ref{asm:1},  for a given $\theta\in\Theta$, \cite{kmp2020}  conclude (see their Theorem 1) the following convergence rate result for the classifier: 
$d_\theta(\hat{D}_{n,m}^\theta,D_\theta)=O_P(\delta_n)$.
While  \cite{kmp2020} focused mainly on the convergence of $\hat{D}_{n,m}^\theta$, here we move the investigation further by establishing the rate of convergence of  $u_\theta(\cdot)/n$ as well as its  limiting shape. To this end, we assume the following support compatibility assumption, a refinement of the bounded likelihood ratio condition in nonparametric maximum likelihood (Theorem 3.4.4 in \cite{vw1996} and Lemma 8.7 in \cite{ggv2000}).

\begin{asm} \label{asm:2}
There exists $M>0$ such that for every $\theta\in\Theta$, $P_{\theta_0}(p_{\theta_0}/p_\theta)$ and $P_{\theta_0}(p_{\theta_0}/p_\theta)^2$ are bounded by $M$ and
\[
	\sup_{D\in\mathcal{D}_{n,\delta_n}^\theta}P_{\theta_0}\biggl(\frac{D_\theta}{D}\biggm|\frac{D_\theta}{D}\geq\frac{25}{16}\biggr)<M,\
	\sup_{D\in\mathcal{D}_{n,\delta_n}^\theta}P_{\theta_0}\biggl(\frac{1-D_\theta}{1-D}\biggm|\frac{1-D_\theta}{1-D}\geq\frac{25}{16}\biggr)<M
\]
for $\delta_n$ in \cref{asm:1}.
The brackets in \cref{asm:1} can be taken so that $P_{\theta_0}\Bigl(\sqrt{\frac{u}{\ell}}-1\Bigr)^2=O(d_\theta(u,\ell)^2)$ and $P_{\theta_0}\Bigl(\sqrt{\frac{1-\ell}{1-u}}-1\Bigr)^2=o(d_\theta(u,\ell))$.
\end{asm}

{
For the cross\hyp{}entropy loss, it is essential to control the tail behavior of the discriminator.
Assumption \ref{asm:2} restricts the tail of the discriminator so that the residual of the cross\hyp{}entropy can be bounded with the bracketing entropy.
For example, for a logistic discriminator, the tail of $D$ is proportional to an exponential function $e^{-x'\beta}$ for some $\beta$. Therefore, if $P_0$ has an exponential tail and $\mathcal{D}_{n,\delta_n}^\theta$ gives a compact support for $\beta$, \cref{asm:2} is satisfied.
By analogy, we see that \cref{asm:2} is reasonable for neural network discriminators that use sigmoid activation functions.
}

The following Theorem will be crucial for understanding theoretical properties of our MHC sampling algorithm, where the rate of convergence of $u_\theta(\cdot)/n$ will be seen to  affect the rate of convergence of the stationary distribution of our Markov chains.

\begin{thm}\label{thm:rate:res}
Let  \cref{asm:1,asm:2} hold for a given $\theta\in\Theta$, then
\[
u_\theta/n=	\mathbb{P}_n\biggl(\log\frac{1-\hat{D}_{n,m}^\theta}{1-D_\theta}-\log\frac{\hat{D}_{n,m}^\theta}{D_\theta}\biggr)=O_P(\delta_n).
\]
\end{thm}
\proof Section \ref{sec:proof_thm:rate:res} in the Appendix.

One seemingly pessimistic conclusion from \cref{thm:rate:res} is that  $u_\theta(\cdot)$ {\em does not} vanish. \cite{kmp2020} shows that if the true likelihood ratio has a low\hyp{}dimensional representation and an appropriate neural network is used for the discriminator, the rate $\delta_n$ depends only on the underlying dimension and not on the original dimension of $X_i$.
In spite of the non-vanishing tilting term $u_\theta(\Xn)$, it turns out that Algorithm 1 can be refined (de-biased) to produce reasonable samples as long as $\hat{D}_{n,m}^\theta$ estimates the score well (see Section \ref{sec:alg_three}).
In the sequel, we show  quadratic approximability for $u_\theta$ at a much faster rate than \cref{thm:rate:res} when the model and the classifier are differentiable in some suitable sense.
\begin{asm}[Differentiability of $p_\theta$] \label{asm:dqm}
There exists $\theta_0\in\Theta\subset\mathbb{R}^d$ such that $P_0=P_{\theta_0}$.
The model $\{p_\theta\}$ is {\em differentiable in quadratic mean at $\theta_0$}, that is, there exists a measurable function $\dot{\ell}_{\theta_0}:\mathcal{X}\to\mathbb{R}^d$ such that%
\footnote{Integration is understood with respect to some dominating measure.}
$
	\int\biggl[\sqrt{p_{\theta_0+h}}-\sqrt{p_{\theta_0}}-\frac{1}{2}h'\dot{\ell}_{\theta_0}\sqrt{p_{\theta_0}}\biggr]^2=o(\|h\|^2).
$
\end{asm}
This is a classical assumption (see e.g. Section 5.5 of \cite{v1998})  which implies local asymptotic normality.
Going back to \eqref{eq:estimated_likelihood}, we write $\wh p_\theta(\Xn)=\prod_{i=1}^n\hat p_\theta(X_i)$, where
\begin{equation}\label{eq:phat}
\hat{p}_\theta=p_{\theta_0}\frac{1-\hat D_{n,m}^\theta}{\hat D_{n,m}^\theta}
\end{equation}
is an estimator of $p_\theta$ that is possibly unscaled so that $\int\hat{p}_\theta$ may not be one.
The scaling constant will be denoted by $c_\theta\vcentcolon=\int\hat{p}_\theta$.
In general, $\hat{p}_\theta$ is not observable since $p_{\theta_0}$ is not available.
From \eqref{eq:u}, we can see that  
$
	u_\theta=n\mathbb{P}_n\log\frac{1-\hat{D}_{n,m}^\theta}{\hat{D}_{n,m}^\theta}-n\mathbb{P}_n\log\frac{1-D_\theta}{D_\theta}=n\mathbb{P}_n\log\frac{\hat{p}_\theta}{p_{\theta_0}}-n\mathbb{P}_n\log\frac{p_\theta}{p_{\theta_0}}
$
and, under \cref{asm:dqm}, \citet[Theorem 7.2]{v1998} derives convergence of the second term above in the local neighborhood of $\theta_0$.
In \cref{thm:linear} below, we derive convergence of the first term under the a similar assumption.
\begin{asm}[Differentiability of $\hat{p}_\theta$] \label{asm:dqmp} \hfill
\vspace{-0.3cm}
\begin{enumerate}[(i)]
	\item \label{asm:dqmp1} {The estimator $\hat{p}_\theta$ is {\em differentiable in quadratic mean in probability at $\theta_0$ with a cubic rate}, which we define as $\hat{P}_{\theta_0}\dot{\ell}_{\theta_0}\dot{\ell}_{\theta_0}'\to^p I_{\theta_0}$ and
		\[
			(P_{\theta_0}+\hat{P}_{\theta_0})\biggl(\sqrt{\frac{\hat{p}_{\theta_0+h}}{\hat{p}_{\theta_0}}}-1-\frac{1}{2}h'\dot{\ell}_{\theta_0}\biggr)^2=O_P(\|h\|^3),
		\]
		where $\dot{\ell}_{\theta_0}:\mathcal{X}\to\mathbb{R}^d$ is the score function in \cref{asm:dqm}.}
	\item \label{asm:proj} Dependence of $\mathbb{P}_n$ and $\hat{p}_\theta$ is asymptotically ignorable in the sense that for every compact $K\subset\mathbb{R}^d$, in outer probability,
		\begin{gather*}
			\sup_{h\in K}\,\biggl|n(\mathbb{P}_n-P_{\theta_0})\biggl(\sqrt{\frac{\hat{p}_{\theta_0+h/\sqrt{n}}}{\hat{p}_{\theta_0}}}-1-\frac{h'\dot{\ell}_{\theta_0}}{2\sqrt{n}}\biggr)\biggr|\conv 0,\\
			\sup_{h\in K}\,\biggl|n(\mathbb{P}_n-P_{\theta_0})\biggl(\sqrt{\frac{\hat{p}_{\theta_0+h/\sqrt{n}}}{\hat{p}_{\theta_0}}}-1\biggr)^2\biggr|\conv 0.
		\end{gather*}
		
	\item \label{asm:const} The scaling factor is asymptotically linear in the sense that there exists a sequence of $\mathbb{R}^d$\hyp{}valued random variables $\dot{c}_{n,\theta_0}$ such that for every compact $K\subset\mathbb{R}^d$,
	in outer probability, 
		$
			\sup_{h\in K}\,\bigl|n\bigl(c_{\theta_0+h/\sqrt{n}}-c_{\theta_0}\bigr)-\sqrt{n}h'\dot{c}_{n,\theta_0}\bigr|\conv 0.
		$
		
\end{enumerate}
\end{asm}
\cref{asm:dqmp} (\ref{asm:dqmp1}) requires that $\hat{p}_\theta$ estimates the score well and is smoother than once differentiable. If $\hat{p}_\theta$ is twice differentiable in $\theta$, then it holds with $O_P(\|h\|^4)$.
\cref{asm:dqmp} (\ref{asm:proj}) requires that the dependence of $\mathbb{P}_n$ and $\hat{p}_\theta$ be ignored asymptotically. If $\mathbb{P}_n$ and $\hat{p}_\theta$ were independent, it would follow from Chebyshev's or Markov's inequality.
\cref{asm:dqmp} (\ref{asm:const}) requires that the quadratic curvature of the scaling constant vanishes asymptotically.
In general, \cref{asm:dqmp} is not verifiable since the likelihood is not available. To  develop intuition behind this assumption, we verify that it holds for a toy  normal location\hyp{}scale model example in Section \ref{sec:example_sup} in the Appendix.
With \cref{asm:dqmp}, the estimated log likelihood asymptotes to a quadratic function that has the oracle curvature but a different center.
\begin{thm} \label{thm:linear}
Let $p_\theta$ and $\hat{p}_\theta$ satisfy \cref{asm:dqm,asm:dqmp} and $\int(\sqrt{\hat{p}_{\theta_0}}-\sqrt{p_{\theta_0}})^2=O_P(\delta_n^2)$ for some $\delta_n=o(n^{-1/4})$.
Then, for every compact $K\subset\mathbb{R}^d$, in outer probability,
\begin{multline*}
	\sup_{h\in K}\,\biggl|n\mathbb{P}_n\log\frac{\hat{p}_{\theta_0+h/\sqrt{n}}}{\hat{p}_{\theta_0}}
	+\frac{1}{2}h'I_{\theta_0}h-\sqrt{n}\mathbb{P}_n h'\dot{\ell}_{\theta_0}
	+\sqrt{n}\hat{P}_{\theta_0}h'\dot{\ell}_{\theta_0}
	-\sqrt{n}h'\dot{c}_{n,\theta_0}\biggr|\conv 0
\end{multline*}
\end{thm}
\proof Section \ref{sec:proof_thm:linear} in the Appendix. 
\begin{rem}
Recall that the true log-likelihood ratio locally approaches  a quadratic curve $-\frac{1}{2}h'I_{\theta_0}h+\sqrt{n}\mathbb{P}_n h'\dot{\ell}_{\theta_0}$. The linear term $h'\sqrt{n}(\dot{c}_{n,\theta_0}-\hat{P}_{\theta_0}\dot{\ell}_{\theta_0})$  in \eqref{eq:linear_term} shifts the center of the quadratic curve but not the curvature. 
\end{rem}
One important implication of \cref{thm:linear} is linearity of $u_\theta$. 
\begin{coro}(Linear $u_\theta$)\label{rem:linear1}
Under assumptions of \cref{thm:linear} we have
\begin{equation}\label{eq:linear_term}
	u_{\theta_0+h/\sqrt{n}}-u_{\theta_0}
	=h'\sqrt{n}(\dot{c}_{n,\theta_0}-\hat{P}_{\theta_0}\dot{\ell}_{\theta_0})+o_P(1).
\end{equation}
\end{coro}
\proof  Follows from \Citet[Theorem 7.2]{v1998} and Theorem \ref{thm:linear}.

We revisit linearity of $u_\theta$ later in Section \ref{sec:misspec} (Example \ref{ex:linear}) as one of the sufficient conditions for the Bernstein-von Mises theorem.
Corollary \ref{rem:linear1} has a very important consequence regarding the limiting shape of the stationary distribution $\pi_n^\star(\theta\C\Xn)$  for  Algorithm 1 defined in \eqref{eq:posterior_density}.  It shows that $\pi_n^\star(\theta\C\Xn)$ approaches a {\em biased} normal distribution with the {\em same variance} as the true posterior. In addition, we have seen in Section \ref{sec:rg} that the stationary distribution $\wt\pi^\star_n(\theta\C\Xn)$ of Algorithm 1 defined in \eqref{eq:stationary_random} is averaged over the bias. Therefore, if  $\mathbb{E}[\dot{c}_{n,\theta_0}-\hat{P}_{\theta_0}\dot{\ell}_{\theta_0}\mid \Xn]=0$, where {the expectation is taken over the latent data $ \Xm$}, then the stationary distribution of Algorithm 3 (in Table \ref{alg:three}) converges to the correct normal posterior, i.e. it has the same limit as the actual posterior $\pi_n(\theta\C\Xn)$. Theorem \ref{thm:linear} thus provides a theoretical justification for de-biasing suggested in Section \ref{sec:alg_three}.

\vspace{-0.5cm}
\subsection{Posterior Concentration Rates}\label{sec:conc_rate}
Having quantified the convergence rate of the posterior residual $u_\theta(\Xn)$ in Theorem \ref{thm:rate:res}, we are now ready to explore the convergence rate of the entire stationary distribution without necessarily imposing differentiability assumptions.
\vspace{-0.5cm}
\subsubsection{Empirical Bayes Lens}\label{sec:EB}
Recall that  the MHC sampler {\em does not} reach $\pi_n(\theta\C X^{(n)})$ in steady state. Recall  that  the stationary distribution (using the fixed generator) takes the form
\begin{equation}\label{eq:post}
\Pi_n^\star(B\C X^{(n)})=\frac{\int_B  {p_{\theta}^{(n)}}/{p_{\theta_0}^{(n)}}\times\e^{u_\theta}\times \pi(\theta)\d\theta}{\int_\Theta  {p_{\theta}^{(n)}}/{p_{\theta_0}^{(n)}}\times\e^{u_\theta}\times\pi(\theta)\d\theta}.
\end{equation}
In the random design, we simply replace  $u_\theta$ in \eqref{eq:post} with $\wt u_\theta$ defined in \eqref{eq:u2}.
Interestingly, \eqref{eq:post} can be viewed  as an actual posterior under a {\em tilted prior} with a density $\pi^*(\theta)\propto\e^{u_\theta}\pi(\theta)$. This shifted prior depends on the data $\Xn$ (through $u_\theta(\Xn)$) and thereby \eqref{eq:post} can be loosely regarded as an empirical Bayes (EB) posterior. While EB uses plug-in estimators of  prior hyper-parameters, here the data enters the prior in a less straightforward manner. 

We first assess the quality of the posterior approximation \eqref{eq:post}  through its concentration  rate  around the {\em true parameter} value $\theta_0$ using the
traditional Hellinger semi-metric $d_n(\theta,\theta')$.
The rate depends on the interplay between the concentration of the {\em actual} posterior\footnote{Using the usual notion \cite{gvdv07}, we say that the posterior $\Pi_n(\cdot\C\Xn)$ concentrates around $\theta_0$ at the rate $\varepsilon_n$ (satisfying $\varepsilon_n\rightarrow0$ and $n\varepsilon_n^2\rightarrow\infty$ ) if  $P_{\theta_0}\Pi_n[\theta\in\Theta: d_n(\theta,\theta_0)>M\varepsilon_n\C\Xn]\rightarrow0$ as $n\rightarrow\infty$ where $M$ possibly depends on $n$.}  $\Pi_n(\theta\C\Xn)$ and the rate at which the residual $u_\theta(\Xn)$ in \eqref{eq:u} diverges. Recall that the rate of $u_\theta(\cdot)/n$ was established earlier in Theorem \ref{thm:rate:res}. The following Theorem uses assumptions on prior concentration around $\theta_0$
 using the typical Kullback-Leibler neighborhood
$
B_n(\theta_0,\epsilon)=\left\{\theta\in\Theta: K(p_{\theta_0}^{(n)},p_{\theta}^{(n)})\leq n\epsilon^2,
\frac{1}{n}\sum_{i=1}^nV_{2}(p_{\theta_0}(X_i),p_{\theta}(X_i))\leq \epsilon^2\right\}.
$

\begin{thm}\label{thm:one}
Consider the pseudo-posterior distribution $\Pi_n^\star$ defined through \eqref{eq:post}. 
Suppose that the prior $\Pi_n(\cdot)$ satisfies conditions (3.2) and (3.4) in \cite{gvdv07}  for a sequence $\varepsilon_n\rightarrow0$ such that $n\varepsilon_n^2\rightarrow\infty$. 
In addition,  let $\wt C_n$ be such that 
\begin{equation}\label{ass:u}
P_{\theta_0}^{(n)}\left(\sup_{\theta\in\Theta}|u_\theta(X^{(n)})/n|>\wt C_n \varepsilon_n^2\right)=o(1)
\end{equation}
and assume that for sets $\Theta_n\subset\Theta$ the prior satisfies
\begin{equation}\label{ass:prior}
\frac{\Pi_n(\Theta\backslash\Theta_n)}{\Pi_n(B_n(\theta_0,\varepsilon_n))}=o(\e^{-2(1+\wt C_n)n\varepsilon_n^2}).
\end{equation}
Then we have, for any $M_n\rightarrow\infty$ such that $\wt C_n=o(M_n)$,
$$
P_{\theta_0}^{(n)}\left[\Pi^\star_n(\theta: d_n(\theta,\theta_0)>M_n\varepsilon_n\C X^{(n)})\right]=o(1)\quad\text{as $n\rightarrow\infty$}.
$$
\end{thm}
\proof 
The proof is a minor modification of Theorem 4 in \cite{gvdv07} and is postponed until Section \ref{sec:proof_thm:one} in the Appendix.

Theorem \ref{thm:one} shows that the concentration rate of the pseudo-posterior nearly matches the concentration rate of the original posterior $\varepsilon_n$ (this is implied by condition (3.2), (3.4) and a variant of \eqref{ass:prior} according to Theorem 4 of \cite{gvdv07}) up to an inflation factor $\wt C_n$ which depends on the rate of $u_\theta(X^{(n)})/n$.  If  $\wt C_n=\mathcal O(1)$  in  \eqref{ass:u},
the rate of the actual posterior and pseudo-posterior will be the same.

\begin{rem}(Random Generator)
Recall that the  stationary distribution $\wt\pi^\star_n(\theta\C\Xn)$ of the random generator MHC version can be written as \eqref{eq:post} where $u_\theta$ is replaced with $\wt u_\theta$ from \eqref{eq:u2}. Theorem \ref{thm:one} holds also for the random generator where $\wt C_n$ is obtained from \eqref{ass:u} with $\wt u_\theta$ instead of $u_\theta$. Due to the averaging aspect, we might expect this $\wt C_n$ to be smaller in the random generator design.
\end{rem}

{
\begin{rem}\label{rem:marginal}(Marginal Reference Distribution)
Theorem  \ref{thm:one} holds {\em also} for the marginal contrastive learning Metropolis-Hastings approach proposed in \cite{hermans}. Indeed,  defining $u_\theta$   in terms of the discriminator $D_\theta^m(X)=p(X)\pi(\theta)/[p(X)\pi(\theta)+p_\theta(X)\pi(\theta)]$, the same conclusion holds for the marginal approach under the assumption in \eqref{ass:u}. 
\end{rem}
}
Theorem \ref{thm:one} describes the behavior of the pseudo-posterior around the truth $\theta_0$. We learned that the rate is artificially inflated due a bias inflicted by the likelihood approximation, where $\Pi^\star_n(\cdot\C\Xn)$ {\em may not} shrink around $\theta_0$ when $\varepsilon_n$ is faster than the rate $\delta_n$ established in Theorem \ref{thm:rate:res}.  This suggest that the truth may not be the most natural centering point for the posterior to concentrate around.
A perhaps more transparent approach is to consider a different (data-dependent) centering which will allow for a more honest reflection of  the contraction speed  devoid of any implicit bias. We look into model misspecification for guidance about reasonable centering points.
\vspace{-0.3cm}
\subsubsection{Model Misspecification Lens}\label{sec:misspec}
In Section \ref{sec:EB}, we reframed  the stationary distribution \eqref{eq:posterior_density} as an empirical Bayes posterior by absorbing the term $\e^{u_\theta(\Xn)}$ inside the prior. 
This section pursues a different approach, absorbing $\e^{u_\theta(\Xn)}$ inside the likelihood instead. This leads a mis-specified model $\wt P_\theta^{(n)}$  prescribed by the following likelihood function
\begin{equation}\label{eq:tilde_lik}
\wt p^{(n)}_\theta(\Xn)=\frac{p^{(n)}_\theta(\Xn)\e^{u_\theta(\Xn)}}{C_\theta}\quad\text{where}\quad C_\theta=\int_{\mX} p^{(n)}_\theta(\Xn)\e^{u_\theta(\Xn)}\d\Xn.
\end{equation}
Defining  $\wt\pi(\theta)\propto \pi(\theta)C_\theta$, we can rewrite \eqref{eq:posterior_density} as a posterior density under a mis-specified likelihood  and the modified prior $\wt\pi(\theta)$ as
\begin{equation}\label{eq:pstar}
\pi^\star_n(\theta\C X^{(n)})= \frac{\wt p_\theta^{(n)}(X^{(n)})\wt \pi(\theta)}{\int_\Theta \wt p_\theta^{(n)}(\Xn)\wt \pi(\theta)\d\theta}.
\end{equation}
Since the model $\wt p^{(n)}_\theta$ is mis-specified (i.e. $P^{(n)}_{\theta_0}$ is {\em not} of the same form as $\wt\mP^{(n)}=\{\wt P^{(n)}_\theta: \theta\in\Theta\}$ {due to the fact that $\hat D_{n,m}$ departs from the oracle discriminator}), the posterior will concentrate around the point $\theta^*$ defined as
\begin{equation}\label{eq:KL_theta}
\theta^*=\arg\min\limits_{\theta\in\Theta} -P_{\theta_0}^{(n)}\log [\wt p_\theta^{(n)}/ p_{\theta_0}^{(n)}]
\end{equation}
which corresponds to the element $\wt P_{\theta^*}^{(n)}\in\wt\mP^{(n)}$ that is closest to $P^{(n)}_{\theta_0}$ in the KL sense \cite{kleijn1}. Unlike  in the iid data case studied, e.g., in
\cite{kleijn1} and \cite{deblasi}, our likelihood \eqref{eq:tilde_lik} is not an independent product due to the non-separability of the function $u_\theta(\Xn)$. Theorem \ref{thm:misspec} (Section \ref{sec:proof_thm_misspec} in the Appendix) quantifies concentration  in terms of a  KL neighborhoods around $\wt P_{\theta^*}^{(n)}$. {Beyond the speed of posterior concentration, we provide sufficient conditions for the stationary distribution to converge to a Gaussian distribution (see Section \ref{sec:bvm} in the Supplement).}

\vspace{-0.5cm}
\section{MHC in Action}\label{sec:simul}
To whet reader's appetite, we present  MHC performance demonstrations  in  two examples which we found  challenging for pseudo-marginal (PM) approaches and ABC.
 The first one (the CIR model) exemplifies data arising as discretizations of continuous-time process  for which likelihood inference  can be problematic \citep{sorensen}. We show that, compared with MCWM, MHC is not only  far more straightforward to implement but also  more scalable.  The second demonstration  involves a generative model (Lotka-Volterra) for which no explicit hierarchical model exists, precluding from straightforward application of    MH methods \citep{beaumont2003estimation}. We thus compare MHC with ABC, showing that  ABC techniques may fall short without a very informative prior and suitable summary statistics.  More examples are shown in the Appendix where we show bias-variance tradeoffs between fixed/random generators  on a toy normal location-scale model (Section \ref{sec:example_sup}) and  the Ricker model \citep{ricker1954stock} (Section \ref{sec:ricker}). We also present  a Bayesian model selection example (Section \ref{sec:bf} in the Appendix) where ABC faces challenges.

\vspace{-0.5cm}
\subsection{The CIR Model}\label{sec:cir}
The CIR model \citep{cir} is prescribed by the stochastic differential equation 
$$
dX_t=\beta(\alpha-X_t)dt+\sigma\sqrt{X_t}dW_t
$$
where $W_t$ is the Brownian motion, $\alpha>0$ is a mean-reverting level, $\beta>0$ is the speed of the process and $\sigma>0$ is the volatility parameter. This process is an integral component of the Heston model \citep{heston} where it is deployed for modelling  instantaneous variances.  We want to perform Bayesian inference for the parameters $\theta=(\alpha,\beta,\sigma)'$ of this continuous-time Markov process which is observed at discrete time points $t_j=j\Delta$ for $j=1,\dots, T$. We will assume that there are $n$ independent observed realizations $\bm x_i=(x_{i1},\dots, x_{iT})'$ of this discretized series for $1\leq i\leq n$. It has been acknowledged that if the data are recoded at discrete times, parametric inference  using the likelihood can be difficult, partially due to the fact that the likelihood function is often not available \citep{sorensen}.  One possible Bayesian inferential platform for such problems is the MH algorithm where the likelihood function can be replaced with its approximation (e.g. using the analytical closed-form likelihood approximations \cite{ait}). \cite{stramer_bognar} perform a delicate Bayesian analysis of this model using  the MCWM algorithm (defined in \cite{neil} and discussed in \cite{beaumont2003estimation} and \cite{andrieu2009pseudo}) and  the GIMH algorithm \cite{beaumont2003estimation}. Here, we compare MHC with  MCWM, referring to  \cite{stramer_bognar}  for a detailed analysis of the CIR model using GIMH.

One common approach in the literature for Bayesian estimation of diffusion models is to consider estimation on the basis of discrete measurements as a classic missing-data problem (see \cite{roberts_stramer} for irreducible diffusion contexts). The idea is to introduce latent observations between every two consecutive data points. The time-step interval $[0,\Delta]$ is thus partitioned into $M$ sub-intervals, each of length $h=\Delta/M$. The granularity $M$ should be large enough so that the grid is sufficiently fine to yield more accurate likelihood approximations. With the introduction of latent variables, the pseudo-marginal approach naturally comes to mind as a possible inferential approach. The MCWM variant (described in Section 3 of \cite{stramer_bognar}) alternates between simulating $\theta$, conditionally on the missing data blocks, say $U$, and  then updating $U$, given $\theta$.
We will be using the following enumeration for the missing data $U=(u^{ij}_{km})$: we have a replicate index $1\leq i\leq n$, a discrete time index $0\leq j\leq T$, an index of the intermittent auxiliary series $1\leq m\leq M$   and 
an index $1\leq k\leq N$ for the number of replications inside MCWM. Given $\theta$, one can generate the missing data  using the Modified Brownian Bridge (MBB) sampler \citep{durham_gallant}.
Denote with $X=[\bm x_1,\dots, \bm x_n]'$ an $n\times (T+1)$ matrix of observations where $x_{i0}=x_0$ is the initial condition. The CIR model is an interesting  test bed for both MCWM and our MHC approach,   because the transition density is {\em actually known}  (i.e. non-central $\chi^2$   \citep{cir}). We can thereby make comparisons with an exact algorithm which constructs the likelihood  from the exact transition function.
\begin{figure}[!t]
\vspace{-1cm}
\scalebox{0.8}{\includegraphics[height=10cm,width=20cm]{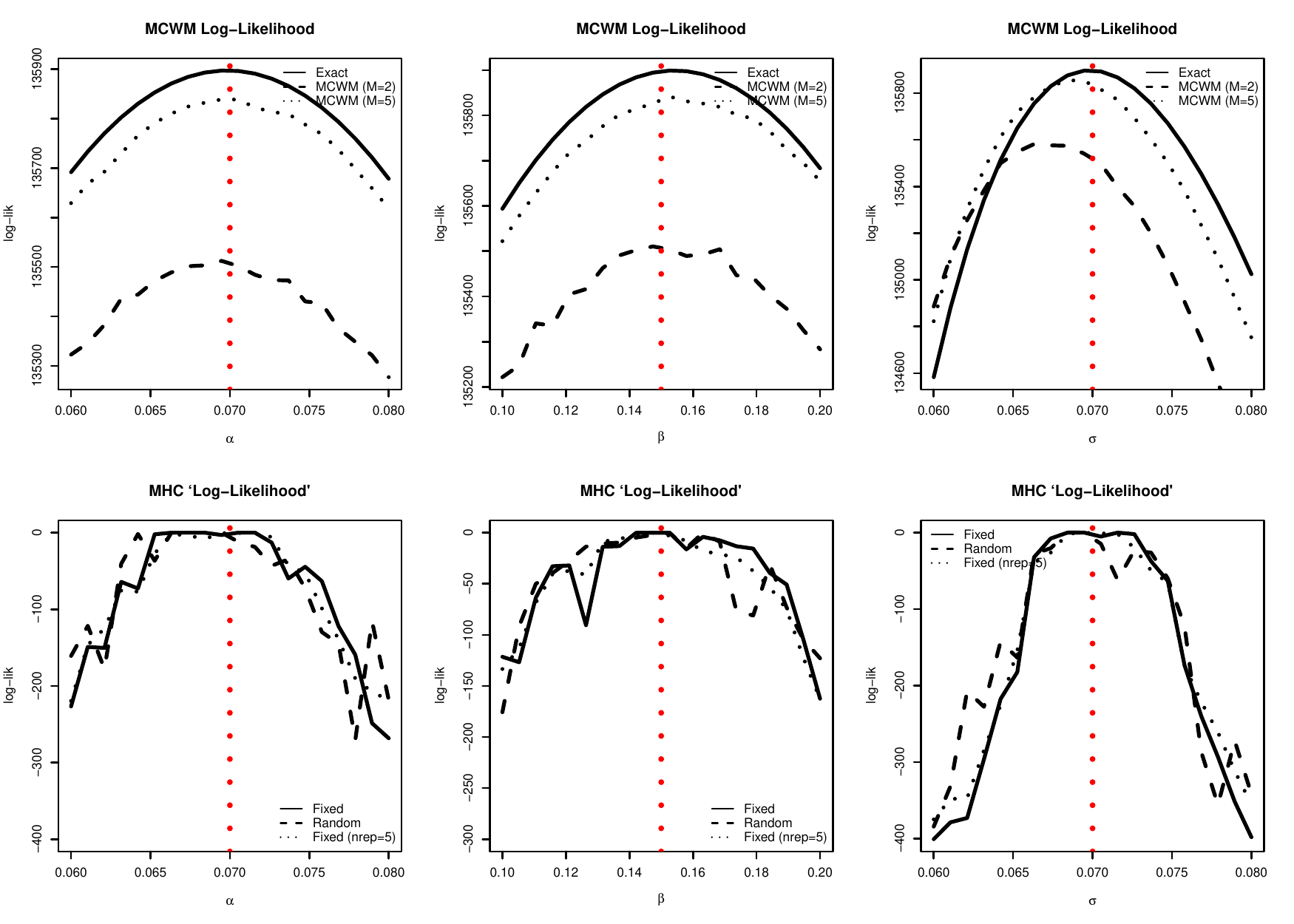}}
\caption{\small Plots of the exact   and estimated log-likelihood function (up to a constant) for MCWM (upper panel using $N=M^2=25$) and MHC (lower panel using fixed and random generators).  Log-Likelihood slice over (Left) $\alpha$ keeping ($\beta_0,\sigma_0$), (Middle) over $\beta$ fixing $(\alpha_0,\sigma_0)$ and (Right)  over $\sigma$ keeping $(\alpha_0,\beta_0)$. }\label{fig:CIR_liks}
\end{figure}
The likelihood can be, however,  stochastically approximated as
\begin{equation}\label{eq:llik_approx_MCWM}
\wh p_\theta(X)=\prod_{i=1}^n \prod_{j=0}^{T-1}\wh \pi(x_{ij+1}\C x_{ij},\theta),\quad\text{where}\quad  \wh\pi(x_{ij+1}\C x_{ij},\theta)=\frac{1}{N}\sum_{k=1}^N R_M(u_k^{ij}),
\end{equation}
where  $u_k^{ij}=(u_{k0}^{ij},\dots,u_{kM}^{ij})' \in R^{M+1}$ is the $k^{th}$ sample of the brownian bridge (described in (3) in \cite{stramer_bognar}) stretching from  $u^{ij}_{k0}=x_{ij}$ and $u^{ij}_{kM}=x_{ij+1}$ and where
$$
R_M(u_k^{ij})=\frac{\prod_{m=0}^{M-1}\phi\left(u_{k m+1}^{ij}; u_{k m}^{ij}+h\beta(\alpha- u_{k m}^{ij})\,,\, \sigma\sqrt{h u_{k m}^{ij}}\right)}{
\prod_{m=0}^{M-2}\phi\left(u_{k m+1}^{ij}; u_{k m}^{ij}+\frac{x_{ij+1}-u^{ij}_{km}}{M-m}\,,\, \sigma\sqrt{h (M-m-1)/(M-m)u_{k m}^{ij}}\right)}
$$
where $\phi(x;\mu,\sigma)$ denotes the normal density with a mean $\mu$ and a standard deviation $\sigma$. Regarding the choice of $M$ and $N$,   asymptotic arguments exist for choosing $N=M^2$  and \cite{stramer_bognar} make thorough comparisons for various choices of $M,N$ and also implement the ('exact' version having the correct stationary distribution) GIMH (see their Section 4) which recycles latent data $U$.  There are some delicate issues regarding dependency between $\sigma$ and $U$ in GIMH and we refer the reader to \cite{stramer_bognar} for further details. 

The true data consist of $n=100$  samples generated using the package \texttt{sde} (using the function \texttt{sed.sim} with `rcCIR' initialized at $x_0=0.1$)  using $\Delta=1$ and $T=500$ and using\footnote{These values are close to parameter estimates found for FedFunds data analyzed in Stramer and Bognar (2011).} $\theta^0=(0.07,0.15,0.07)'$.
In order to implement MHC, we use the LASSO-regularized logistic regression (using an R package \texttt{glmnet} with a value $\lambda$ chosen by $10$-fold cross-validation) using the entire series $\bm x_i$ as predictors. 
While using the entire series  is useful for identifying the location parameter $\alpha$, capturing more subtle aspects of the series such as speed of fluctuation and spread are needed to identify $(\beta,\sigma)$. To this end,
 we add summary statistics (mean, log-variance, auto-correlations at lag 1 and 2 as well as the first 3 principal components of $X$)  yielding the total of $507$ predictors (denoted with $\bm z_i$). We consider both fixed and random generators where, for the fixed variant, we  fix the random seed before generating fake data which essentially corresponds to having a deterministic generative mapping.

\begin{figure}[!t]
\vspace{-1cm}
\includegraphics[height=5.3cm,width=5.3cm]{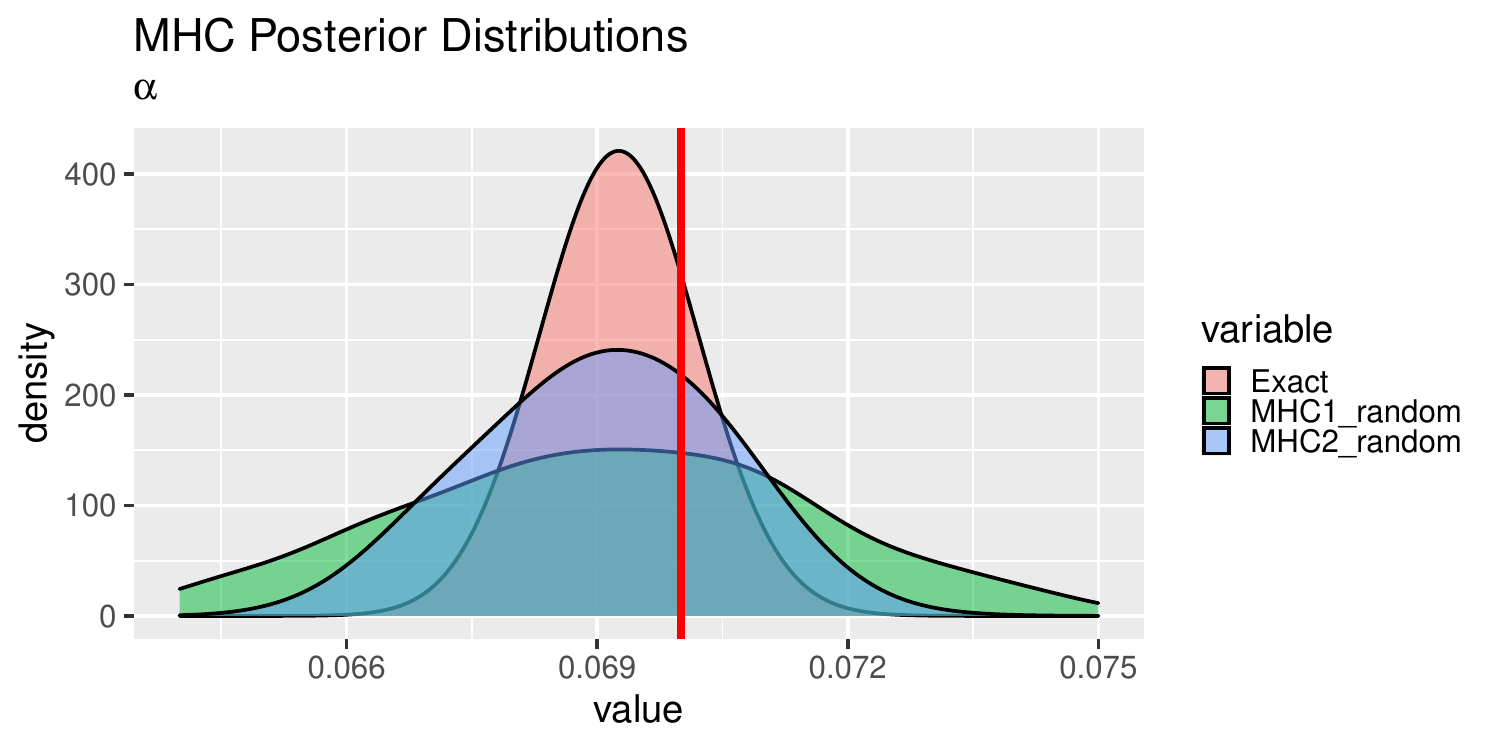} \includegraphics[height=5.3cm,width=5.3cm]{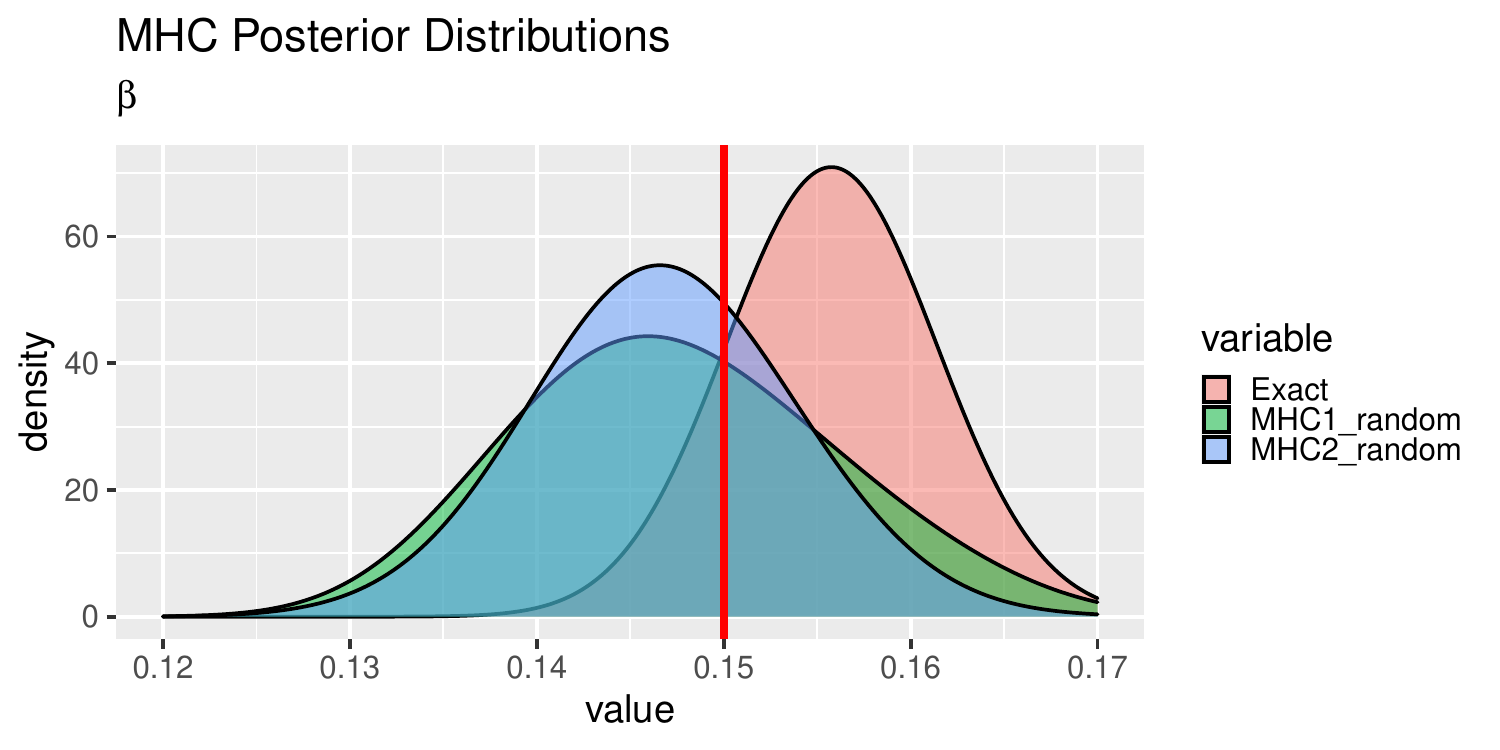} 
\includegraphics[height=5.3cm,width=5.3cm]{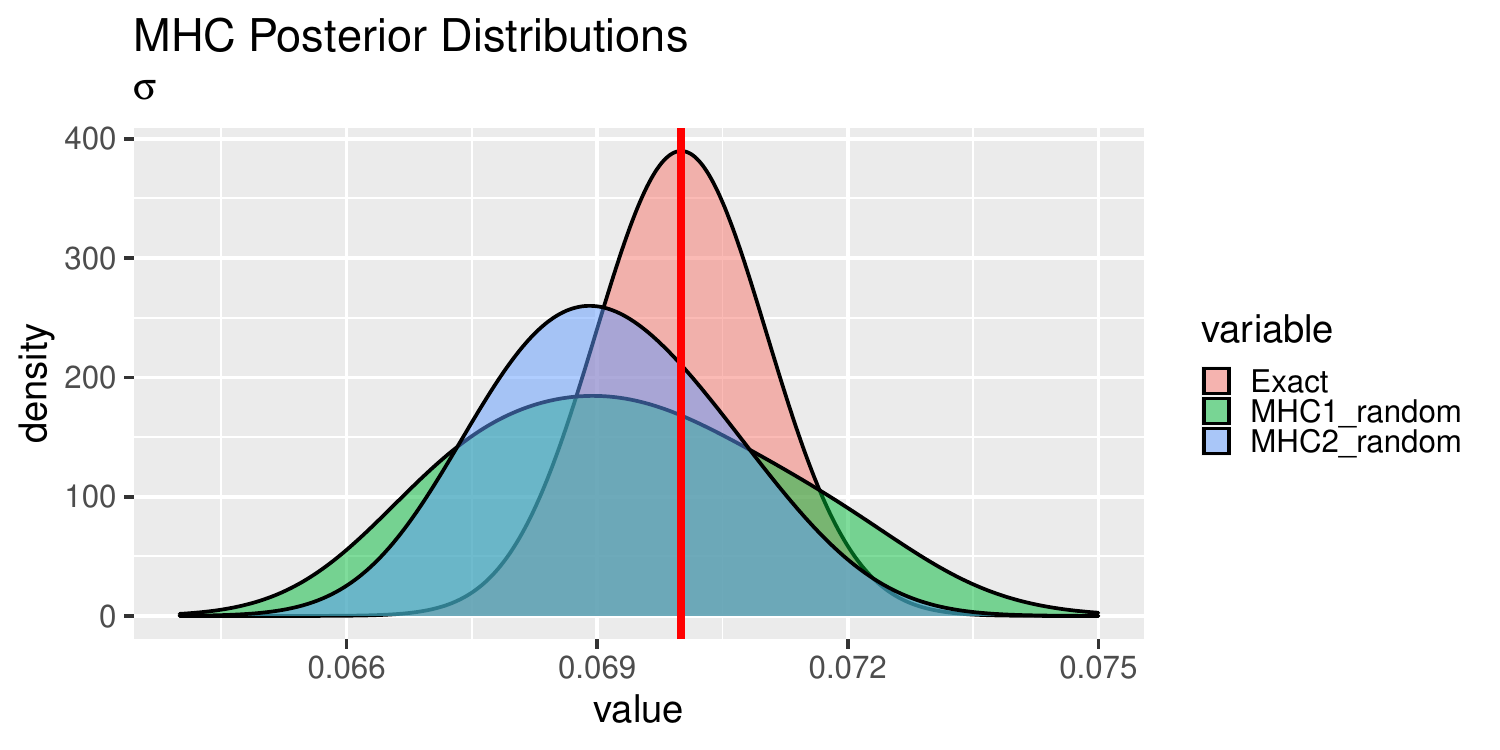} 
\caption{\small Smoothed posterior densities obtained for the CIR model by simulation using exact MH and MHC using $nrep=1$ (green) and $nrep=5$ (blue). Vertical lines are the true values.}\label{fig:MHC_rg}
\end{figure}

We compare the MCWM likelihood approximations obtained in MCWM (using \eqref{eq:llik_approx_MCWM})  with various choices $N=M^2$ with the exact one using the explicit transition distribution (top panel in Figure \ref{fig:CIR_liks}). We can see that, even for a small value of $N=2$, the likelihood approximation seems to have a correct shape and is peaked close to the true values (marked by vertical dotted lines). The plots show likelihood slices along each parameter, one at a time, fixing the others at their true values. The approximation quality improves for $M=5$ and $N=M^2$. 
The lower panel in Figure \ref{fig:CIR_liks} portrays our classification-based log-likelihood (ratio) estimates $\eta=\sum_{i=1}^n\log[(1-\hat D(\bm z_i))/\hat  D(\bm z_i)]$ for the fixed and random generators.
 The curves  are nicely wrapped around the true values (perhaps even more so than for MCWM) with no visible systematic bias (even for the fixed generator). While, in the fixed  case (solid lines), we would expect entirely smooth curves,  recall that our classifier is based on cross-validation which introduces some randomness (thereby the wiggly estimate).  The wigglyness can be  alleviated by averaging over ($nrep$) many fake data replicates (dotted lines). The random generator (dashed lines) yields slightly more variable curves compared to the fixed design, as was expected. 
 These plots  indicate that MHC `pseudo-likelihood' contains relevant inferential information.

\begin{figure}[!t]
\vspace{-1cm}
\includegraphics[height=5.3cm,width=5.3cm]{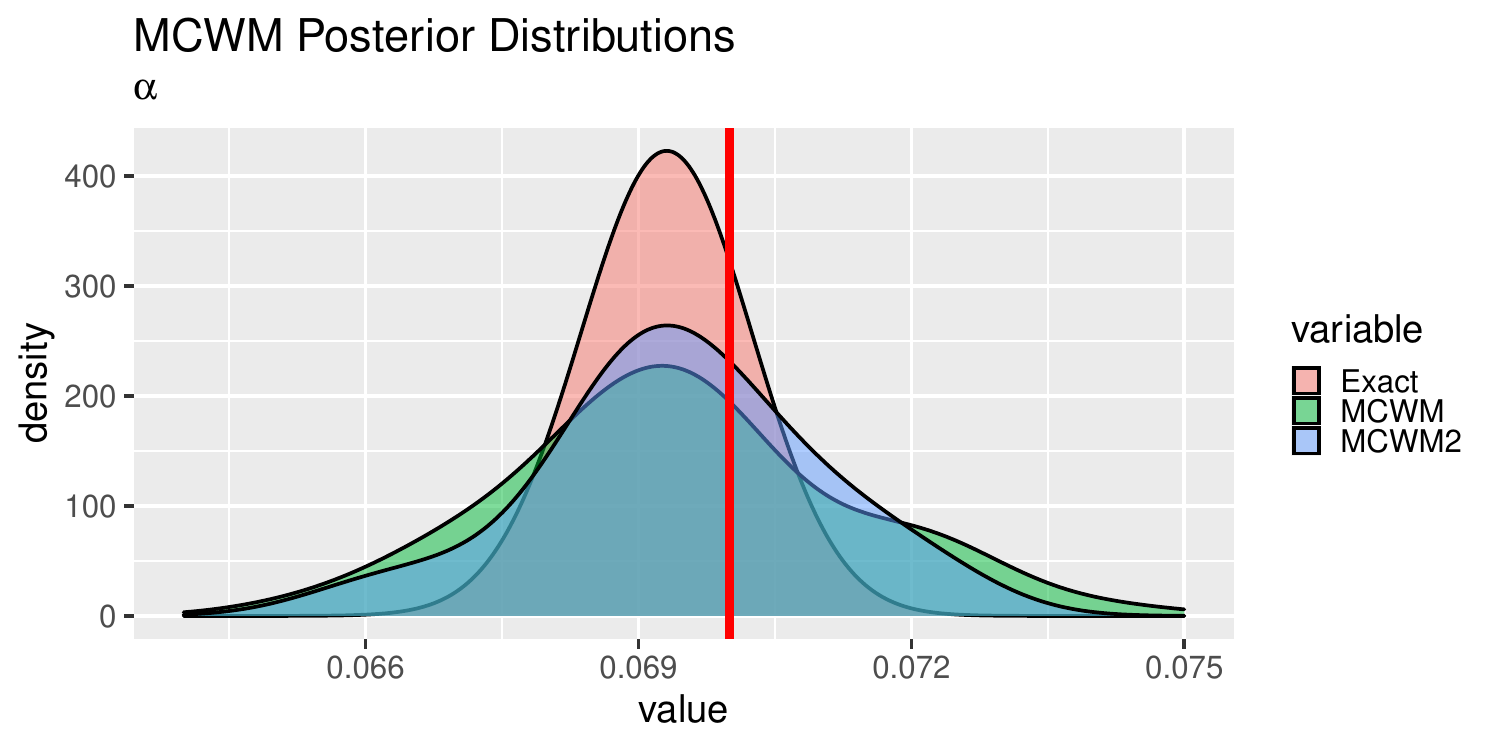} \includegraphics[height=5.3cm,width=5.3cm]{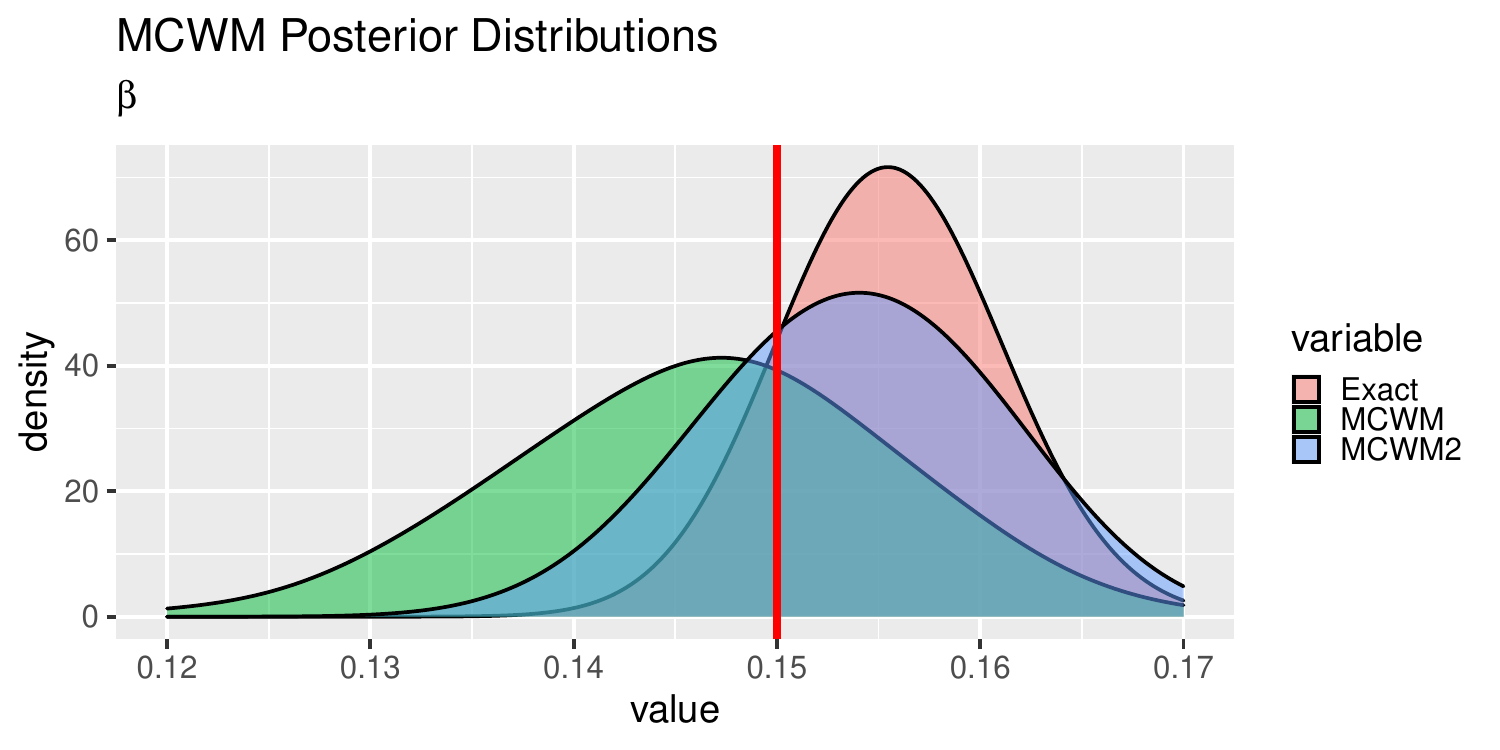} 
\includegraphics[height=5.3cm,width=5.3cm]{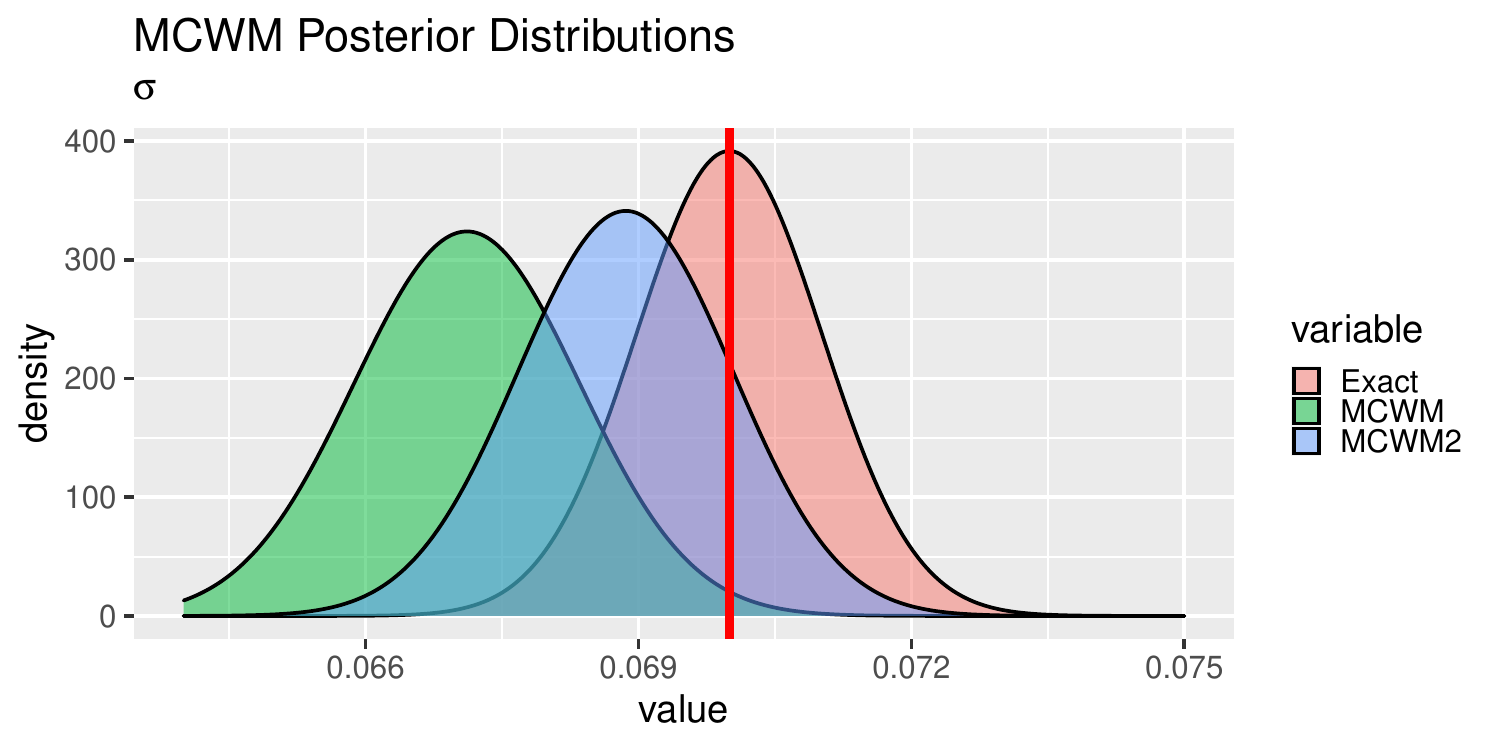} 
\caption{\small  Smoothed posterior densities obtained by simulation using  MCWM (with $N=M^2$) for $M=2$ (MCWM1 green) and $M=5$ (MCWM2 blue). Vertical lines are the true values.}\label{fig:MCWM}
\end{figure}

To implement the exact MH, MCWM and MHC (with $nrep\in\{1,5\}$), we adopt the same prior settings as in \cite{stramer_bognar}, where $\pi(\theta)=\mathbb I_{(0,1)}(\alpha) \mathbb I_{(0,\infty)}(\beta)\sigma^{-1}\mathbb I_{(0,\infty)}(\sigma)$. We also use their random walk proposals.\footnote{With probability $2/3$ propose a joint move $(\alpha^\star,\beta^\star)$ by generating $\alpha^\star\sim U(\alpha-0.01,\alpha+0.01)$ and  $\beta^\star\sim U(\beta-0.01,\beta+0.01)$ and  with probability $1/3$ propose $\sigma^\star\sim U(\sigma-0.01,\sigma+0.01)$. To increase the acceptance rate of the exact MH algorithm, we change the window from $0.01$ to $0.005$.}  All three algorithms are initialized at the same perturbed truth and ran for $10\,000$ iterations with a burnin period $1\,000$. 
Smoothed posterior densities obtained by simulation using the exact MH and  MHC are in Figure \ref{fig:MHC_rg}  (random generator using $nrep\in\{1,5\}$ where fixed generator is portrayed in Figure \ref{fig:MHC_sup} in the Appendix).  {The trace-plots of $10\,000$ iterations are depicted in Figure \ref{fig:MHC_trace_CIR1} and \ref{fig:MHC_trace_CIR2} in the Supplement, where we can see that the random generator variant yields smaller acceptance rates (especially for $\sigma$) which masks the fact  that the random generator sampler generally yields more spread-out posterior approximations.}
Smoothing out the likelihood ratio by averaging over $nrep$ repetitions reduces variance where fixed and random generators seem to yield qualitatively similar results in this example (this is why we have not used the de-biasing variant here). Histograms (together with the demarkation of $95\%$ credible set) are in Figure \ref{fig:MHC_hist_CIR1}  in the Appendix. Compared with the smoothed densities obtained from MCWM (using $N=M^2$ with $M\in\{2,5\}$ in Figure \ref{fig:MCWM}) we can see that MHC yields posterior reconstructions that are wrapped    more closely around the true values. Increasing $M$, MCWM yields posterior reconstructions that are getting closer to the actual posterior (not necessarily centered more narrowly around the truth). Recall, however, that MCWM generates Markov chains whose invariant distribution is not necessarily the exact posterior. The posterior summaries (means and $95\%$ credible intervals) are reported in Table \ref{tab:CIR} (Supplement). Interestingly, both MCWM intervals for $\sigma$ {\em do not} include the true value $0.07$ and the MCWM computation is considerably slower relative to MHC. In particular, MCWM with $N=M^2=25$ (resp. $N=M^2=4$)  took $238.6$ hours  (resp. $15.9$ hours) while MHC with $nrep=5$ (resp. $nrep=1$) took  $13.9$ hours (resp. $4.6$ hours). {See Table 2 in the Supplement for more run-time and effective sample size comparisons.}



\vspace{-0.5cm}

\subsection{Lotka-Volterra Model}\label{sec:lotka}
The Lotka-Volterra (LV) predator-prey model \citep{wilkinson} describes population evolutions   in ecosystems where predators   interact with prey.
The model is deterministically prescribed via a system of first-order non-linear ordinary differential equations with four parameters $\theta=(\theta_1,\dots,\theta_4)'$ controlling (1) the rate $r_1^t=\theta_1X_tY_t$ of a  predator being born, 
(2) the rate  $r_2^t=\theta_2X_t$ of a predator dying, (3) the rate $r_3^t=\theta_3Y_t$ of a prey being born and (4) the rate  $r_4^t=\theta_4X_tY_t$ of a prey dying. 
Given the initial population sizes  $X_0$ (predators) and $Y_0$ (prey) at time $t=0$,  the process can be simulated from exactly using the Gillespie algorithm \citep{gillespie}. In particular, this algorithm samples times  to an event from  an exponential distribution (with a rate $\sum_{j=1}^4 r_j^t$) and then picks one of the 4 reactions  with probabilities proportional to their individual rates $r_j^t$. Despite {being} easy to sample from, the likelihood for this model is unavailable which makes this model a natural candidate for ABC \citep{prangle17} and other likelihood-free methods \citep{papa16,meeds}.
It is not entirely obvious, however, how to implement the pseudo-marginal approach since there is no explicit hierarchical model structure with a conditional likelihood, given latent data, which could be marginalized through simulation to obtain a likelihood estimate.  

\begin{figure}[!t]
\vspace{-1cm}
\centering
\begin{subfigure}[b]{0.3\textwidth}
\scalebox{0.3}{\includegraphics{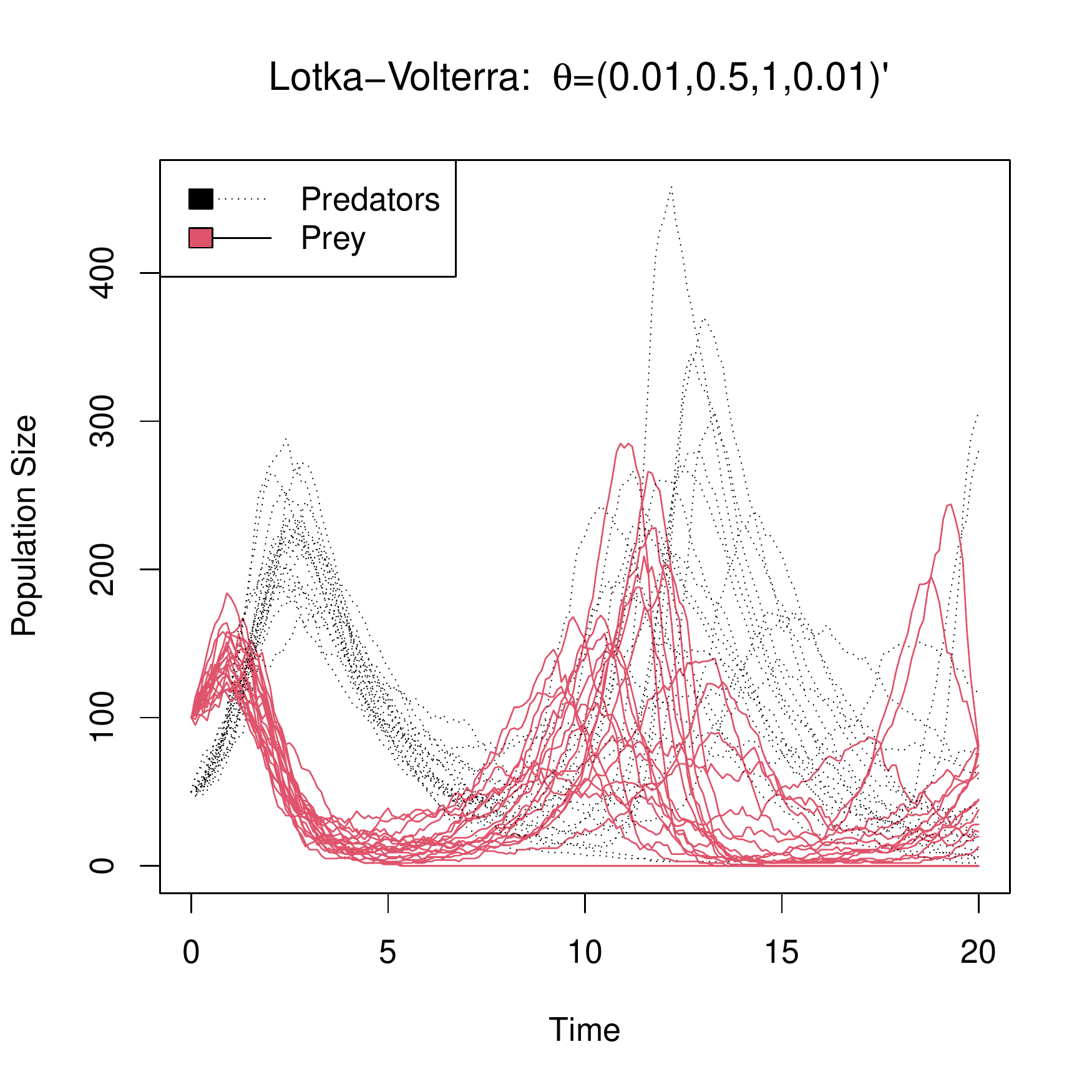}}
\caption{$\theta=(0.01,0.5,1,0.01)'$}\label{fig:predator_a}
\centering
\end{subfigure}
\begin{subfigure}[b]{0.3\textwidth}
\scalebox{0.3}{\includegraphics{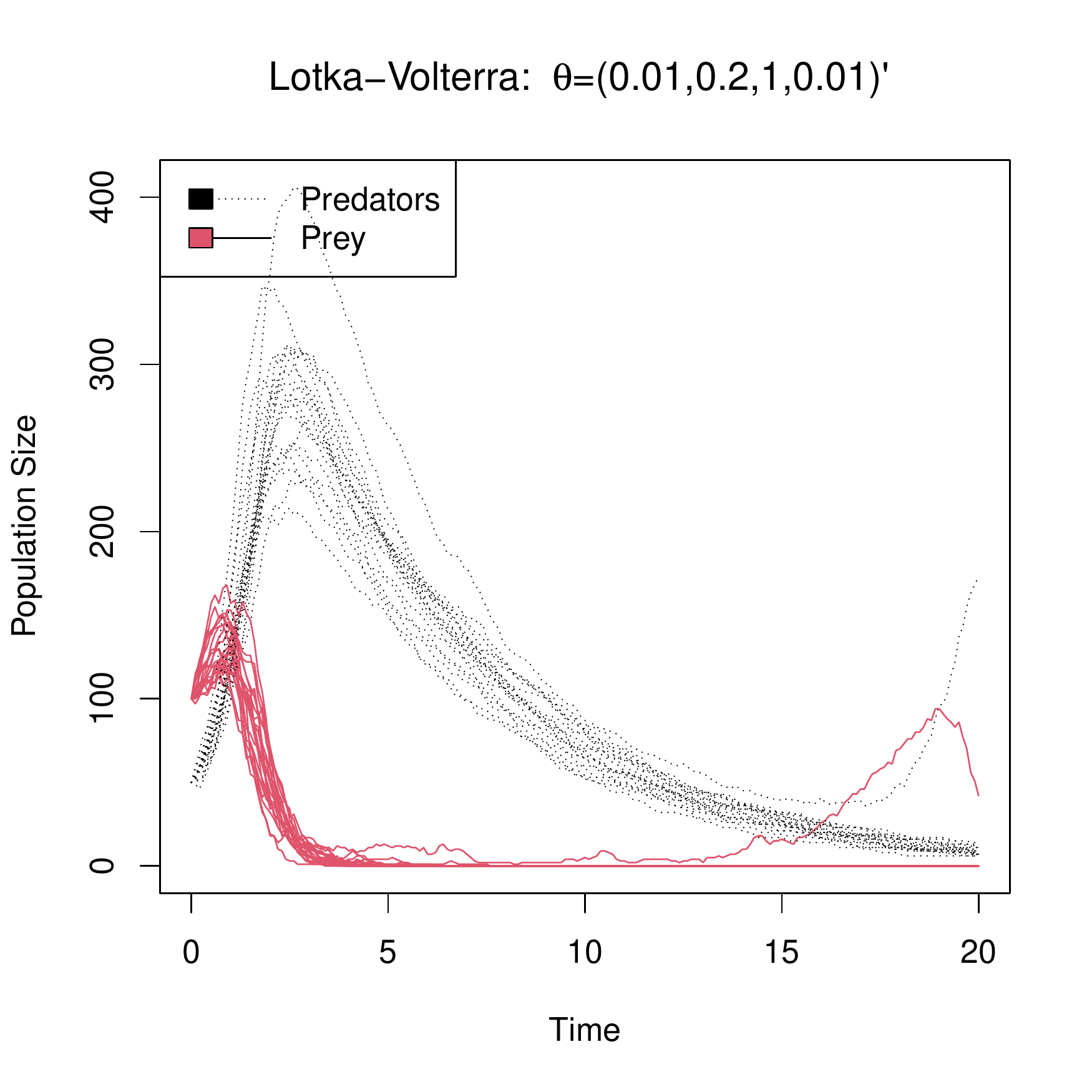}}
\caption{$\theta=(0.01,0.2,1,0.01)'$}\label{fig:predator_b}
\centering
\end{subfigure}
\begin{subfigure}[b]{0.3\textwidth}
\scalebox{0.3}{\includegraphics{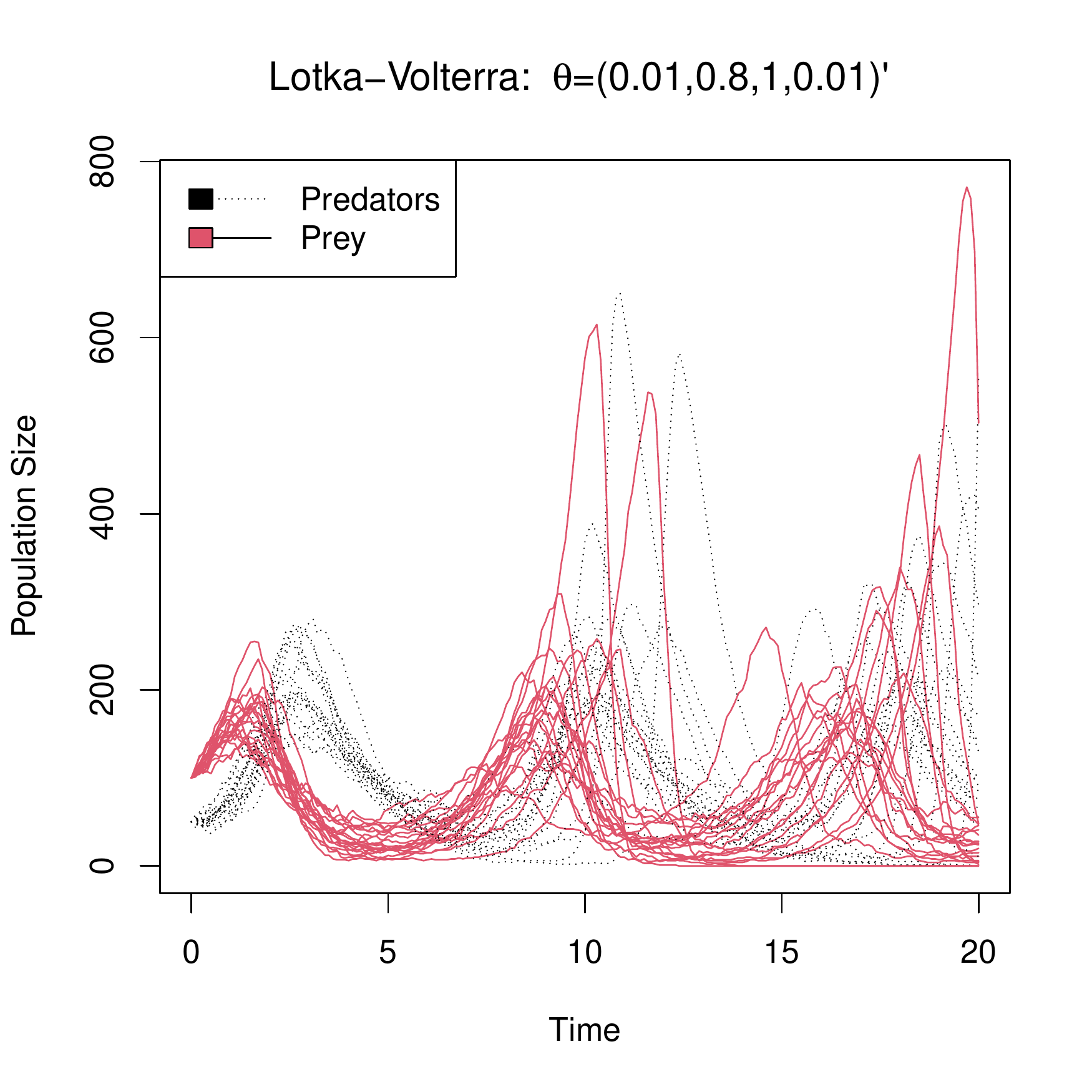}}
\caption{$\theta=(0.01,0.8,1,0.01)'$}\label{fig:predator_c}
\centering
\end{subfigure}
\caption{\small  Lotka-Volterra realizations for three choices of $\theta$}\label{fig:lotka}
\end{figure}

In our experiments, each simulation is started at $X_0=50$ and $Y_0=100$ simulated over $20$ time units and recorded observations every $0.1$ time units, resulting in a series of $T=201$ observations each. 
We plot $n=20$ time series realizations for three particular choices of $\theta$ in Figure \ref{fig:lotka} which differ in the second argument $\theta_2$ with larger values accentuating the cyclical behavior. Slight shifts in parameters  result in (often) dramatically different trajectories. Typical behaviors include (a) predators quickly eating all the prey and then slowly decaying (as in Figure \ref{fig:predator_b}), (b) predators quickly dying out and then the prey population sky-rocketing. For certain carefully tuned values  $\theta$, the two populations exhibit oscillatory behavior. For example, in Figure \ref{fig:predator_a} and \ref{fig:predator_c} we can see how the value $\theta_2$  determines the frequency of the population renewal cycle. We rely on the ability of  the discriminator to tell such different shapes apart. The real data ($n=20$) is generated under the scenario (a) with $\theta^0=(0.01,0.5,1,0.01)'$.

ABC analyses of this model reported in the literature have relied on various summary statistics\footnote{In addition to the summary statistics suggested in \cite{papa16}, we have also considered the classification accuracy ABC summary statistic $CA=\frac{1}{n+m}\left(\sum_{i=1}^n{\hat D}(\bm x_i)+\sum_{j=1}^m(1-\hat D(\wt{\bm x}_j)\right)$ proposed by \cite{gdkc2018}. This ABC version did not provide much better results.}    including the mean, log-variance, autocorrelation (at lag 1 and 2) of each series as well as their cross-correlation \cite{papa16}.  These summary statistics  seem to be able to capture the oscillatory behavior (at different frequencies) and distinguish it from exploding population growth (see Figure \ref{fig:lotka_summaries} in Section \ref{sec:lv_comparisons} of the Supplement). This creates hope that ABC based on these summary statistics has the capacity to provide a reliable posterior reconstruction. In a similar vein, we plotted the estimated log-likelihood  $\eta\equiv \sum_{i=1}^n\log[(1-\hat D(\bm x_i))/\hat D(\bm x_i)]$  where $\bm x_i=(X^i_1,\dots, X^i_T, Y^i_1,\dots, Y^i_T)'$ after training a classifier (using the R package \texttt{glmnet} and \texttt{randomForest}) on $m=n$ fake data observations $\wt{\bm x}_i=
(\wt X^i_1,\dots, \wt X^i_T, \wt Y^i_1,\dots, \wt Y^i_T)'$ for $1\leq i\leq m$. 
See heat-map plots of the estimated likelihood $\eta$ as a function of $(\theta_2,\theta_3)'$  (Figure \ref{fig:llik1}) and as a function of $(\theta_1,\theta_4)'$  (Figure \ref{fig:llik2} for \texttt{glmnet} and Figure \ref{fig:llik3} for \texttt{randomForest}), keeping the remaining parameters at the truth.  Figure \ref{fig:llik2}  reveals a sharp  spike (approximating a point-mass) around the true value at $\theta_1=\theta_4=0.01$  in a otherwise vastly flat landscape.  This peculiar likelihood property  may require a very careful consideration of initializations and proposal densities for MH and the prior domain for ABC. {The random forest classifier, however, did not yield as spiky likelihood estimators (Figure \ref{fig:llik3}), suggesting that it will be less sensitive to MHC initialization. We also inspected estimated log-likelihoods using  (a) the fixed reference approach of \cite{pham} which uses the fake (not observed) data for comparisons  (see Figure \ref{fig:lliks_new} in the Supplement) and (b) the marginal reference approach  \cite{hermans} which trains the log-likelihood ratio estimator ahead of the MCMC simulation. See Figure \ref{fig:lliks_new} in the Supplement for log-likelihood estimators using various  training database sizes $m\in\{5\,000,10\,000,50\,000\}$. We can see that with enough training observations (i.e. $m=50\,000$), the estimator is smooth and peaked around the truth. However, the training time alone (including fake data generation) took roughly $2.7$ hours. }

\begin{figure}[!t]
\vspace{-1cm}
\centering
\begin{subfigure}[b]{0.32\textwidth}
\scalebox{0.3}{\includegraphics{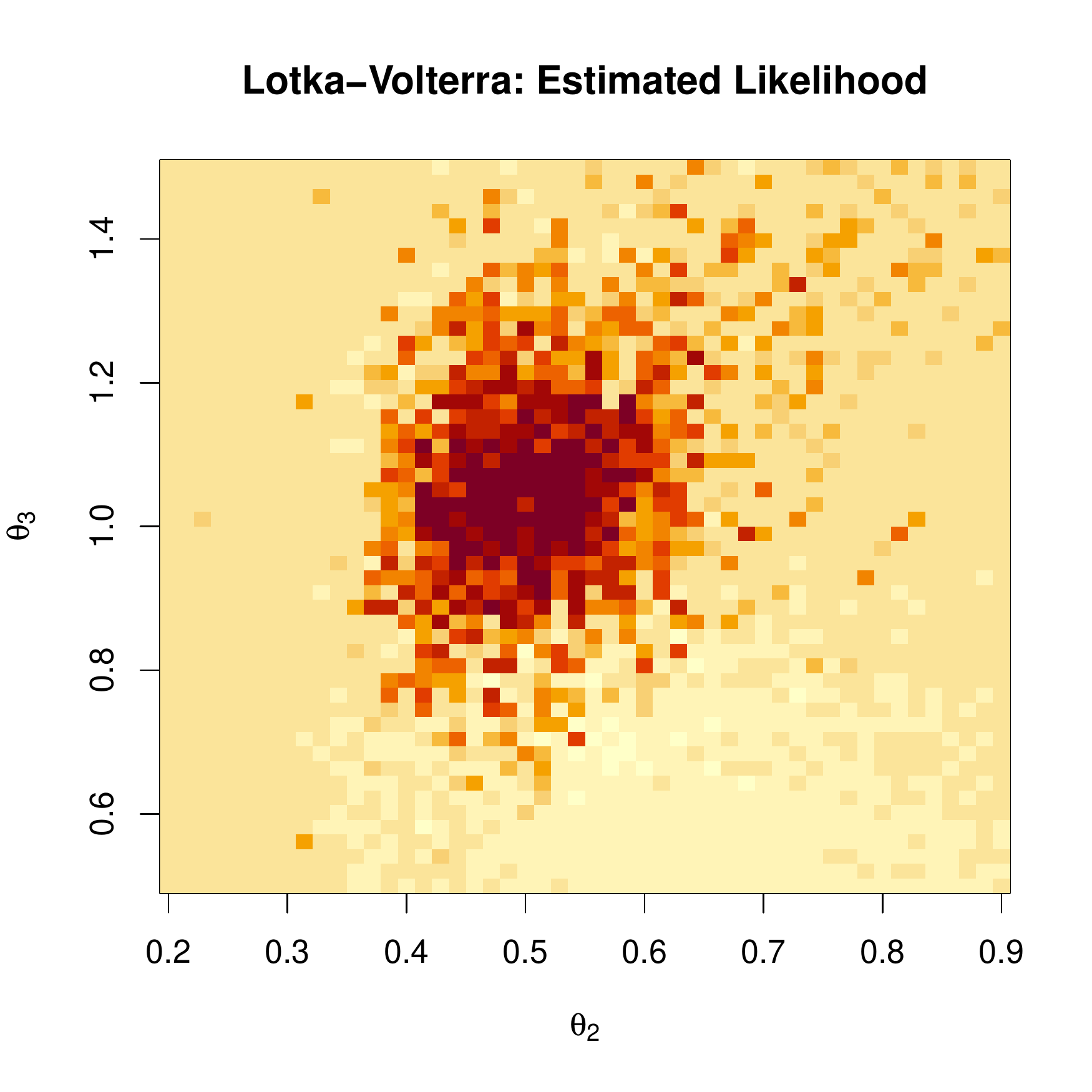}}
\caption{\scriptsize $\eta(\theta_2,\theta_3)$ for $\theta_1=\theta_4=0.01$\\ \centering \texttt{glmnet} }\label{fig:llik1}
\end{subfigure}
\begin{subfigure}[b]{0.32\textwidth}
\centering
\scalebox{0.3}{\includegraphics{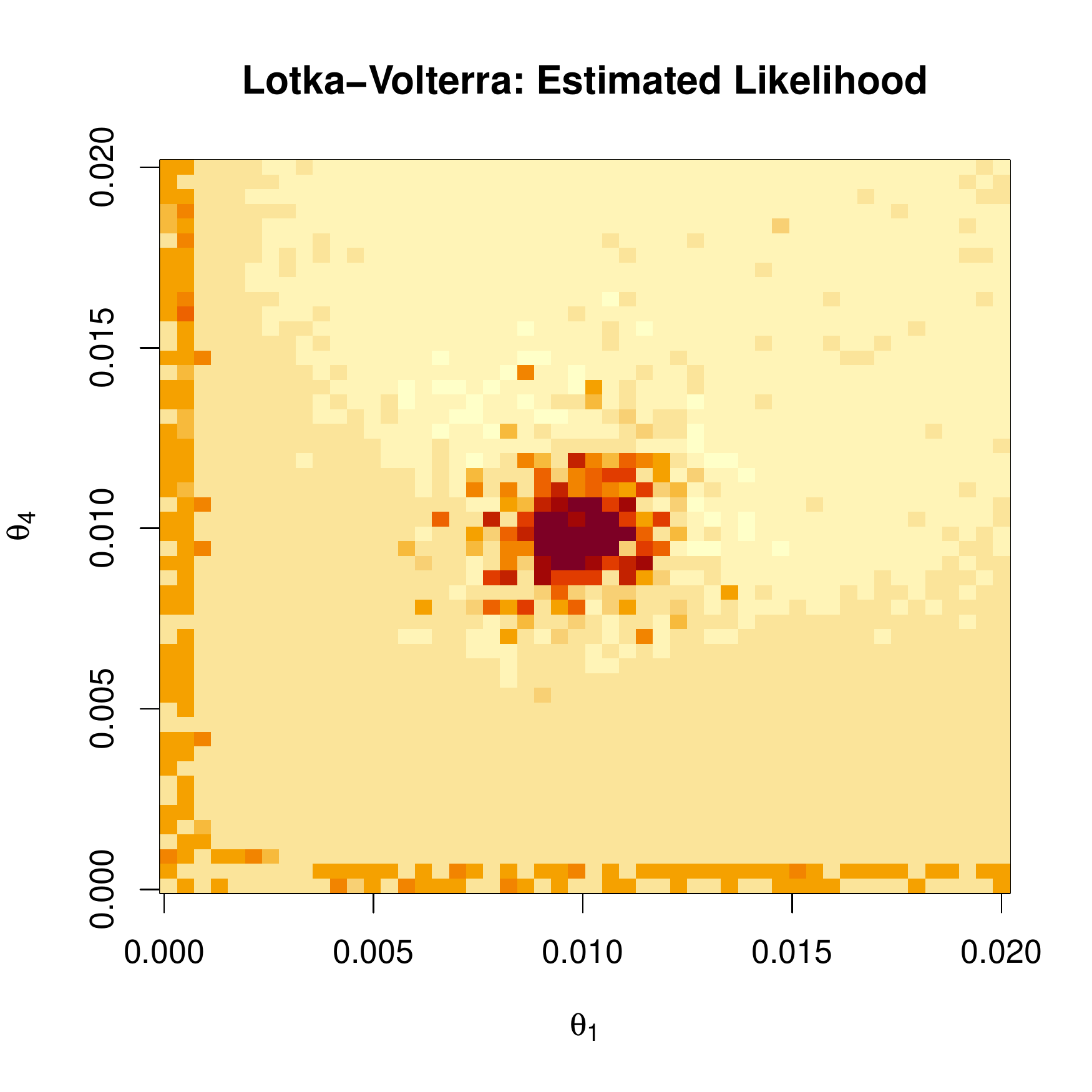}}
\caption{\scriptsize $\eta(\theta_1,\theta_4)$ for $\theta_2=0.5,\theta_3=1$ \\ \centering \texttt{glmnet}}\label{fig:llik2}
\end{subfigure}
\begin{subfigure}[b]{0.32\textwidth}
\scalebox{0.3}{\includegraphics{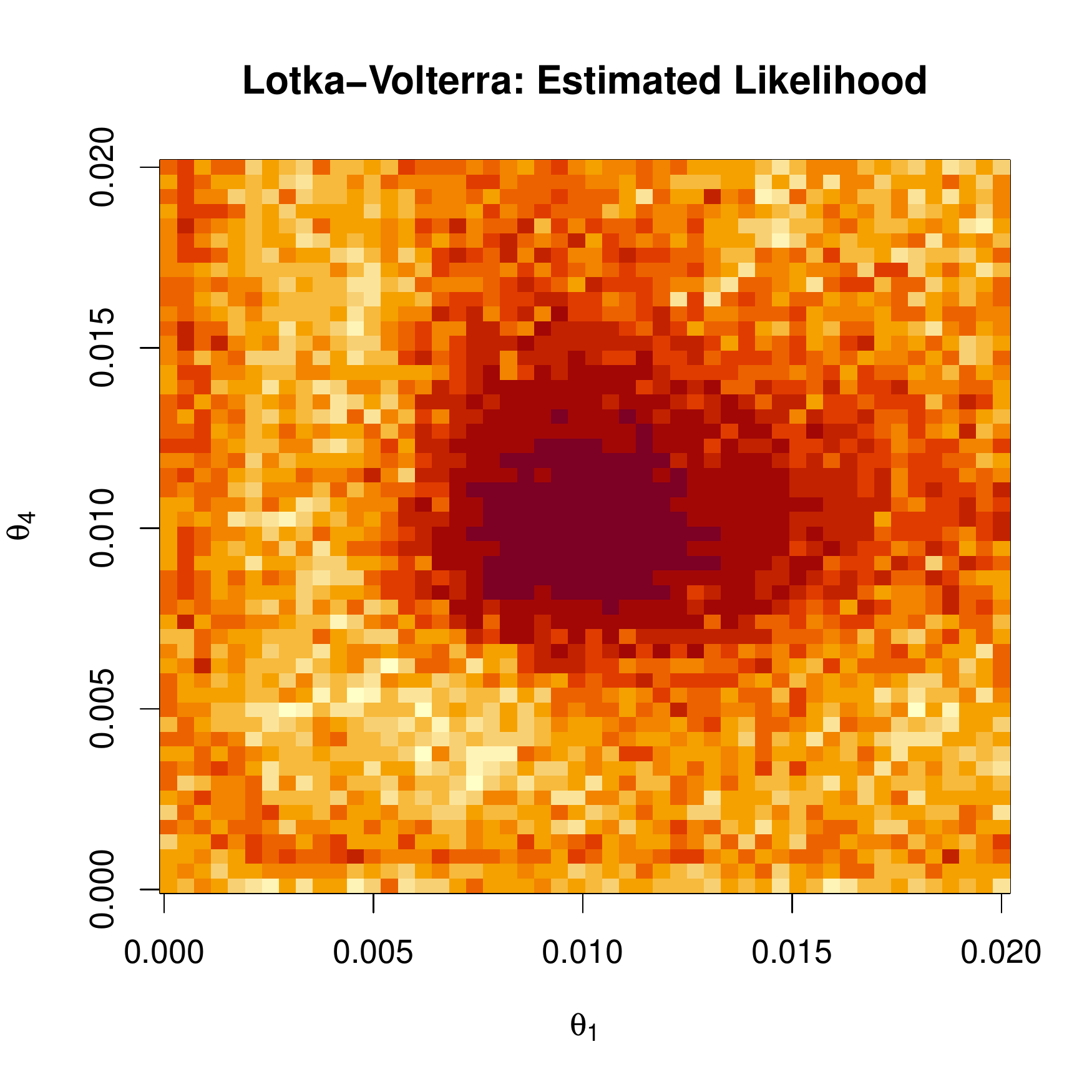}}
\caption{\scriptsize $\eta(\theta_1,\theta_4)$ for $\theta_2=0.5,\theta_3=1$ \\ \centering \texttt{randomForest}}\label{fig:llik3}
\end{subfigure}
\caption{\small  Lotka-Volterra model. Estimated log-likelihood for a grid of parameters.}\label{fig:lliks}
\end{figure}

 \begin{table}[!t]
\scalebox{0.75}{ \begin{tabular}{l| c c c| c c c| c c c| c c c| c c}
 \hline\hline
  &  \multicolumn{3}{c}{$\theta_1^0=0.01$} &\multicolumn{3}{c}{$\theta_2^0=0.5$}& \multicolumn{3}{c}{$\theta_3=1$}& \multicolumn{3}{c}{$\theta_4=0.01$} &  Time (h) \\
  \hline
 Method& $\bar\theta$ & $l$ & $u$   &  $\bar\theta$ & $l$ & $u$   & $\bar\theta$ & $l$ & $u$   & $\bar\theta$ & $l$ & $u$                                               & \\
 ABC1         	                        &0.015  &0.003   &0.038     & 0.554 & 0.037  & 0.985          &1.315  & 0.189 & 1.955            &0.012   &0.004  &0.029     &4.7 \\
 ABC2     		                        &0.016  &0.003   &0.042     & 0.604 & 0.087  & 0.980          &1.259  & 0.205 & 1.971            &0.013   &0.003  &0.024     &47.46  \\
 MHC (rf)                                   &0.01  &0.008   &0.011     &0.490 & 0.421  & 0.575          &1.063  & 0.872 & 1.258            &0.01   &0.009  &0.012          &2.58 \\ 
 MHC (glmnet)                         &0.01  &0.009   &0.015     &0.514 & 0.417  & 0.636          &1.026  & 0.826 & 1.323            &0.01   &0.008  &0.013          &2.45 \\
 ALR MH \scriptsize($m=10\,000$)     &0.006  &0.002   &0.012     &0.477 & 0.285  & 0.849          &1.041  & 0.565 & 1.622            &0.06   &0.001  &0.013     &0.68 \\
 ALR MH \scriptsize($m=50\,000$)     &0.008  &0.005   &0.013     &0.527 & 0.376  & 0.675          &1.199  & 0.864 & 1.752            &0.008   &0.003  &0.013       &3.76 \\
 Classif MH  \scriptsize$(m=20)$             &0.01  &0.008   &0.012     &0.5 & 0.405  & 0.612           &1.027  & 0.798 & 1.307            &0.01   &0.008  &0.013             &6.24 \\
 Classif MH \scriptsize $(m=100)$           &0.01  &0.009   &0.011     &0.501 & 0.45  & 0.558          &1.026  & 0.875 & 1.187            &0.010   &0.009  &0.012          &38.2 \\
\hline\hline
\end{tabular}}
\caption{Posterior summary statistics using ABC1 ($M=10\,000$ and $r=100$), ABC2 ($M=100\,000$ and $r=1\,000$) and MH variants ($M=10\,000$ with burnin $1\,000$). $\bar\theta$ is the posterior mean, $l$ and $u$
denote the lower and upper boundaries of $95\%$ credible intervals. MHC variants are implemented with random forests and \texttt{glmnet} classifiers. \texttt{ALR MH} is the amortized likelihood ratio MH of \cite{hermans} (using random forests). \texttt{Classif MH} is the classifier MCMC of \cite{pham} (using random forests). $m$ is the fake data sample size. }
\label{lotka:stats}
  \end{table}

 In order to facilitate ABC analysis, we have used an informative uniform prior $\theta\sim U(\Xi)$  with a restricted domain $\Xi=[0,0.1]\times[0,1]\times[0,2]\times [0,0.1]$ so that the procedure does not waste time sampling from unrealistic parameter values. These values were chosen based on a visual inspection of simulated evolutions, where we have seen only a limited range of values to yield periodic behavior. 
In a pilot ABC run,   we rank $M=10\,000$ ABC samples based on $\varepsilon$ in an ascending manner and report the histogram of the first $r=100$ samples  (Figure \ref{fig:ABC1} in the Appendix, the upper panel).  We can see that ABC was able to narrow down  the region of interest for $(\theta_1,\theta_4)$, but is still largely uninformative about parameters  $(\theta_2,\theta_3)$ with histograms stretching from the boundaries of the prior domain. Given how narrow the range of  likely parameter values  is (according to Figure \ref{fig:lliks}), the likelihood of encountering such values even  under the restricted uniform prior is  still quite negligible. We thereby tried many more  ABC samples   ($M=100\,000$ which took $47.46$ hours) only to find out that the histograms (top $r=1\,000$ samples) did not improve much (Figure \ref{fig:ABC1} in the Appendix, the lower panel).
 
The hostile likelihood landscape will create problems not only for ABC but also for Metropolis-Hastings. Indeed, initializations  that are too far from the likelihood domain may result in Markov chains wandering aimlessly\footnote{This is a valid concern for the \texttt{glmnet} classifier.} in the vast plateaus for a long time.  Rather than competing with ABC, a perhaps more productive strategy is to combine the strengths of both. We have thereby used the pilot ABC run (the closest $100$ samples out of $M=10\,000$ which took roughly $4$ hours) to obtain ABC approximated posterior means $\wh\theta=(0.015, 0.55, 1.31, 0.012)'$. We use these to initialize\footnote{{MHC with random forests did not seem as sensitive to initialization compared to \texttt{glmnet}.}} all  MH procedures to accelerate convergence (i.e. prevent painfully long burn-in).  To implement MHC, we define a Gaussian random walk proposal for log-parameter values with a proposal standard deviation $0.05$ and deploy the same prior as for the ABC method. We use the random generator variant here, where the fixed one can be implemented (for example) by fixing the random seed prior generating the fake data.  {We compare our approach with the Classification Metropolis-Hastings of \cite{pham}  and the marginal reference approach of \cite{hermans}, both with  the default \texttt{randomForest} implementation, with the same  ABC initialization and $M=10\,000$ MCMC iterations. Details on the comparisons and implementations are in Section \ref{sec:lv_comparisons} in the Supplement.} The histograms after $M=10\,000$ iterations  with the burn-in period $1\,000$ are portrayed in Figure \ref{fig:hists} in the Supplement. The trace plots (Figure \ref{fig:trace1}, \ref{fig:trace2} and \ref{fig:trace3} in the Supplement) show reasonable mixing where the \texttt{glmnet} classifier shows more sensitivity to MHC initialization. The histograms report much sharper concentration around true values (compared to ABC in Figure \ref{fig:lliks}) and were obtained under considerable less time (again compared to ABC with $M=100\,000$).
The posterior summaries (mean $\bar\theta$ and $95\%$ credible intervals $(l,u)$ are compared in Table \ref{lotka:stats}. Compared to ABC, we can see that not only  MHC posterior means  accurately estimate the true parameters, but the $95\%$ credible intervals are much tighter and thereby perhaps more informative for inference. {There are differences (both in timing and performance) depending on the choice of the classifier, with  random forests yielding better and faster results. The  Classification MH method of \cite{pham} yields similar results as MHC (rf) but is slower due to the fact that more fake data need to be generated at each step. The marginal approach (ALR MH) of \cite{hermans} perhaps needed more training samples to learn the likelihood-ratio generator. We have used only $m=50\,000$ database samples so that the overall computing time (training together with MCMC simulation) would be comparable to MHC.}

We believe that MHC provided an inferential framework which was not attainable using neither ABC alone (with our choice of summary statistics), nor the pseudo-marginal method. Potentially more fruitful ABC results could be obtained with {ABC-within-Gibbs style algorithms \cite{clarte}} or by instead deploying Wasserstein distance between the empirical distributions of real and fake data \citep{wasser}, in particular its curve-matching variants tailored for dependent data. {Other promising alternative is the sequential neural likelihood approach \cite{papa18} which uses masked autoregressive flows to learn the conditional probability density of data given parameters and adaptively adjusts  the proposal distribution for sampling new parameter values.

  }

\vspace{-0.5cm}
\section{Discussion}\label{sec:discuss}
This paper develops an approximate Metropolis-Hastings (MH) posterior sampling method for when the likelihood is not tractable. By deploying a Generator and  a Classifier (similarly as in Generative Adversarial Networks \cite{gan}),  likelihood ratio estimators are obtained which are then plugged into the MH sampling routine. One of the main distinguishing features of our work  is that we consider two variants: (1) a fixed generator design yielding biased samples, and (2) a random generator yielding more dispersed samples. 
{Compared to related existing approaches  \citep{pham,hermans}, our approach uses   observed data as the contrasting dataset. This ultimately poses limitations on the classifier when the sample size $n$ is small in which case the approaches \cite{pham} and \cite{hermans} are more appropriate.}
We provide a thorough frequentist analysis of the stationary distribution including convergence rates and asymptotic normality. Under suitable differentiability assumptions, we conclude that correct shape and location can be recovered by deploying a debiasing combination of the fixed and random generator variants. We demonstrate a very satisfactory performance on non-trivial time series examples which render existing techniques (such as PM or ABC) less practical. 


\section*{Acknowledgements}

Tetsuya Kaji gratefully acknowledges the support from the Richard N.\ Rosett Faculty Fellowship and the Liew Family Faculty Fellowship at the University of Chicago Booth School of Business.
Veronika Rockova gratefully acknowledges the support from James S. Kemper Faculty Scholarship and the National Science Foundation (DMS: 1944740).

\bibliographystyle{apalike}

\clearpage

{\begin{center}{\LARGE SUPPLEMENTAL MATERIALS}\end{center}}

\section{Notation}\label{sec:notation}
 The following notation has been  used throughout the manuscript. We employ the operator notation for expectation, e.g., $P_0 f=\int f dP_0$ and $\mathbb{P}_m^\theta f=\frac{1}{m}\sum_{i=1}^m f(X_i^\theta)$. The {\em $\varepsilon$\hyp{}bracketing number} $N_{[]}(\varepsilon,\mathcal{F},d)$ of a set $\mathcal{F}$ with respect to a premetric $d$ is the minimal number of $\varepsilon$\hyp{}brackets in $d$ needed to cover $\mathcal{F}$.%
\footnote{A {\em premetric} on $\mathcal{F}$ is a function $d:\mathcal{F}\times\mathcal{F}\to\mathbb{R}$ such that $d(f,f)=0$ and $d(f,g)=d(g,f)\geq0$.}
The {\em $\delta$\hyp{}bracketing entropy integral} of $\mathcal{F}$ with respect to $d$ is 
$$
J_{[]}(\delta,\mathcal{F},d)\vcentcolon=\int_0^\delta\sqrt{1+\log N_{[]}(\varepsilon,\mathcal{F},d)}d\varepsilon.
$$ We denote the usual Hellinger semi-metric for independent observations as 
$$
d_n^2(\theta,\theta')=\frac{1}{n}\sum_{i=1}^n\int (\sqrt{p_{\theta,i}}-\sqrt{p_{\theta',i}})^2\d\mu_i.$$ 
 Next, $K(p_{\theta_0}^{(n)},p_{\theta}^{(n)})=\sum_{i=1}^nK(p_{\theta_0,i},p_{\theta,i})$ denotes the Kullback-Leibler divergence between product measures and $V_{2}(f,g)=\int f|\log(f/g)|^2\d\mu$.  Define $\langle a,b\rangle=\sum_{i=1}^da_ib_i$ for $a,b\in\R^d$.

\section{Proof of Theorem \ref{thm:rate:res}}\label{sec:proof_thm:rate:res}

The following lemma bounds the Kullback\hyp{}Leibler divergence and variation by possibly non\hyp{}diverging multiples of the Hellinger distance.%
\footnote{\cref{lem:KL} (iv) first appeared in \citet[Lemma 5]{kmp2020}. We reproduce the proof here as it is used to prove other statements.}
This can be used to derive sharper rates of posterior contraction in models with unbounded likelihood ratios \citep[see also][p.\ 199 and Appendix B]{gv2017}.

\begin{lem} \label{lem:KL}
For probability measures $P$ and $P_0$ such that $P_0(p_0/p)<\infty$, let $M\coloneqq\inf_{c\geq1}c P_0(\frac{p_0}{p}\mid\frac{p_0}{p}\geq[1+\frac{1}{2c}]^2)$ where $P_0(\cdot\mid A)=0$ if $P_0(A)=0$.
For $k\geq2$, the following hold.
\begin{enumerate}[(i)]
	\item $-P_0\log\frac{p}{p_0}\leq(3+M)h(p,p_0)^2$.
	\item $P_0|\log\frac{p}{p_0}|^k\leq2^{k-1}\Gamma(k+1)(2+M)h(p,p_0)^2$.
	\item $P_0|\log\frac{p}{p_0}-P_0\log\frac{p}{p_0}|^k\leq2^{2k-1}\Gamma(k+1)(2+M)h(p,p_0)^2$.
	\item $\|\frac{1}{2}\log\frac{p}{p_0}\|_{P_0,B}^2\leq(2+M)h(p,p_0)^2$.
	\item $\|\frac{1}{4}(\log\frac{p}{p_0}-P_0\log\frac{p}{p_0})\|_{P_0,B}^2\leq(2+M)h(p,p_0)^2$.
\end{enumerate}
Here, $\|f\|_{P,B}\coloneqq\sqrt{2P(e^{|f|}-1-|f|)}$ is the Bernstein ``norm''.
\end{lem}

\begin{proof}
(iv) Using $e^{|x|}-1-|x|\leq(e^{x}-1)^2$ for $x\geq-\frac{1}{2}$ and $e^{|x|}-1-|x|<e^x-\frac{3}{2}$
for $x>\frac{1}{2}$,
\[
	\Bigl\|\log\sqrt{\tfrac{p}{p_0}}\Bigr\|_{P_0,B}^2
	\leq2P_0\Bigl(\sqrt{\tfrac{p}{p_0}}-1\Bigr)^2\mathbbm{1}\bigl\{\tfrac{p}{p_0}\geq\tfrac{1}{e}\bigr\}+2P_0\Bigl(\sqrt{\tfrac{p_0}{p}}-\tfrac{3}{2}\Bigr)\mathbbm{1}\bigl\{\tfrac{p_0}{p}>e\bigr\}.
\]
The first term is bounded by $2h(p,p_0)^2$.
For every $c\geq1$,
\begin{multline*}
	P_0\Bigl(\sqrt{\tfrac{p_0}{p}}-\tfrac{3}{2}\Bigr)\mathbbm{1}\bigl\{\tfrac{p_0}{p}>e\bigr\}
	\leq P_0\Bigl(\sqrt{\tfrac{p_0}{p}}-1-\tfrac{1}{2c}\Bigr)\mathbbm{1}\Bigl\{\sqrt{\tfrac{p_0}{p}}\geq1+\tfrac{1}{2c}\Bigr\}\\
	=P_0\Bigl(\sqrt{\tfrac{p_0}{p}}\geq1+\tfrac{1}{2c}\Bigr)\Bigl[P_0\Bigl(\sqrt{\tfrac{p_0}{p}}-1\Bigm|\sqrt{\tfrac{p_0}{p}}\geq1+\tfrac{1}{2c}\Bigr)-\tfrac{1}{2c}\Bigr].
\end{multline*}
Since $x-\frac{1}{2c}\leq\frac{c}{2}x^2$ for every $x$,
\begin{multline*}
	P_0\Bigl(\sqrt{\tfrac{p_0}{p}}-1\Bigm|\sqrt{\tfrac{p_0}{p}}\geq1+\tfrac{1}{2c}\Bigr)-\tfrac{1}{2c}
	\leq\tfrac{c}{2}\Bigl[P_0\Bigl(\sqrt{\tfrac{p_0}{p}}-1\Bigm|\sqrt{\tfrac{p_0}{p}}\geq1+\tfrac{1}{2c}\Bigr)\Bigr]^2\\
	\leq\tfrac{c}{2}P_0\Bigl(\tfrac{p_0}{p}\Bigm|\sqrt{\tfrac{p_0}{p}}\geq1+\tfrac{1}{2c}\Bigr)P_0\Bigl(\Bigl[1-\sqrt{\tfrac{p}{p_0}}\Bigr]^2\Bigm|\sqrt{\tfrac{p_0}{p}}\geq1+\tfrac{1}{2c}\Bigr)
\end{multline*}
by the Cauchy\hyp{}Schwarz inequality.
Then the result follows.

{
(i) Write
\(
	-P_0\log\tfrac{p}{p_0}=P_0(\tfrac{p}{p_0}-1-\log\tfrac{p}{p_0})+P(p_0=0)
\).
With $x-1-\log x\leq 3(\sqrt{x}-1)^2$ for $x>\tfrac{1}{3}$ and $\frac{1}{x}-1-\log\frac{1}{x}<2(\sqrt{x}-\frac{3}{2})$ for $x\geq3$,
\[
	P_0\bigl(\tfrac{p}{p_0}-1-\log\tfrac{p}{p_0}\bigr)\leq 3 P_0\Bigl(\sqrt{\tfrac{p}{p_0}}-1\Bigr)^2\mathbbm{1}\bigl\{\tfrac{p}{p_0}>\tfrac{1}{3}\bigr\}+2P_0\Bigl(\sqrt{\tfrac{p_0}{p}}-\tfrac{3}{2}\Bigr)\mathbbm{1}\bigl\{\tfrac{p_0}{p}\geq3\bigr\}.
\]
The second term is bounded as above.
The first term and $P(p_0=0)=\int(\sqrt{p}-\sqrt{p_0})^2\mathbbm{1}\{p_0=0\}$ are collectively bounded by $3 h(p,p_0)^2$.
}

(ii) Since $e^x-1-x\geq x^k/\Gamma(k+1)$ for $k\geq2$ and $x\geq0$,%
\footnote{$\Gamma(k-1)\geq\int_x^\infty y^{k-2}e^{-y}dy\geq x^{k-2}e^{-x}$ implies $\frac{d^2}{dx^2}(e^x-1-x)\geq\frac{d^2}{dx^2}x^k/\Gamma(k+1)$.}
\(
	P_0|\log\tfrac{p}{p_0}|^k\leq 2^{k-1}\Gamma(k+1)\|\tfrac{1}{2}\log\tfrac{p}{p_0}\|_{P_0,B}^2
\).
Then, apply (iv).

{
(iii) By the triangle and Jensen's inequalities,
\(
	P_0|\log\tfrac{p}{p_0}-P_0\log\tfrac{p}{p_0}|^k\leq[(P_0|\log\tfrac{p}{p_0}|^k)^{1/k}+|P_0\log\tfrac{p}{p_0}|]^k\leq2^k P_0|\log\tfrac{p}{p_0}|^k
\) for $k\geq1$.
Then, use (ii).
}

(v) By the convexity of $e^{|x|}-1-|x|$ and Jensen's inequality,
\(
	\|\tfrac{1}{4}(\log\tfrac{p}{p_0}-P_0\log\tfrac{p}{p_0})\|_{P_0,B}^2\leq\tfrac{1}{2}\|\tfrac{1}{2}\log\tfrac{p}{p_0}\|_{P_0,B}^2+\tfrac{1}{2}\|P_0\tfrac{1}{2}\log\tfrac{p}{p_0}\|_{P_0,B}^2
	\leq\|\tfrac{1}{2}\log\tfrac{p}{p_0}\|_{P_0,B}^2
\).
With (iv) follows the result.
\end{proof}

\begin{proof}[Proof of \cref{thm:rate:res}]

For $D\in\mathcal{D}_{n,\delta_n}^\theta$, write $\mathbb{P}_n(\log\tfrac{1-D}{1-D_\theta}-\log\tfrac{D}{D_\theta})$ as
\[
	P_0\log\tfrac{1-D}{1-D_\theta}-P_0\log\tfrac{D}{D_\theta}+(\mathbb{P}_n-P_0)\log\tfrac{1-D}{1-D_\theta}-(\mathbb{P}_n-P_0)\log\tfrac{D}{D_\theta}.
\]
Since $\log(x)\leq 2(\sqrt{x}-1)$ for $x>0$, we have
\[
	-2 P_0\Bigl(\sqrt{\tfrac{D_\theta}{D}}-1\Bigr)\leq P_0\log\tfrac{D}{D_\theta}\leq 2 P_0\Bigl(\sqrt{\tfrac{D}{D_\theta}}-1\Bigr).
\]
By the Cauchy\hyp{}Schwarz inequality and \cref{asm:2},
\begin{gather*}
	P_0\Bigl|\sqrt{\tfrac{D}{D_\theta}}-1\Bigr|\leq \sqrt{P_0\Bigl(\sqrt{\tfrac{D}{D_\theta}}-1\Bigr)^2}=h_\theta(D,D_\theta)\leq\delta_n,\\
	P_0\Bigl|\sqrt{\tfrac{D_\theta}{D}}-1\Bigr|\leq\sqrt{P_0\tfrac{D_\theta}{D}}\sqrt{P_0\Bigl(1-\sqrt{\tfrac{D}{D_\theta}}\Bigr)^2}\leq\sqrt{M}\delta_n.
\end{gather*}
Therefore, $|P_0\log\frac{D}{D_\theta}|\leq 2(1\vee\sqrt{M})\delta_n$.
Next, let $W\vcentcolon=\sqrt{\frac{1-D}{1-D_\theta}}-1$ and define a function $R$ by $\log(1+x)=x-\frac{1}{2}x^2+\frac{1}{2}x^2 R(x)$, which implies $R$ is increasing and $R(x)<1$ for $x>-1$, and $R(x)=O(x)$ as $x\to 0$.
With this, write
\[
	P_0\log\tfrac{1-D}{1-D_\theta}=2 P_0 W-P_0 W^2+P_0 W^2 R(W).
\]
By the Cauchy\hyp{}Schwarz inequality,
\begin{gather*}
	P_0|W|\leq\sqrt{P_0\tfrac{p_0}{p_\theta}}\cdot h_\theta(1-D,1-D_\theta)\leq\sqrt{M}\delta_n,\\
	P_0 W^2\leq\sqrt{(P_0+P_\theta)\bigl(\tfrac{p_0}{p_\theta}\bigr)^2(\sqrt{1-D}-\sqrt{1-D_\theta})^2}\cdot h_\theta(1-D,1-D_\theta).
\end{gather*}
Since $D$ and $D_\theta$ are bounded by $0$ and $1$,
\[
	(P_0+P_\theta)\bigl(\tfrac{p_0}{p_\theta}\bigr)^2(\sqrt{1-D}-\sqrt{1-D_\theta})^2\leq P_0\bigl(\tfrac{p_0}{p_\theta}\bigr)^2+P_0\tfrac{p_0}{p_\theta}\leq 2 M.
\]
Therefore, $P_0 W^2\leq\sqrt{2M}\delta_n$.
Next, the residual is bounded as
\begin{align*}
	|P_0 W^2 R(W)|&\leq P_0 W^2|R(W)|\mathbbm{1}\{W\leq-\tfrac{1}{5}\}+P_0 W^2|R(W)|\mathbbm{1}\{W>-\tfrac{1}{5}\}\\
	&\leq P_0(-R(W)\mathbbm{1}\{W\leq-\tfrac{1}{5}\})+P_0 W^2|R(-\tfrac{1}{5})\vee R(W)|,
\end{align*}
where the second inequality uses $W\geq-1$ and $R$ increasing.
Since $R<1$ and $P_0 W^2\leq\sqrt{2 M}\delta_n$, the second term is also bounded by $\sqrt{2 M}\delta_n$.
With $0<-R(x)<-2\log(1+x)$ for $x\leq-\frac{1}{5}$, the first term is bounded by
\begin{align*}
	P_0\bigl(\log\tfrac{1-D_\theta}{1-D}\mathbbm{1}\{W\leq-\tfrac{1}{5}\}\bigr)
	&=P_0\bigl(\tfrac{1-D}{1-D_\theta}\log\tfrac{1-D_\theta}{1-D}\cdot\tfrac{1-D_\theta}{1-D}\mathbbm{1}\{W\leq-\tfrac{1}{5}\}\bigr)\\
	&\leq\sup_{\sqrt{x}-1\leq-1/5}|x\log\tfrac{1}{x}|\cdot P_0\bigl(\tfrac{1-D_\theta}{1-D}\mathbbm{1}\{W\leq-\tfrac{1}{5}\}\bigr).
\end{align*}
The supremum is $1/e$.
The second term is bounded by $P_0(W\leq-\frac{1}{5})P_0(\frac{1-D_\theta}{1-D}\mid\frac{1-D_\theta}{1-D}\geq\frac{25}{16})\leq P_0(W\leq-\frac{1}{5})M$ by \cref{asm:2}. By Markov's inequality, $P_0(W\leq-\frac{1}{5})\leq 25 P_0 W^2\leq 25\sqrt{2 M}\delta_n$.
Thus, $|P_0 W^2 R(W)|\leq(1+25 M/e)\sqrt{2 M}\delta_n$.
Altogether, we have $|P_0\log\frac{1-D}{1-D_\theta}|\leq(\sqrt{2}+2+25 M/e)\sqrt{2 M}\delta_n$.

Next, we bound
\(
	\mathbb{E}^\ast\sup_{D\in\mathcal{D}_{n,\delta_n}^\theta}|\sqrt{n}(\mathbb{P}_n-P_0)\log\tfrac{D}{D_\theta}|
\).
Under \cref{asm:2}, an analogous argument as \cref{lem:KL} (iv) yields
\[
	\bigl\|\tfrac{1}{2}\log\tfrac{D}{D_\theta}\bigr\|_{P_0,B}^2\leq 2(1+M)h_\theta(D,D_\theta)^2\leq 2(1+M)\delta_n^2.
\]
By \citet[Lemma 3.4.3]{vw1996}, we have
\[
	\mathbb{E}^\ast\sup_{D\in\mathcal{D}_{n,\delta_n}^\theta}\bigl|\sqrt{n}(\mathbb{P}_n-P_0)\log\tfrac{D}{D_\theta}\bigr|\lesssim J\Bigl(1+\tfrac{J}{\delta_n^2\sqrt{n}}\Bigr)
\]
for $J\vcentcolon=J_{[]}(\delta_n,\{\log\tfrac{D}{D_\theta}:D\in\mathcal{D}_{n,\delta_n}^\theta\},\|\cdot\|_{P_0,B})$.
Note that a $\delta_n$\hyp{}bracket in $\mathcal{D}_{n,\delta_n}^\theta$ induces a $C\delta_n$\hyp{}bracket in $\{\log\frac{D}{D_\theta}\}$ for some constant $C$ since
\[
	\bigl\|\log\tfrac{u}{D_\theta}-\log\tfrac{\ell}{D_\theta}\bigr\|_{P_0,B}^2\leq 4 P_0\Bigl(\sqrt{\tfrac{u}{\ell}}-1\Bigr)^2=O(d_\theta(u,\ell)^2)
\]
by \cref{asm:2}.
Therefore, $J\leq J_{[]}(\delta_n,\mathcal{D}_{n,\delta_n}^\theta,d_\theta)$ and hence $J(1+\frac{J}{\delta_n^2\sqrt{n}})\lesssim\delta_n^2\sqrt{n}$ by \cref{asm:1}.

Finally, we bound
\(
	\mathbb{E}^\ast\sup_{D\in\mathcal{D}_{n,\delta_n}^\theta}|\sqrt{n}(\mathbb{P}_n-P_0)\log\tfrac{1-D}{1-D_\theta}|
\).
As in \cref{lem:KL} (iv), we obtain
\(
	\rho^2\vcentcolon=\bigl\|\tfrac{1}{2}\log\tfrac{1-D}{1-D_\theta}\bigr\|_{P_0,B}^2\leq 2(1+M)P_0 W^2\leq 2(1+M)\sqrt{2 M}\delta_n
\).
Therefore, by \citet[Lemma 3.4.3]{vw1996}, we have
\(
	\mathbb{E}^\ast\sup_{D\in\mathcal{D}_{n,\delta_n}^\theta}|\sqrt{n}(\mathbb{P}_n-P_0)\log\tfrac{1-D}{1-D_\theta}|\lesssim J\bigl(1+\tfrac{J}{\delta_n^2\sqrt{n}}\bigr)
\)
for $J=J_{[]}(\rho,\{\log\frac{1-D}{1-D_\theta}:D\in\mathcal{D}_{n,\delta_n}^\theta\},\|\cdot\|_{P_0,B})$.
With a $\delta_n$\hyp{}bracket in $\mathcal{D}_{n,\delta_n}^\theta$, \cref{asm:2} implies
\[
	\bigl\|\log\tfrac{1-\ell}{1-D_\theta}-\log\tfrac{1-u}{1-D_\theta}\bigr\|_{P_0,B}^2\leq 4 P_0\Bigl(\sqrt{\tfrac{1-\ell}{1-u}}-1\Bigr)^2=O(\delta_n).
\]
Therefore, the expectation of the supremum is of order $O(\delta_n\sqrt{n})$.
\end{proof}

\section{Proof of Theorem \ref{thm:linear}}\label{sec:proof_thm:linear}

{
Let $h_n$ be a bounded sequence and denote $\theta_n\vcentcolon=\theta_0+\frac{h_n}{\sqrt{n}}$ and $W_n\vcentcolon=\sqrt{\hat{p}_{\theta_n}/\hat{p}_{\theta_0}}-1$.
Define $R$ by $\log(1+x)=x-\frac{1}{2}x^2+\frac{1}{2}x^2 R(x)$ for $R(x)=O(x)$. Then,
\[
	n\mathbb{P}_n\log\tfrac{\hat{p}_{\theta_n}}{\hat{p}_{\theta_0}}=2 n\mathbb{P}_n W_n-n\mathbb{P}_n W_n^2+n\mathbb{P}_n W_n^2 R(W_n).
\]
By \cref{asm:dqmp} (\ref{asm:proj}) and $P_{\theta_0}\dot{\ell}_{\theta_0}=0$,
\[
	2 n\mathbb{P}_n W_n-n\mathbb{P}_n W_n^2=
	2 n P_{\theta_0}W_n+\sqrt{n}\mathbb{P}_n h_n'\dot{\ell}_{\theta_0}-n P_{\theta_0}W_n^2+o_P(1).
\]
By \cref{asm:dqmp} (\ref{asm:dqmp1}), 
\[
	n P_{\theta_0}W_n^2=\tfrac{1}{4}P_{\theta_0}h_n'\dot{\ell}_{\theta_0}\dot{\ell}_{\theta_0}'h_n+o_P(1)
	=\tfrac{1}{4}h_n'I_{\theta_0}h_n+o_P(1).
\]
Also, since $P_{\theta_0}\dot{\ell}_{\theta_0}=0$,
\begin{align*}
	2 n P_{\theta_0}W_n&=2 n \hat{P}_{\theta_0}W_n+2 n(P_{\theta_0}-\hat{P}_{\theta_0})W_n\\
	&=-n\int\bigl({\textstyle\sqrt{\hat{p}_{\theta_n}}-\sqrt{\hat{p}_{\theta_0}}}\bigr)^2+n(c_{\theta_n}-c_{\theta_0})-\sqrt{n}\hat{P}_{\theta_0}h_n'\dot{\ell}_{\theta_0}\\
	&\hphantom{={}}+2 n\int\bigl(\sqrt{p_{\theta_0}}-{\textstyle\sqrt{\hat{p}_{\theta_0}}}\bigr)\bigl(\sqrt{p_{\theta_0}}+{\textstyle\sqrt{\hat{p}_{\theta_0}}}\bigr)\Bigl(W_n-\tfrac{h_n'\dot{\ell}_{\theta_0}}{2\sqrt{n}}\Bigr).
\end{align*}
By \cref{asm:dqmp} (\ref{asm:dqmp1}),
\(
	n\int(\sqrt{\hat{p}_{\theta_n}}-\sqrt{\hat{p}_{\theta_0}})^2
	=\tfrac{1}{4}\hat{P}_{\theta_0}h_n'\dot{\ell}_{\theta_0}\dot{\ell}_{\theta_0}'h_n+o_P(1)
	=\tfrac{1}{4}h_n'I_{\theta_0}h_n+o_P(1)
\).
By the Cauchy\hyp{}Schwarz inequality,
\begin{multline*}
	\biggl|\int\bigl(\sqrt{p_{\theta_0}}-{\textstyle\sqrt{\hat{p}_{\theta_0}}}\bigr)\bigl(\sqrt{p_{\theta_0}}+{\textstyle\sqrt{\hat{p}_{\theta_0}}}\bigr)\Bigl(W_n-\tfrac{h_n'\dot{\ell}_{\theta_0}}{2\sqrt{n}}\Bigr)\biggr|\\
	\leq\biggl[\int\bigl(\sqrt{p_{\theta_0}}-{\textstyle\sqrt{\hat{p}_{\theta_0}}}\bigr)^2\int\bigl(\sqrt{p_{\theta_0}}+{\textstyle\sqrt{\hat{p}_{\theta_0}}}\bigr)^2\Bigl(W_n-\tfrac{h_n'\dot{\ell}_{\theta_0}}{2\sqrt{n}}\Bigr)^2\biggr]^{1/2},
\end{multline*}
which is $O_P(\delta_n n^{-3/4})=o_P(n^{-1})$ under \cref{asm:dqmp} (\ref{asm:dqmp1}) and $\delta_n=o(n^{-1/4})$.

Since $|n\mathbb{P}_n W_n^2 R(W_n)|\leq|n\mathbb{P}_n W_n^2|\max_{1\leq i\leq n}|R(W_n(X_i))|$ and $n\mathbb{P}_n W_n^2$ ``converges'' to $n P_{\theta_0}W_n^2=O_P(1)$ by \cref{asm:dqmp} (\ref{asm:proj}), it remains to show that the maximum is $o_P(1)$.
Write $V_n\vcentcolon=W_n-\frac{h_n'\dot{\ell}_{\theta_0}}{2\sqrt{n}}$.
Then,
\[
	\max_i|W_n(X_i)|\leq\max_i\,\bigl|\tfrac{1}{2\sqrt{n}}h_n'\dot{\ell}_{\theta_0}(X_i)\bigr|+\max_i|V_n(X_i)|.
\]
By Markov's inequality,
\begin{align*}
	P\Bigl(\max_{1\leq i\leq n}\bigl|\tfrac{1}{\sqrt{n}}h_n'\dot{\ell}_{\theta_0}(X_i)\bigr|>\varepsilon\Bigr)
	&\leq n P\bigl(\bigl|\tfrac{1}{\sqrt{n}}h_n'\dot{\ell}_{\theta_0}(X_i)\bigr|>\varepsilon\bigr)\\
	&\leq\varepsilon^{-2}P_{\theta_0}((h_n'\dot{\ell}_{\theta_0})^2\mathbbm{1}\{(h_n'\dot{\ell}_{\theta_0})^2>n\varepsilon^2\}),
\end{align*}
which converges to zero as $n\to\infty$ for every $\varepsilon>0$.
Thus, $\max_i\bigl|\tfrac{1}{\sqrt{n}}h_n'\dot{\ell}_{\theta_0}(X_i)\bigr|$ converges to zero in probability.
Since \cref{asm:dqmp} (\ref{asm:proj}) and (\ref{asm:dqmp1}) imply that $n\mathbb{P}_n V_n^2=n P_{\theta_0}V_n^2+o_P(1)=o_P(1)$, we have $\max_i V_n^2(X_i)=o_P(1)$ and hence $\max_i|V_n(X_i)|=o_P(1)$.
Conclude that $\max_i|W_n(X_i)|$ converges to zero in probability and so does $\max_i|R(W_n(X_i))|$.
}

\section{Proof of Theorem \ref{thm:one}}\label{sec:proof_thm:one}
We will prove   Theorem \ref{thm:one} under   weaker assumptions. 
In particular, we slightly relax Assumption \ref{ass:u}  by considering the aggregate behavior of $u_{\theta}(X^{(n)})$ around $\theta_0$ with respect to the prior $\Pi_n(\cdot)$.
Instead, we   assume 
$$
P_{\theta_0}^{(n)}\left(I_n(\Pi_n,X^{(n)},\varepsilon_n)\leq  \e^{-\wt C_n n\varepsilon_n^2}\right)=o(1)
$$ 
where 
\begin{equation}\label{eq:int_I}
I_n({\Pi}_n,X^{(n)},\epsilon)=\int_{B_n(\theta_0,\epsilon)} \e^{u_\theta(X^{(n)})}\d\Pi_n(\theta)
\end{equation}
and, at the same time,  
$$
P_{\theta_0}^{(n)}\left[\sup_{  \Theta_n^c\cup d_n(\theta,\theta_0)>\epsilon}|u_\theta(X^{(n)})|> \wt C_nn\varepsilon_n^2\right]=o(1)
$$
for any $\epsilon>\varepsilon_n$. Assumption  \eqref{ass:prior} is not needed if one is only interested in the concentration inside $\Theta_n$.
Alternatively, we could also  replace Assumption \eqref{ass:u} with the following condition to lower-bound the denominator in \eqref{eq:post} 
$$
\sup_{\theta\in B_n(\theta_0,\varepsilon_n)}P_{\theta_0}^{(n)}\left[\ln (p_\theta^{(n)}/p_{\theta_0}^{(n)})+u_\theta<-n\varepsilon_n^2\right]=o(n\varepsilon_n^2).
$$
Instead of relying on the existence of exponential tests (through Lemma 9 in \cite{gvdv07}), we could  then directly assume that for any   $\epsilon>\varepsilon_n$ and for all $\theta\in\Theta_n$ such that $d(\theta,\theta_0)>j\epsilon$ for any $j\in\N$ there exists a test $\phi_n(\theta)$ satisfying
$$
P_{\theta_0}^{(n)}\phi_n  \lesssim \e^{-n\epsilon^2/2}\quad\text{and}\quad \int_{\mX}(1-\phi_n) p_\theta^{(n)}\e^{u_\theta}\leq \e^{-j^2n\epsilon^2/2}.
$$

We will use the following Lemma (an analogue of Lemma 10 \cite{gvdv07}).

\begin{lem}\label{lemma:event}
Recall the definition $I_n({\Pi}_n,X^{(n)},\epsilon)$ in \eqref{eq:int_I}  and define $q_{\theta}^{(n)}=p_\theta^{(n)}/p_{\theta_0}^{(n)}\e^{u_\theta}$. Then we have for any $C,\varepsilon>0$
$$
P_{\theta_0}^{(n)}\left(\int_{B(\theta_0,\varepsilon)} q_\theta^{(n)}\d\Pi_n(\theta)\leq   \e^{-(1+C)n\varepsilon^2} \times I_n({\Pi}_n,X^{(n)},\epsilon)  \right)\leq \frac{1}{C^{2} n\varepsilon^2}.
$$
\end{lem}
\begin{proof}
Define a changed prior measure $\Pi^\star_n(\cdot)$   through  $\d\Pi^\star_n(\theta)=\frac{\e^{u_\theta(X^{(n)})}}{\int \e^{u_\theta(X^{(n)})}\d\theta}\d\Pi_n(\theta)$.
 Lemma 10 of  \cite{gvdv07} then yields
\begin{align*}
&P_{\theta_0}^{(n)}\left(\int_{B(\theta_0,\varepsilon)} q_\theta^{(n)}\d\Pi_n(\theta)\leq \e^{-(1+C)n\varepsilon^2} I_n({\Pi}_n,X^{(n)},\epsilon)\right)\\
&=P_{\theta_0}^{(n)}\left(\int_{B(\theta_0,\varepsilon)} p_\theta^{(n)}/p_{\theta_0}^{(n)}\d\Pi^\star_n(\theta)\leq \Pi^\star_n(B(\theta_0,\varepsilon)) \e^{-(1+C)n\varepsilon^2}\right)\leq \frac{1}{C^{2} n\epsilon^2}.\quad\quad \qedhere
\end{align*}
\end{proof}

\medskip

Recall the definition $I_n(\Pi_n,X^{(n)},\varepsilon_n)=\int_{B(\theta_0,\varepsilon_n)}\e^{u_\theta(X^{(n)})}\d\Pi_n(\theta)$ and define an event
$$
\mA_n=\left\{X^{(n)}: \int_{B(\theta_0,\varepsilon_n)} q_\theta^{(n)}\d\Pi_n(\theta)> \e^{-2n\varepsilon_n^2}I_n(\Pi_n,X^{(n)},\varepsilon_n) \right\}
$$
where $q_{\theta}^{(n)}=p_\theta^{(n)}/p_{\theta_0}^{(n)}\e^{u_\theta}$. From our  assumptions, there exists a  sequence  $\wt C_n>0$ such that the complement of the set
$$
\mathcal B_n=\left\{X^{(n)}: I_n(\Pi_n,X^{(n)},\varepsilon_n) > \e^{-\wt C_nn\varepsilon_n^2}\,\,\text{and}\,\, 
\sup_{\Theta_n^c\cup d_n(\theta,\theta_0)>\varepsilon_n}|u_\theta(X^{(n)})|\leq \wt C_nn\varepsilon_n^2\right\}
$$
has a vanishing probability. Lemma \ref{lemma:event} then yields $P_{\theta_0}^{(n)}[\mA_n^c\cup \mathcal B_n^c]=o(1)\,\,\text{as $n\rightarrow\infty$}.$
The following calculations are  thus conditional on the set $\mA_n\cap\mathcal B_n$. On this set, we can lower-bound the denominator of \eqref{eq:post} as follows 
$$
\int_{\Theta}q_{\theta}^{(n)}\d\Pi_n(\theta)>\int_{B(\theta_0,\varepsilon_n)}q_{\theta}^{(n)}\d\Pi_n(\theta)>\e^{-2n\varepsilon_n^2} I_n(\Pi_n,X^{(n)},\varepsilon_n)\geq  \e^{-(2+\wt C_n)n\varepsilon_n^2}.
$$
We first show that
$P_{\theta_0}^{(n)}[\Pi_n^\star(\Theta\backslash\Theta_n\C X^{(n)})]=o(1)$ as $n\rightarrow\infty$.
On the set $\mA_n\cap\mathcal B_n$ we have from \eqref{ass:prior} and from the Fubini's theorem
\begin{align*}
 P_{\theta_0}^{(n)}\left[\Pi^\star_n(\Theta\backslash\Theta_n\C X^{(n)})\right]&=
 P_{\theta_0}^{(n)} \left[\frac{\int_{\Theta\backslash\Theta_n} 
q_{\theta}^{(n)}\d\Pi_n(\theta)}{\int_{\Theta}q_{\theta}^{(n)}\d\Pi_n(\theta)}\right]
\leq  \e^{2n\varepsilon_n^2}
 \frac{\Pi_n^\star(\Theta\backslash\Theta_n)}{\Pi_n^\star(B_n(\theta_0,\varepsilon_n))}\\
 &= \e^{2(1+\wt C_n)n\varepsilon_n^2} \frac{\Pi_n(\Theta\backslash\Theta_n)}{\Pi_n(B_n(\theta_0,\varepsilon_n))}=o(1).
\end{align*}

For some $J>0$ (to be determined later) we define the complement of the ball around the truth as a union of shells
$$
U_n=\{\theta\in\Theta_n: d_n(\theta,\theta_0)>  MJ\varepsilon_n\}=\bigcup_{j\geq J}\Theta_{n,j}
$$
where each shell equals
$$
\Theta_{n,j}=\{\theta\in\Theta_n:Mj\varepsilon_n<d_n(\theta,\theta_0)\leq M(j+1)\varepsilon_n\}.
$$
We now invoke the local entropy Assumption (3.2) in \cite{gvdv07} which guarantees  (according to Lemma 9 in \cite{gvdv07})  that 
 there exist tests $\phi_n$ (for each $n$) such that 
 \begin{equation}\label{eq:tests}
P_{\theta_0}^{(n)}\phi_n\lesssim \e^{n\varepsilon_n^2-nM^2\varepsilon_n/2}\quad\text{and}\quad P_{\theta}^{(n)}(1-\phi_n)\leq \e^{-nM^2\varepsilon_n^2j^2/2}
\end{equation}
for all $\theta\in\Theta_n$ such that $d_n(\theta,\theta_0)>M\varepsilon_nj$ and for every $j\in\N\backslash\{0\}$ and $M>0$. 
One can then write
\begin{align*}
P_{\theta_0}^{(n)}\Pi\left(\theta\in\Theta:d(\theta,\theta_0)>MJ\varepsilon_n\C X^{(n)}\right)&\leq  P_{\theta_0}^{(n)}\Pi(\Theta_n^c\C X^{(n)})+P_{\theta_0}^{(n)}\phi_n+P_{\theta_0}^{(n)}(\mA_n^c)+P_{\theta_0}^{(n)}(\mathcal B_n^c)\\
&+\sum_{j\geq J}P_{\theta_0}^{(n)}[
\Pi(\Theta_{n,j}\C X^{(n)})(1-\phi_n)\mathbb I(\mA_n\cap \mathcal B_n)]
\end{align*}
For the last term above, we  recall that
$\Pi(\Theta_{n,j}\C X^{(n)})=\frac{\int_{\Theta_{n,j}} q_{\theta}^{(n)}\d\Pi_n(\theta)}{\int_{\Theta}q_{\theta}^{(n)}\d\Pi_n(\theta)}.$ We bound the denominator as before.
Regarding the numerator,  on the event $\mathcal B_n$ we have from \eqref{eq:tests} and from the Fubini's theorem
\begin{align}\label{eq:second_type}
P_{\theta_0}^{(n)}\int_{\Theta_{n,j}}q_{\theta}^{(n)}\d\Pi_n(\theta)(1-\phi_n)&\leq  \e^{-nM^2\varepsilon_n^2j^2/2+\wt C_nn\varepsilon_n^2}\Pi_n(\Theta_{n,j})
\end{align}
Putting the pieces together, we obtain
\begin{align*}
P_{\theta_0}^{(n)}[\Pi(\Theta_{n,j}\C X^{(n)})(1-\phi_n)\mathrm I(\mA_n\cap \mathcal B_n)]&\leq 
\e^{-nM^2\varepsilon_n^2j^2/2+2(1+\wt C_n)n\varepsilon_n^2}\frac{\Pi_n(\Theta_{n,j})}{\Pi_n[B_n(\theta_0,\varepsilon_n)]}.
\end{align*}
Assumption (3.4) of \cite{gvdv07} writes as 
\begin{equation}\label{ass:prior2}
\frac{\Pi_n(\Theta_{n,j})}{\Pi_n[B_n(\theta_0,\varepsilon_n)]}\leq \e^{nM^2\varepsilon_n^2j^2/4}
\end{equation}
which yields
$$
P_{\theta_0}^{(n)}\Pi\left(\theta\in\Theta:d(\theta,\theta_0)>MJ\varepsilon_n\C X^{(n)}\right)\leq o(1)+\sum_{j\geq J}\e^{-n\varepsilon_n^2(M^2j^2/4-2-2\wt C_n)}.
$$
The right hand side converges to zero as long as $J=J_n\rightarrow\infty$ fast enough so that $\wt C_n=o(J_n)$ and $n\varepsilon_n^2$ is bounded away from zero. \qedhere

\section{Posterior Concentration Rate: Misspecification Lens}\label{sec:proof_thm_misspec}

The following Theorem \ref{thm:misspec} quantifies concentration  in terms of a  KL neighborhoods around $\wt P_{\theta^*}^{(n)}$ defined  as 
$B(\epsilon,\wt P_{\theta^*}^{(n)},P_{\theta_0}^{(n)})=\left\{\wt P_\theta^{(n)}\in\wt \mP^{(n)}: K(\theta^*,\theta_0) \leq n\epsilon^2,V(\theta^*,\theta_0)\leq n\epsilon^2\right\},
$
where $K(\theta^*,\theta_0)\equiv P_{\theta_0}^{(n)}\log \frac{\wt p_{\theta^*}^{(n)}}{\wt p_\theta^{(n)}}
$ and $V(\theta^*,\theta_0)=P_{\theta_0}^{(n)}\left|\log \frac{\wt p_{\theta^*}^{(n)}}{\wt p_\theta^{(n)}}- K(\theta^*,\theta_0)\right|^2$.

\begin{thm}\label{thm:misspec}
 Denote with $Q_\theta^{(n)}$ a measure defined through $\d Q_\theta^{(n)}=\frac{p_{\theta_0}^{(n)}}{\wt p_{\theta^*}^{(n)}}\d P_\theta^{(n)}$ and let  $d(\cdot,\cdot)$ be a semi-metric on $\mathcal P^{(n)}$. Suppose that there exists a sequence $\varepsilon_n>0$ satisfying  $\varepsilon_n\rightarrow0$ and $n\varepsilon_n^2\rightarrow\infty$  such that for every $\epsilon>\varepsilon_n$ there exists a test $\phi_n$ (depending on $\epsilon$) such that for every $J\in\mathbb N_0$
\begin{align}
P_{\theta_0}^{(n)}\phi_n\lesssim \e^{-n\epsilon^2/4}\quad\text{and}\quad \sup\limits_{\wt P_\theta^{(n)}: d(\wt P_\theta^{(n)},\wt P_{\theta^*}^{(n)})>J\epsilon} Q_\theta^{(n)}(1-\phi_n)\leq \e^{-nJ^2\epsilon^2/4}\label{test2}.
\end{align}
Let  $B(\epsilon,\wt P_{\theta^*}^{(n)},P_{\theta_0}^{(n)})$ be as before and let $\wt\Pi_n(\theta)$ be a prior distribution with a density $\wt\pi(\theta)\propto C_\theta\pi(\theta)$. Assume that there exists a constant $L>0$ such that,  for all $n$ and $j\in\mathbb N$, 
\begin{align}\label{eq:prior_mass_misspec}
\frac{\wt\Pi_n\left( \theta\in\Theta: j\varepsilon_n<d(\wt P_{\theta}^{(n)},\wt P_{\theta^*}^{(n)})\leq (j+1)\varepsilon_n \right)}{\wt\Pi_n\left( B(\epsilon,\wt P_{\theta^*}^{(n)},P_{\theta_0}^{(n)}) \right)}&\leq \e^{n\varepsilon_n^2j^2/8}.
\end{align}
Then for every sufficiently large constant $M$, as $n\rightarrow\infty$,
\begin{equation}\label{eq:statement_misspec}
P_{\theta_0}^{(n)}\Pi_n^\star\left(\wt P_\theta^{(n)}: d(\wt P_\theta^{(n)}, \wt P_{\theta^*}^{(n)})\geq M\varepsilon_n\C X^{(n)}\right)\rightarrow 0.
\end{equation}
\end{thm}
\proof

We define the event 
$$
\mA=\left\{X^{(n)}\in\mathcal X: \int\frac{\wt p_\theta^{(n)}}{\wt p_{\theta^*}^{(n)}}\d\wt\Pi_n(\theta)> \e^{-(1+C)n\epsilon^2}\wt\Pi_n[B(\epsilon,\wt P_{\theta^*}^{(n)},P_{\theta_0}^{(n)})] \right\}.
$$ 
The following lemma shows that $P_{\theta_0}^{(n)}[\mA^c]=o(1)$ as $n\rightarrow\infty$.
\begin{lem}
For $k\geq 2$, every $\epsilon>0$ and a prior measure $\wt\Pi_n(\theta)$ on $\Theta$, we have for every $C>0$
$$
P_{\theta_0}^{(n)} \left(\int\frac{\wt p_\theta^{(n)}}{\wt p_{\theta^*}^{(n)}}\d\wt\Pi_n(\theta)\leq \e^{-(1+C)n\epsilon^2}\wt\Pi_n[B(\epsilon,\wt P_{\theta^*}^{(n)},P_{\theta_0}^{(n)})]\right)\leq 
\frac{1}{C^2n\epsilon^2}.
$$
\end{lem}
\proof This follows directly from  Lemma 10 in \cite{gvdv07}.

\medskip
We now define $U_n(\epsilon)=\Pi_n(\theta\in\Theta: d(\wt P_\theta^{(n)},\wt P_{\theta^*}^{(n)})>\epsilon\C X^{(n)})$. For every $n\geq 1$ and $J\in\mathbb N\backslash\{0\}$, we can decompose
\begin{align*}
P_{\theta_0}^{(n)} U_n(JM\varepsilon_n)=&P_{\theta_0}^{(n)}[U_n(JM\varepsilon_n)\phi_n]+P_{\theta_0}^{(n)}[U_n(JM\varepsilon_n)(1-\phi_n)\mathbb I(\mA^c)]\\
&+
P_{\theta_0}^{(n)}[U_n(JM\varepsilon_n)(1-\phi_n)\mathbb I(\mA)].
\end{align*}
The first term is bounded (from the assumption \eqref{test2}) as
$$
P_{\theta_0}^{(n)}[U_n(JM\varepsilon_n)\phi_n]\leq P_{\theta_0}^{(n)}\phi_n\lesssim\e^{-n\varepsilon^2_nJ^2M^2}.
$$
The second term can be bounded by $P_{\theta_0}^{(n)}[\mathbb I(\mA^c)]\leq \frac{1}{C^2J^2M^2n\varepsilon_n^2}$ which converges to zero as $n\varepsilon_n^2\rightarrow\infty$.
The last term satisfies
\begin{align*}
&P_{\theta_0}^{(n)}[U_n(JM\varepsilon_n)(1-\phi_n)\mathbb I(\mA)]=P_{\theta_0}^{(n)}\left[(1-\phi_n)\mathbb I(\mA)
\frac{\int_{\theta: d(\wt P_\theta^{(n)},\wt P_{\theta^*}^{(n)})>JM\varepsilon_n} \frac{\wt p_\theta^{(n)}}{\wt p_{\theta^*}^{(n)}}\wt\Pi_n(\theta)\d\theta }{\int_\Theta \frac{\wt p_\theta^{(n)}}{\wt p_{\theta^*}^{(n)}}\wt\Pi_n(\theta)\d\theta }\right]\\
&\qquad\leq  \frac{\e^{(1+C)n\epsilon^2}}{\wt\Pi_n[B(\epsilon,\wt P_{\theta^*}^{(n)},P_{\theta_0}^{(n)})]}
{\int_{\theta: d(\wt P_\theta^{(n)},\wt P_{\theta^*}^{(n)})>JM\varepsilon_n}\left[\int_{\mX}(1-\phi_n)p_{\theta_0}^{(n)}\frac{\wt p_\theta^{(n)}}{\wt p_{\theta^*}^{(n)}}\right]\wt\Pi_n(\theta)\d\theta}\\
&\qquad\leq  \frac{\e^{(1+C)n\epsilon^2}}{\wt\Pi_n[B(\epsilon,\wt P_{\theta^*}^{(n)},P_{\theta_0}^{(n)})]}\sum_{j\geq J}\int_{U_{n,j}} Q_\theta^{(n)}(1-\phi_n) \d\wt\Pi_n(\theta),
\end{align*}
where $U_{n,j}=\{\theta: jM\varepsilon_n<d(\wt P_\theta^{(n)},\wt P_{\theta^*}^{(n)})\leq (j+1)M\varepsilon_n)\}$. The tests (from the assumption \eqref{test2}) satisfy 
$Q_\theta^{(n)}(1-\phi_n)\leq \e^{-nj^2M^2\varepsilon_n^2/4}$ uniformly on $U_{n,j}$. Then we find (using the assumption \eqref{eq:prior_mass_misspec})
$$
P_{\theta_0}^{(n)}[U_n(JM\varepsilon_n)(1-\phi_n)\mathbb I(\mA)]\leq \e^{(1+C)n\varepsilon_n^2}\sum_{j\geq J} 
\e^{-nj^2M^2\varepsilon_n^2/4+ nj^2M^2\varepsilon_n^2/8}. 
$$
The sum converges to zero when $n\varepsilon_n^2$ is bounded away from zero and $J\rightarrow\infty$. \qedhere

\begin{rem}
For iid data, \cite{kleijn1} introduce a condition involving entropy numbers under misspecification which implies the existence of exponential tests for a testing problem that involves non-probability measures. Since we have a non-iid situation, we assumed the existence of tests directly.
\end{rem}
 \begin{rem}(Friendlier Metrics)\label{rem:parametric_misspec}
In parametric models indexed by $\theta$ in a metric space $(\Theta,d)$, it is more natural to characterize the posterior concentration in terms of $d(\cdot,\cdot)$ rather than the Kullback-Leibler divergence\footnote{Hellinger neighborhoods are less appropriate for misspecified models}. 
Section 5 of \cite{kleijn1} clarifies how Theorem \ref{thm:misspec} can be reformulated in terms of some metric $d(\cdot,\cdot)$ on $\Theta$.

 \end{rem}

\section{Normal Location-Scale Example}\label{sec:example_sup}

Let $X_i\sim P_0=N(0,1)$ and $P_\theta=N(\mu,\sigma^2)$ where $\theta=(\mu,\sigma^2)$ are the unknown parameters and $\theta_0=(0,1)$ are the true values.
This model satisfies \cref{asm:dqm} with the score
\(
	\dot{\ell}_{\theta_0}(x)=\begin{bsmallmatrix}x\\(x^2-1)/2\end{bsmallmatrix}
\)
and the Fisher information matrix
\(
	I_{\theta_0}=\begin{bsmallmatrix}1&0\\0&1/2\end{bsmallmatrix}
\).
The oracle discriminator of $P_0$ from $P_\theta$ is
\(
	D_\theta(x)=\bigl[1+\exp\bigl(-\frac{1}{2}\log\sigma^2+\frac{x^2}{2}-\frac{(x-\mu)^2}{2\sigma^2}\bigr)\bigr]^{-1}
\).
Let us use the logistic regression using regressors $(1,x,x^2)$ to estimate $D_\theta$, i.e., 
$$
D_\theta(x)=[1+\exp(-\beta_0-\beta_1 x-\beta_2 x^2)]^{-1}.
$$
Thus, the true parameter for the logistic regression is $\beta=(\beta_0,\beta_1,\beta_2)=\bigl(\frac{1}{2}\log\sigma^2+\frac{\mu^2}{2\sigma^2},-\frac{\mu}{\sigma^2},\frac{1}{2\sigma^2}-\frac{1}{2}\bigr)$.
Let $\hat{\beta}=(\hat{\beta}_0,\hat{\beta}_1,\hat{\beta}_2)$ be the estimator of $\beta$.
Then,
\[
	\hat{p}_\theta(x)=\frac{\exp\bigl(-\frac{x^2}{2}-\hat{\beta}_0-\hat{\beta}_1 x-\hat{\beta}_2 x^2\bigr)}{\sqrt{2\pi}}\qquad\text{and}\quad
	c_\theta=\frac{\exp\bigl(-\hat{\beta}_0+\frac{1}{2}\frac{\hat{\beta}_1^2}{1+2\hat{\beta}_2}\bigr)}{\sqrt{1+2\hat{\beta}_2}}.
\]
Being a MLE, $\hat{\beta}$ is regular and efficient, so $\sqrt{n}(\hat{\beta}-\beta)=\Delta+o_P(1)$ for a normal vector $\Delta$.
Moreover, if we generate $X_i^\theta$ through $X_i^\theta=\mu+\sigma\wt{X}_i$, $\wt{X}_i\sim N(0,1)$, there is one\hyp{}to\hyp{}one correspondence between $X_i^{\theta_1}$ and $X_i^{\theta_2}$ for every $\theta_1$ and $\theta_2$, so the dependence of $\Delta$ on $\theta$ disappears as $n\to\infty$ for otherwise a more efficient estimator exists to contradict efficiency.
Therefore, the formula for $\hat{p}_\theta$ implies that \cref{asm:dqmp} (\ref{asm:dqmp1}) is satisfied with the oracle score function $\dot{\ell}_{\theta_0}$; since $\hat{p}_\theta$ is twice differentiable, it holds with a faster rate of $O_P(\|h\|^4)$.
{
Meanwhile, if inflated with $\sqrt{n}$, the dependence of $\Delta$ on $\theta$ may not be ignorable.
Simulation suggests that this dependence is linear and of order $O(n^{-1/2})$, so write $\sqrt{n}(\hat{\beta}-\beta)=\Delta+n^{-1/2}\dot{\Delta}(\theta-\theta_0)+o_P(n^{-1/2})$ for some $\dot{\Delta}$ independent of $\theta$.
Considering $c_\theta$ as a function of $\hat{\beta}$ and $\beta$ as a function of $\theta$, Taylor's theorem implies
\begin{align*}
	n(c_\theta-c_{\theta_0})&=\sqrt{n}\tfrac{\partial c_\theta}{\partial\beta'}\sqrt{n}(\hat{\beta}_\theta-\hat{\beta}_{\theta_0})
	+\tfrac{1}{2}\sqrt{n}(\hat{\beta}_\theta-\beta_{\theta_0})'\tfrac{\partial^2 c_\theta}{\partial\beta\partial\beta'}\sqrt{n}(\hat{\beta}_\theta-\beta_{\theta_0})\\
	&\hspace{10pt}-\tfrac{1}{2}\sqrt{n}(\hat{\beta}_{\theta_0}-\beta_{\theta_0})'\tfrac{\partial^2 c_\theta}{\partial\beta\partial\beta'}\sqrt{n}(\hat{\beta}_{\theta_0}-\beta_{\theta_0})+o_P(1),\\
	\sqrt{n}(\hat{\beta}_\theta-\hat{\beta}_{\theta_0})&=\tfrac{\partial\beta}{\partial\theta'}\sqrt{n}(\theta-\theta_0)
	+\tfrac{1}{2}(\mu-\mu_0)\tfrac{\partial^2\beta}{\partial\mu\partial\theta'}\sqrt{n}(\theta-\theta_0)\\
	&\hspace{10pt}+\tfrac{1}{2}(\sigma^2-\sigma_0^2)\tfrac{\partial^2\beta}{\partial\sigma^2\partial\theta'}\sqrt{n}(\theta-\theta_0)+\tfrac{\dot{\Delta}}{\sqrt{n}}(\theta-\theta_0)+o_P(n^{-1/2}).
\end{align*}
At $\theta=\theta_0$,
\[
	\tfrac{\partial c_\theta}{\partial\beta}=\begin{bsmallmatrix}-1\\0\\-1\end{bsmallmatrix},
	\tfrac{\partial^2 c_\theta}{\partial\beta\partial\beta'}=\begin{bsmallmatrix}1&0&1\\0&1&0\\1&0&3\end{bsmallmatrix},
	\tfrac{\partial\beta}{\partial\theta'}=\begin{bsmallmatrix}0&\frac{1}{2}\\-1&0\\0&-\frac{1}{2}\end{bsmallmatrix},
	\tfrac{\partial^2\beta}{\partial\mu\partial\theta'}=\begin{bsmallmatrix}1&0\\0&1\\0&0\end{bsmallmatrix},
	\tfrac{\partial^2\beta}{\partial\sigma^2\partial\theta'}=\begin{bsmallmatrix}0&-\frac{1}{2}\\1&0\\0&1\end{bsmallmatrix}.
\]
Substituting these, we can derive that
\[
	n(c_\theta-c_{\theta_0})=\sqrt{n}(\theta-\theta_0)'\biggl(\begin{bsmallmatrix}0&-1&0\\0&0&-1\end{bsmallmatrix}\Delta
	+\dot{\Delta}'\begin{bsmallmatrix}-1\\0\\-1\end{bsmallmatrix}\biggr)+o_P(1),
\]
yielding \cref{asm:dqmp} (\ref{asm:const}).
Finally, \Cref{fig:exa:normal:asm} illustrates \cref{asm:dqmp} (\ref{asm:proj}) and (\ref{asm:const}).
The black lines plot $n(c_\theta-c_{\theta_0})$ as we change $\theta$; they are linear and its quadratic curvatures are ignorable.
The blue lines represent $n(\mathbb{P}_n-P_{\theta_0})\bigl(\sqrt{\hat{p}_\theta/\hat{p}_{\theta_0}}-1-(\theta-\theta_0)'\dot{\ell}_{\theta_0}/2\bigr)$ and the red lines $n(\mathbb{P}_n-P_{\theta_0})\bigl(\sqrt{\hat{p}_\theta/\hat{p}_{\theta_0}}-1\bigr)^2$; compared to the values of $n(c_\theta-c_{\theta_0})$, both are uniformly ignorable.
}
\begin{figure}[!t]
\centering
\begin{subfigure}[b]{0.45\textwidth}
\centering
\scalebox{0.25}{\includegraphics{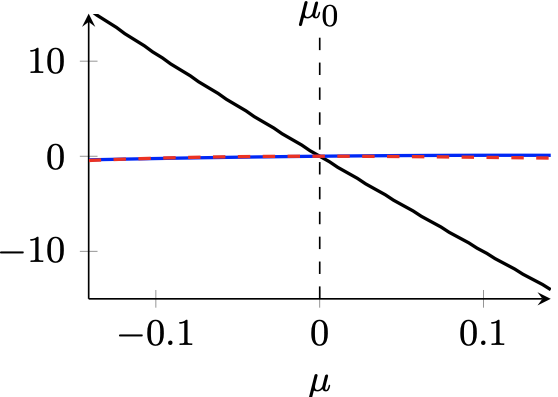}}
\caption{The black line $n(c_\theta-c_{\theta_0})$; the blue line $n(\mathbb{P}_n-P_{\theta_0})(\sqrt{\hat{p}_\theta/\hat{p}_{\theta_0}}-1-(\theta-\theta_0)'\dot{\ell}_{\theta_0}/2)$; the red line $n(\mathbb{P}_n-P_{\theta_0})(\sqrt{\hat{p}_\theta/\hat{p}_{\theta_0}}-1)^2$. $\sigma^2$ is fixed at $\sigma_0^2$.}
\label{fig:exa:normal:asm:mu}
\end{subfigure}
\qquad
\begin{subfigure}[b]{0.45\textwidth}
\centering
\scalebox{0.25}{\includegraphics{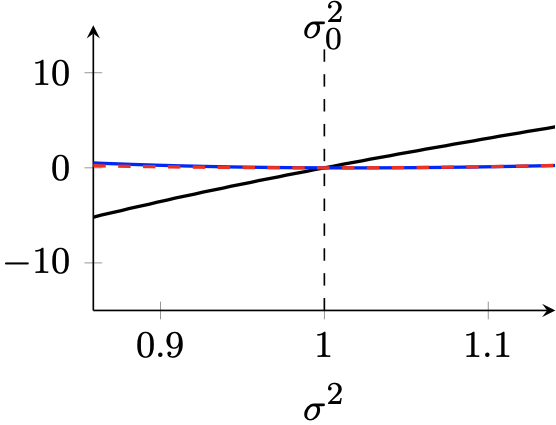}}
\caption{The black line $n(c_\theta-c_{\theta_0})$; the blue line $n(\mathbb{P}_n-P_{\theta_0})(\sqrt{\hat{p}_\theta/\hat{p}_{\theta_0}}-1-(\theta-\theta_0)'\dot{\ell}_{\theta_0}/2)$; the red line $n(\mathbb{P}_n-P_{\theta_0})(\sqrt{\hat{p}_\theta/\hat{p}_{\theta_0}}-1)^2$. $\mu$ is fixed at $\mu_0$.}
\label{fig:exa:normal:asm:sig2}
\end{subfigure}
\caption{Illustration of \cref{asm:dqmp} (\ref{asm:proj}\textendash\ref{asm:const}) in the normal location-scale example with $n=m=5000$.}
\label{fig:exa:normal:asm}
\end{figure}




%
%
Since this model with the logistic classifier satisfies \cref{asm:dqm,asm:dqmp}, it is susceptible to \cref{thm:linear}. This is supported by a diagnostics plot in \Cref{fig:exa:normal:quad} which portrays true and estimated likelihood ratios.
In \Cref{fig:exa:normal:quad:mu}, $\mu$ is varied with $\sigma^2$ fixed at $\sigma_0^2$ while, in \Cref{fig:exa:normal:quad:sig2}, $\sigma^2$ is varied with $\mu$ held at $\mu_0$.
The difference between the estimated log likelihood (blue) and the quadratic approximation (dashed red) is negligible, demonstrating that the validity of \cref{thm:linear} is justifiable.
Compared to the oracle log likelihood (black), the estimated log likelihood is shifted by the random term $\sqrt{n}(\dot{c}_{n,\theta_0}-\hat{P}_{\theta_0}\dot{\ell}_{\theta_0})$.
The curvature, however, is the same as oracle since the red line curves by the Fisher information $I_{\theta_0}$.
Thus,  we expect Algorithm 1 to produce a biased sample and Algorithm 2 a dispersed sample.
%
Note that we can compute
\(
	\sqrt{n}\hat{P}_{\theta_0}\dot{\ell}_{\theta_0}=c_{\theta_0}\sqrt{n}\bigl[-\frac{\hat{\beta}_1}{1+2\hat{\beta}_2},-\frac{1}{2}+\frac{1}{2(1+2\hat{\beta}_2)}+\frac{\hat{\beta}_1^2}{2(1+2\hat{\beta}_2)^2}\bigr]'
\),
which is asymptotically linear in $\Delta$ by the delta method. 
It is then reasonable to expect that this term has mean zero when averaged over $\wt{X}$ since $\hat{\beta}$ is asymptotically unbiased.
If $\dot{c}_{n,\theta_0}$ also has mean zero, then Algorithm 2 is unbiased and Algorithm 3 recovers the exact normal posterior.

\begin{figure}[t!]
\centering
\begin{subfigure}[b]{0.45\textwidth}
\centering
\scalebox{0.25}{\includegraphics{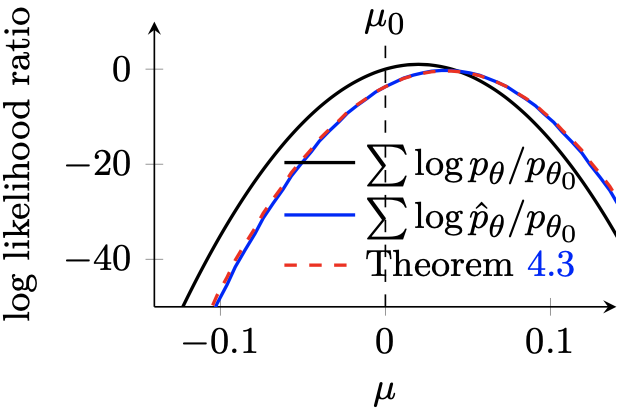}}
\caption{True log likelihood, estimated log likelihood, and quadratic approximation by \cref{thm:linear}. $\sigma^2=\sigma_0^2$.}
\label{fig:exa:normal:quad:mu}
\end{subfigure}
\qquad
\begin{subfigure}[b]{0.45\textwidth}
\centering
\scalebox{0.25}{\includegraphics{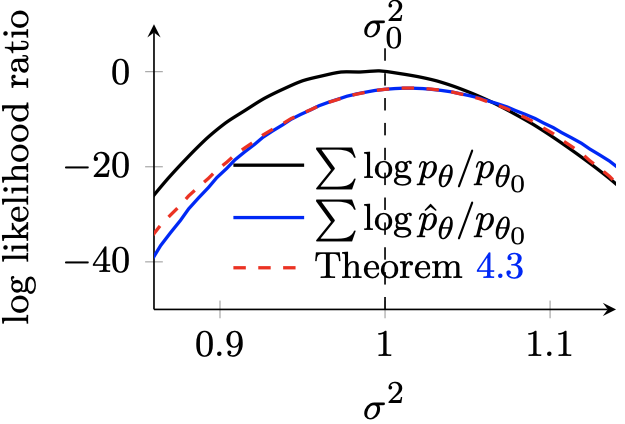}}
\caption{True log likelihood, estimated log likelihood, and quadratic approximation by \cref{thm:linear}. $\mu=\mu_0$.}
\label{fig:exa:normal:quad:sig2}
\end{subfigure}
\caption{Illustration of \cref{thm:linear} in the normal mean-scale example with  $n=m=5000$.}
\label{fig:exa:normal:quad}
\end{figure}

\begin{figure}[!t]
\centering

\begin{subfigure}{0.45\textwidth}
\centering
\scalebox{0.3}{\includegraphics{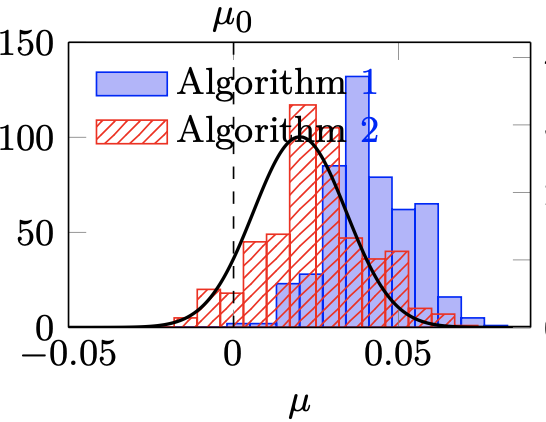}}
\caption{ Algorithm 1 and 2.}
\label{fig:normal:mu:hist2}
\end{subfigure}
\begin{subfigure}{0.45\textwidth}
\scalebox{0.3}{\includegraphics{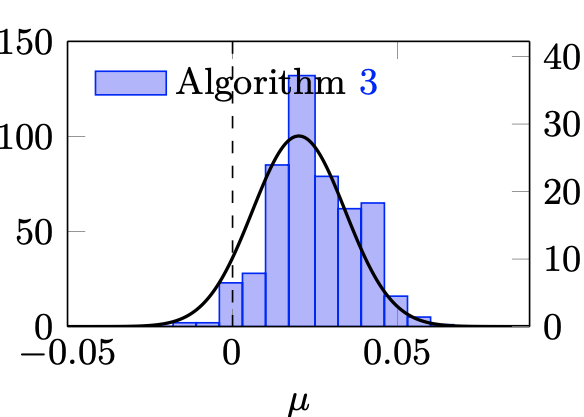}}
\caption{ Algorithm 3.}
\label{fig:normal:mu:hist3}
\end{subfigure}
\begin{subfigure}{0.45\textwidth}
\centering
\scalebox{0.3}{\includegraphics{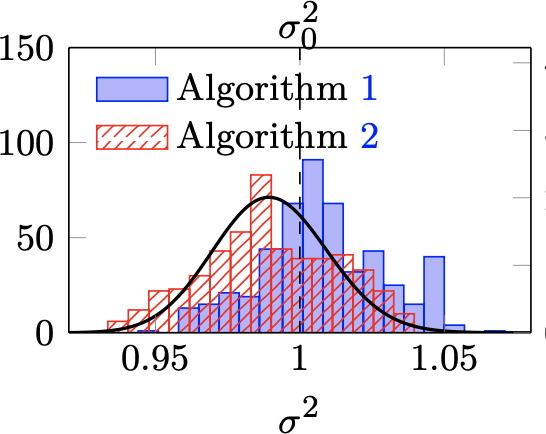}}
\caption{ Algorithm 1 and 2.}
\label{fig:normal:sig2:hist2}
\end{subfigure}
\begin{subfigure}{0.45\textwidth}
\scalebox{0.3}{\includegraphics{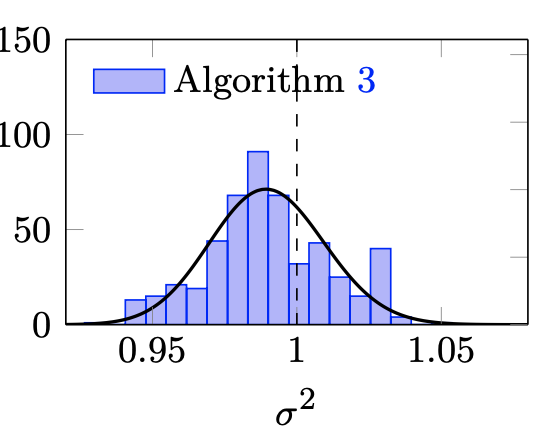}}
\caption{ Algorithm 3.}
\label{fig:normal:sig2:hist3}
\end{subfigure}
\caption{\small Histograms of the MHC samples of $\mu$ and $\sigma^2$ in the normal location\hyp{}scale model. Algorithm 1 (resp. 2) yield more biased (resp. dispersed) samples compared to the true posterior (black curve).  Algorithm 3 (on the right) tracks the black curve more closely.}
\label{fig:normal:hist}
\end{figure}

To see that this is indeed the case, we impose a conjugate normal\hyp{}inverse\hyp{}gamma prior, $\theta\sim N\Gamma^{-1}(\mu_0,\nu,\alpha,\beta)$, that is, the marginal prior of $\sigma^2$ is the inverse\hyp{}gamma $\Gamma^{-1}(\alpha,\beta)$ and the conditional prior of $\mu$ given $\sigma^2$ is $N(\mu_0,\frac{\sigma^2}{\nu})$.
The posterior is then analytically calculated as (for $\bar{X}_n =\frac{1}{n}\sum_i X_i$)
$$
\theta\mid X\sim N\Gamma^{-1}\bigl(\frac{\nu\mu_0+n\bar{X}_n}{\nu+n},\nu+n,\alpha+\frac{n}{2},\beta+\frac{1}{2}\sum_i(X_i-\bar{X}_n)^2+\frac{n\nu}{\nu+n}\frac{(\bar{X}_n-\mu_0)^2}{2}\bigr).
$$
Figure \ref{fig:normal:hist} shows the histograms of Algorithm 1, 2 and  3 after $K=500$ MCMC steps.
Since the estimated log likelihood has a rightward bias (as seen from \Cref{fig:exa:normal:quad}), Algorithm 1 produces a sample that is shifted to the right (Figures \ref{fig:normal:mu:hist2} and \ref{fig:normal:sig2:hist2}).
Algorithm 2, on the other hand, gives a sample that is more dispersed than the posterior but is correctly placed, indicating that the random bias has mean zero.
Consequently, Algorithm 3 generates a sample that is placed and shaped correctly (\Cref{fig:normal:mu:hist3,fig:normal:sig2:hist3}).

\section{Mixing Properties of MHC}\label{sec:speed}
A critical issue for MCMC algorithms is the determination of the number of iterations needed for the result  to be approximately a sample from the distribution of interest. 
This section sheds light on the mixing rate of Algorithm 1.
Under standard assumptions on   $q(\cdot\C \cdot)$ (such as positivity almost surely, see Corollary 4.1 in \cite{tha2017}), the distribution of the MHC Markov chain after $t$ steps  will converge to $\pi^\star_n(\theta\C\Xn)$ from any initialization in $\Theta$ in total variation as $t\rightarrow\infty$. \cite{mengersen} derive necessary and sufficient conditions for the Metropolis algorithms (with independent or symmetric candidate distributions) to converge at a geometric rate to a prescribed continuous distribution. 
\cite{bc2009} studied the speed of convergence of MH  when {both} $n\rightarrow\infty$ {\em and} $d\rightarrow\infty$ where $\theta\in\Theta\subset\R^d$.

We can reformulate their sufficient conditions for showing polynomial mixing times of MHC. Recall that the stationary distribution $\pi^\star_n(\theta\C\Xn)$ of the MHC sampler in \eqref{eq:posterior_density}
normalized to a compact set $K\subset\Theta$, writes as
$
\Pi^\star_K(B)=\int_B \pi^\star_n(\theta\C\Xn)/\int_K\pi^\star_n(\theta\C\Xn).
$
We are interested in bounding the number of steps needed to draw a random variable from $\Pi^*_K$ with a given precision. We denote with $\Pi_K^{*t}$ the distribution obtained after $t$ steps of the MHC algorithm starting from $\Pi_K^{*0}$. It is known  (see e.g. \cite{lovasz}) that the total variation distance between $Q$ and $Q_t$ can be bounded by
$
\|\Pi^*_K-\Pi^{*t}_K\|_{TV}\leq \sqrt{M}(1-\phi^2/2)^t,
$
where $M$ is a constant which depends on the initial distribution $\Pi_K^{*0}$ and $\phi$ is the {\em conductance} of the Markov chain defined, e.g., in (3.13) in \cite{bc2009}.
To obtain bounds on the conductance, the Markov chain needs to transition somewhat smoothly (see assumption D1 and D2 in \cite{bc2009}). These assumptions pertain to the continuity of the transitioning measure  and are satisfied by the Gaussian random walk with a suitable choice of the proposal variance (see Section 3.2.4 in \cite{bc2009})
The following Lemma summarizes Theorem 2 of \cite{bc2009} in the  context of Algorithm 1 under asymptotic normality assumptions examined in more detail in Section \ref{sec:bvm}.
\begin{lem}(Mixing Rate)\label{lemma:mix}
Under conditions in equations \eqref{eq:loglike}-\eqref{eq:conditions_eps} and a Gaussian random walk $q(\cdot\C\cdot)$  satisfying Lemma 4 of \cite{bc2009},
the global conductance $\phi$ of the Markov chain obtained from Algorithm 1 satisfies
$
1/\phi=\mathcal O(d)$ in  $P_{\theta_0}^{(n)}$-probability. In addition,  the minimal number of MCMC iterations needed to achieve $\|\Pi^*_K-\Pi^{*t}_K\|_{TV}<\epsilon$ is  $\mathcal O(d^2\log (M/\epsilon))$ for some suitable constant $M$ depending on the initial distribution $\Pi^{*0}_K$.
\end{lem}
MHC thus attains  bounds on the mixing rate that are  {\em polynomial} in $d$ (i.e. rapid mixing) under suitable Bernstein-von Mises  conditions formalized later in Section \ref{sec:bvm}.
This section investigates how fast the Markov chain converges to its target $\pi^\star_n(\theta\C\Xn)$ as the number of iterations $t$ grows. 
In Section \ref{sec:EB} (resp. Section \ref{sec:misspec}),  we investigate a fundamentally different question. We assess the speed at which the target $\pi^\star_n(\theta\C\Xn)$ shrinks around the truth $\theta_0$ (resp. a Kullback-Leibler projection) as $n$ grows.

{
 The multiplication constant $M$ in  Lemma \ref{lemma:mix} depends on the initial distribution. Namely,  the initial distribution needs to be ``$M$-warm" according to assumption (3.5) in \cite{bc2009}. Loosely speaking, $M$ quantifies the amount of overlap between the initial and stationary distributions. Our convergence rate result thereby implicitly incorporates the properties of the initialization algorithm by regarding the  constant $M$ as dependent on the initialization routine.  In our Lotka-Volterra example, we found that the mixing  mixing performance of MHC depends on the classifier. With  random forests, the initialization was not as important since the shape of the likelihood approximation did not have a sharp peak (compare  Figure 5 and 6 in the main manuscript). On the other hand,   \texttt{glmnet} yields likelihood approximations with only a very narrow area of likelihood support and the initialization needed to be close in order to avoid a very long burn-in.  We have considered an ABC pilot run for initialization. Alternatively, one could try less accurate/costly classifiers in a pilot run to obtain a good initialization. 
}

\section{Bernstein-von Mises Theorem}\label{sec:bvm}
The Bernstein-von Mises (BvM)  theorem asserts that the posterior distribution of a parameter  in a suitably regular finite-dimensional model is approximately  normally distributed as the number of observations grows to infinity.  
More precisely, {if $p_\theta$ is appropriately smooth and identifiable in $\theta$} and the prior  $\Pi_n(\cdot)$ puts positive mass around the true parameter $\theta_0$, then the posterior distribution of $\sqrt{n}(\theta-\wh\theta_n)$ tends to $N(0,I_{\theta_0}^{-1})$ for most observations $\Xn$,
where $\wh \theta_n$ is an efficient estimator and $I_\theta$ is the Fisher information matrix of the model at $\theta$. In this section, we want to understand   the effect of the tilting factor $\e^{u_\theta(X^{(n)})}$ on the limiting shape of the pseudo-posterior  in \eqref{eq:posterior_density} that is proportional to $\pi_n(\theta\C\Xn)\e^{u_\theta(\Xn)}$. Exponential tilting is particularly intuitive  for linear $u_\theta(\Xn)$ and for Gaussian posteriors   where it implies a location shift.
Example \ref{ex:linear} below reveals how  the behavior of $u^*(\Xn)$ affects the centering of the posterior limit  (under linearity and Gaussianity)
\begin{exa}(Linear $u_\theta$)\label{ex:linear}
Suppose that the posterior $\pi_n(\theta\C\Xn)$ is Gaussian with some mean $\mu$ and covariance $\Sigma$. This  holds approximately in regular models according to the BvM theorem (Theorem 10.1 in \cite{v1998}). 
Assume that there exists an invertible mapping $\tau:\Theta\rightarrow\Theta$ such that $\theta=\tau(\bar\theta)$ where the density for $\bar\theta$ satisfies
$
\pi_n(\theta\C\Xn)\e^{u_\theta(X^{(n)})}\d\theta\propto \pi^*_n(\bar\theta\C\Xn)\d\bar\theta.
$
Assuming the following linear form   (justified in Remark \ref{rem:linear1})
\begin{equation}\label{eq:linear}
u_\theta(X^{(n)})=a^*(X^{(n)})+ \theta' u^*(X^{(n)})
\end{equation} we obtain
$\bar\theta\sim \mathcal N(\mu+\Sigma \, u^*(X^{(n)}),\Sigma)$. In this case, the mapping $\tau$ satisfies
$\theta=\tau(\bar\theta)=\bar\theta-\Sigma u^*(X^{(n)})$, implying a location shift. We had concluded a similar property below Theorem  \ref{thm:linear} at the end of Section \ref{sec:residual}.
\end{exa}

We now turn to more precise statements by recollecting the BvM phenomenon under misspecification in LAN models \cite{kleijn2}. The centering and the asymptotic covariance matrix will be ultimately affected by $\theta^*$ in \eqref{eq:KL_theta}.  
 
\begin{lem}(Bernstein von-Mises)\label{thm:bvm}
Assume that the posterior \eqref{eq:pstar} concentrates around $\theta^*$ at the rate $\varepsilon_n^*$
and that   for every compact $K\subset \R^d$
\begin{equation}\label{eq:LAN_misspec}
\sup_{h\in K}\left| \log\frac{\wt p_{\theta^*+\varepsilon_n^*h}^{(n)}(\Xn)}{\wt p_{\theta^*}^{(n)}(\Xn)}-h'\wt V_{\theta^*}\wt\Delta_{n,\theta^*}-\frac{1}{2}h'\wt V_{\theta^*}h\right|\rightarrow 0\quad\text{in $P_{\theta_0}^{(n)}$-probability}
\end{equation}
for some random vector $\wt\Delta_{n,\theta^*}$ and a non-singular matrix $\wt V_{\theta^*}$. Then  the pseudo-posterior converges to a sequence of normal distributions in total variation at the rate $\varepsilon_n^*$, i.e.
$$
\sup_B \left| \Pi_n^*\left(\varepsilon_n^{*-1}(\theta-\theta^*) \in B\C\Xn\right)-N_{\wt\Delta_{n,\theta^*},\wt V_{\theta^*}}(B)\right|\rightarrow 0 \quad\text{in 
$P_{\theta_0}^{(n)}$-probability}.
$$
\end{lem}
\proof Follows from Theorem 2.1 of \cite{kleijn2}.\smallskip

It remains to examine the assumption \eqref{eq:LAN_misspec}.  For \texttt{iid} data, \cite{kleijn2} derived sufficient conditions (Lemma 2.1) for \eqref{eq:LAN_misspec} to hold. Due to the non-separability of the term $u_\theta(\Xn)$, the mis-specified model cannot be regarded as arriving from an \texttt{iid} experiment. 
In Lemma \ref{lemma:lan} below we nevertheless provide  intuition for when \eqref{eq:LAN_misspec} is expected to hold  if $u_\theta(\Xn)$ is linear.  Recall that in Remark \ref{rem:linear1} we have concluded that under differentiability, the posterior residual $u_\theta(\Xn)$ does converge to a linear function in $\theta$.
Below, we provide sufficient conditions for the  LAN assumption \eqref{eq:LAN_misspec},  relaxing slightly Lemma 2.1 in 
\cite{kleijn2}. The assumptions in Lemma \ref{lemma:lan} are closely related to the ones in Theorem \ref{thm:linear}. The main difference is that Lemma \ref{lemma:lan} is concerned with the behavior of the (misspecified) likelihood around $\theta^*$ as opposed to $\theta^0$.

\begin{lem}\label{lemma:lan}
Assume that $P_{\theta_0}^{(n)}=P_{\theta_0}^n$ with a density $\prod_{i=1}^np_{\theta_0}(x_i)$ where the function
$\theta\rightarrow\log p_\theta(x)$ is differentiable at $\theta^*$ with a derivative $\dot \ell_\theta$. Assume  there exists an open neighborhood $U$ of $\theta^*$ such that
$
\left| \log\frac{p_{\theta_1}(x)}{p_{\theta_2}(x)}\right|\leq m_{\theta^*}\|\theta_1-\theta_2\|\,\,P_{\theta_0}-a.s.\,\forall \theta_1,\theta_2\in U
$
where $m_\theta$ is a square integrable function. Assume that the log-likelihood has a 2nd order Taylor expansion around $\theta^*$ (i.e. \eqref{eq:quadratic} holds).
 Assume that $u_\theta$ is asymptotically linear around $\theta^*$(i.e. \eqref{eq:as_linearity} holds), then \eqref{eq:LAN_misspec} holds with $\varepsilon_n^*=1/\sqrt{n}$ and
\begin{equation}\label{eq:verification}
\wt V_{\theta}=V_{\theta}\quad\text{and}\quad \wt\Delta_{n,\theta}=V_{\theta}^{-1}\left[\frac{\dot C_{\theta}}{\sqrt{n}}
+ \sqrt{n}\P_n\dot\ell_{\theta}+\frac{u^*(\Xn)}{\sqrt{n}} \right]
\end{equation}
\end{lem}
\proof 

We can write
\begin{equation}\label{eq:expansion}
\log\frac{\wt p_{\theta^*+\varepsilon_nh}^{(n)}}{\wt p_{\theta^*}^{(n)}}=
\log \frac{C_{\theta^*+\varepsilon_nh}}{C_{\theta^*}}+\log\frac{ p_{\theta^*+\varepsilon_nh}^{(n)}}{p_{\theta^*}^{(n)}}
+u_{\theta^*+\varepsilon_nh}-u_{\theta^*}.
\end{equation}
{This yields, from Lemma 19.31 in  \cite{v1998}, that
$$
\bG_n\left(\sqrt{n}\log \frac{p_{\theta^*+h/\sqrt{n}}}{p_\theta^*}-h'\dot\ell_{\theta^*}\right)\rightarrow 0\quad \text{in $P_{0}$},
$$
where $\bG_n=\sqrt{n}(\P_n-P_{\theta_0})$ is the empirical process.
Assuming that 
\begin{equation}\label{eq:quadratic}
P_{\theta_0}\log \left(\frac{p_\theta}{p_{\theta^*}}\right)=P_{\theta_0}\dot\ell_{\theta^*}'(\theta-\theta^*)+
\frac{1}{2}(\theta-\theta^*)'V_{\theta^*}(\theta-\theta^*)+o(\|\theta-\theta^*\|^2)\quad\text{as $\theta\rightarrow\theta^*$}
\end{equation}
one obtains
\begin{align*}
\log \frac{p^{(n)}_{\theta^*+h/\sqrt{n}}}{p^{(n)}_{\theta^*}}=n\P_n\log \frac{p_{\theta^*+h/\sqrt{n}}}{p_{\theta^*}}
&=o_{P}(1)+\bG_nh'\dot\ell_{\theta^*}+nP_{\theta_0}\log \frac{p_{\theta^*+h/\sqrt{n}}}{p_{\theta^*}}\\
&=o_{P}(1)+\bG_nh'\dot\ell_{\theta^*}+\frac{h_n'V_{\theta^*}h}{2}+\sqrt{n}P_{\theta_0}h'\dot\ell_{\theta}
\end{align*}
If we assume asymptotic linearity of $u_\theta$ around $\theta^*$, i.e. 
\begin{equation}\label{eq:as_linearity}
u_{\theta^*+h/\sqrt{n}}(\Xn)-u_{\theta^*}(\Xn)=\frac{1}{\sqrt{n}}h'u^\star(\Xn)+o_{P}(1)
\end{equation}
for some $u^\star(\Xn)$
and
$$
\log  \frac{C_{\theta^*+h_n/\sqrt{n}}}{C_{\theta^*}}=\frac{\dot C_{\theta^*}'h_n}{\sqrt{n}}+o(1)
$$
then \eqref{eq:LAN_misspec} holds with \eqref{eq:verification}. \qed

Related BvM conditions have been characterized in \cite{bc2009}. We restate these conditions utilizing the localized re-parametrization $h=\sqrt{n}(\theta-\theta_0)-s$, where $s=\sqrt{n}(\hat\theta-\theta_0)$  is a {\em zero-mean} vector where $\hat\theta$ is some suitable estimator. We first define a localized criterion function 
$
\ell(h)\equiv \frac{\wt p_{\hat \theta+h/\sqrt{n}}(\Xn)\wt\pi(\hat\theta +h/\sqrt{n})}{
\wt p_{\hat\theta}(\Xn)\wt\pi(\hat\theta)},
$ 
which corresponds to the normalized pseudo-posterior $\pi^*(\theta\C\Xn)/\pi^*(\hat \theta\C\Xn)$. \cite{bc2009} impose a centered  variant of \eqref{eq:LAN_misspec} requiring that  $\ell(h)$
approaches a quadratic form on a closed ball $K$ (such that\footnote{$\int_K \ell(h)\d h/\int_\Lambda \ell(h)\d h\geq 1-o_{P_{\theta_0}}(1)$ and $\int_K\phi(h)\d h$ for $\phi(\cdot)$ standard Gaussian density} $\Lambda\equiv \sqrt{n}(\Theta-\theta_0)-s=K\cup K^c$) in the sense that
\begin{equation}\label{eq:loglike}
|\log \ell(h)- (-h'Jh)/2|\leq \epsilon_1+\epsilon_2\times  h'Jh /2\quad\forall h \in K,
\end{equation}
for some matrix $J>0$ with eigenvalues bounded away from zero. If 
\begin{equation}\label{eq:conditions_eps}
\epsilon_1=o(1)\quad\text{and}\quad \epsilon_2\times \lambda_{max}^2(J)(\sup_{h\in K}\|h\|)^2=o(1) \quad\text{in $P_{\theta_0}^{(n)}$-probability.}
\end{equation}
Theorem 1 of \cite{bc2009} shows that $\ell(h)/\int_\Lambda \ell(h)\d h$ approaches the standard normal density in $P_{\theta_0}^{(n)}$-probability as $n,d\rightarrow\infty$. The condition \eqref{eq:loglike} (a) allows for mild deviations from smoothness and log-concavity, (b) involves also the prior (unlike \eqref{eq:LAN_misspec}) but, (c) requires the existence of a $\sqrt{n}-$consistent estimator $\hat\theta$. Lemma \ref{thm:bvm} is more general, where the rate $\varepsilon^*_n$ does not need to be $1/\sqrt{n}$ and where the posterior is allowed to have a non-vanishing bias.  The requirement \eqref{eq:conditions_eps} imposes certain restrictions on $u_\theta(\Xn)$. For example, in the linear case \eqref{eq:linear} one would need $u^\star(\Xn)=o(\sqrt{n})$ in $P_{\theta_0}^{(n)}$-probability from  \eqref{eq:conditions_eps}.

{
\section{Alternatives to MHC}\label{sec:variants}
A recent paper \cite{hermans} suggests a related Metropolis-Hastings strategy which relies on a simulation-based likelihood ratio estimator trained separately from the Markov chain simulation. This estimator is based on contrastive learning between two fake data-parameter pairs, with parameters sampled from the prior and with fake data generated  either from the marginal or the conditional likelihood evaluated at sampled prior parameters \cite{gutman2}. See also \cite{density_ratio} (Chapter 12) for conditional density estimation using machine learning.
Using the marginal distribution $p(\cdot)$ as a reference and denoting with $D_\theta^m(X)=\frac{p(X)\pi(\theta)}{p(X)\pi(\theta)+p_\theta(X)\pi(\theta)}$ 
 we can re-write \eqref{eq:likelihood}  as $p_\theta^{(n)}=p^{(n)}(X^{(n)})\exp\left(\sum_{i=1}^n\log\frac{1-D_\theta^m(X_i)}{D_\theta^m(X_i)}\right)$, where $p^{(n)}(X^{(n)})$ is the marginal likelihood. 
 Similarly as in (2.5), a likelihood estimator can be then obtained by replacing $D_\theta^m$ with $\hat D_\theta^m$, which is now trained solely on simulated data.  The expression (2.5) then still holds with $u_\theta$ now defined  using $D_\theta^m$ and 
 $\hat D_\theta^m$. We implement this approach  in Section \ref{sec:lotka} (main document) and discuss its theoretical properties in Remark \ref{rem:marginal}.  This approach will be advantageous when  the cost of learning the likelihood ratio 
 simulator prior to MCMC simulation outweighs the costs of performing classification at each MH step.  Another related strategy was proposed in \cite{pham}, where no reference is used and the likelihood ratio inside Metropolis-Hastings is estimated by contrastive learning between two fake data generated from conditional likelihoods evaluated at new versus current parameter values. This approach also requires classification at each step but is not limited by the sample size $n$ when choosing the fake data sample size $m$ for classification.  We also implement this approach later in Section \ref{sec:lotka}  and make comparisons with our approach in Section 12 in the Supplement. The choice of the contrasting density in the context of parameter estimation in unnormalized models is discussed in   \cite{gh_12}. 
}

\section{Ricker Model}\label{sec:ricker}
The Ricker model is a classic discrete  model that describes partially observed population dynamics of fish and animals in ecology.
The latent population $N_{i,t}$ follows
\[
	\log N_{i,t+1}=\log r+\log N_{i,t}-N_{i,t}+\sigma\varepsilon_{i,t}, \qquad \varepsilon_{i,t}\sim N(0,1),
\]
where $r$ denotes the intrinsic growth rate and $\sigma$ is the dispersion of innovations.
The index $t$ represents time and runs through 1 to $T=20$. The index $i$ represents independent observations and runs through 1 to $n=300$.
The initial population $N_{i,0}$ may be set as $1$ or set randomly after some burn\hyp{}in period.
We observe $X_{i,t}$ such that
\[
	X_{i,t}\mid N_{i,t}\sim\text{Poisson}(\varphi N_{i,t}),
\]
where $\varphi$ is a scale parameter.
The objective is to make inference on $\theta\vcentcolon=(\log r,\sigma^2,\varphi)$.
Each time sequence $X_i\vcentcolon=(X_{i,1},\dots,X_{i,T})$ constitutes an observation, where $i$ runs through $n$.
In our notation, we can define the underlying data-generating process as $\wt{X}_{i,t}\vcentcolon=(U_{i,t},\varepsilon_{i,t})$ for $U_{i,t}\sim U[0,1]$ and set the function $T_\theta$ to map $\varepsilon_i$ to $N_i$ and then $(U_i,N_i)$ to $X_i$ through the Poisson inverse transform sampling of $U_{i,t}$ into $X_{i,t}$.
We set the true parameter as $(\log r_0,\sigma_0^2,\varphi_0)=(3.8,1,10)$ and employ an improper, flat prior.
Note that our method can accommodate an improper prior, unlike ABC.

\begin{figure}
\scalebox{0.25}{\includegraphics{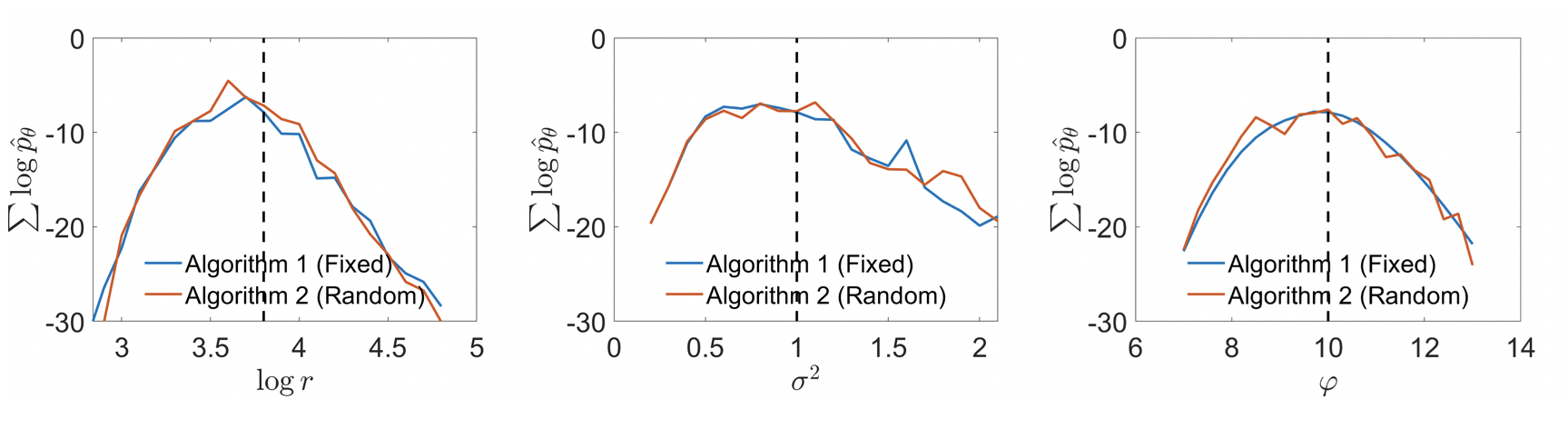}}
\caption{\small Estimated log likelihood ratio for the Ricker model: (Left) function of  $\log r$ fixing $\sigma=\sigma_0$ and $\varphi=\varphi_0$, (Middle) 
function of  $\sigma^2$ fixing $r=r_0$ and $\varphi=\varphi_0$, (Right) function of  $\varphi$ fixing $\sigma=\sigma_0$ and $r=r_0$.}
\label{fig:ricker:quad}
\end{figure}

There is no obvious sufficient statistic for this model, and the likelihood is intractable due to the nontrivial time dependence of $N_{i,t}$.
We use an average of neural network discriminators to adapt to the unknown likelihood ratio.
First, we estimate $D_\theta$ by a neural network with one hidden layer with 50 nodes, each of which is equipped with the hyperbolic tangent sigmoid activation function.
Then, we compute the log likelihood of the data $\sum_i\log\frac{1-\hat{D}_\theta}{\hat{D}_\theta}$.
We repeat this for 20 times with independently drawn $\wt{X}$ and take the average of the log likelihood.
This specification produces approximately quadratic likelihood-ratio curves (\Cref{fig:ricker:quad}). Unlike  the location-scale normal model, the fixed design does not produce entirely smooth curves due to the averaging aspect over many discriminators. The quadratic shape is nevertheless recovered here, implying that the differentiability assumptions from Section \ref{sec:residual} are not entirely objectionable.

\begin{figure}[!t]
\scalebox{0.3}{\includegraphics{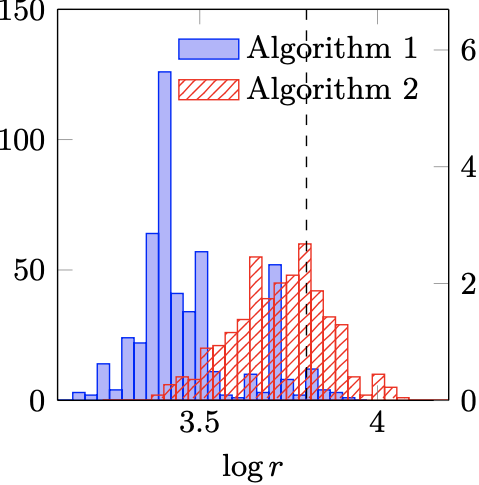}}
\scalebox{0.3}{\includegraphics{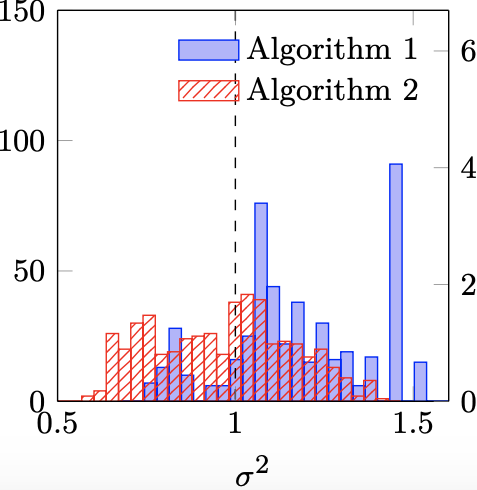}}
\scalebox{0.3}{\includegraphics{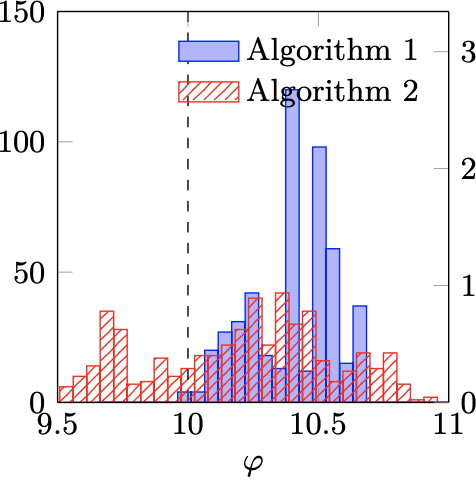}}
\caption{MHC samples for the Ricker model.}
\label{fig:ricker:hist}
\end{figure}

\Cref{fig:ricker:hist} shows the marginal histograms of the MHC samples ($500$ MCMC iterations). The proposal distribution is independent across parameters; $\log r$ uses the normal distribution, $\sigma^2$ the inverse\hyp{}gamma distribution, and $\varphi$ the gamma distribution; each of them has the mean equal to the previous draw and variance $1/n$. The vertical dashed lines indicate the true parameter $\theta_0$.
Note that the posterior is asymptotically centered at the MLE, not $\theta_0$.
However, the blue histograms on the left (Algorithm 1) seem too far away from $\theta_0$ relative to the widths of the histograms.
On the other hand, the red histograms (Algorithm 1) are more dispersed but located closer to $\theta_0$. These observations confirm our theoretical findings.
Histograms of Algorithm 3 (Figure \ref{fig:ricker:hist2}) look reasonable as a posterior sample, center around the true values.

\begin{figure}[!t]
\scalebox{0.3}{\includegraphics{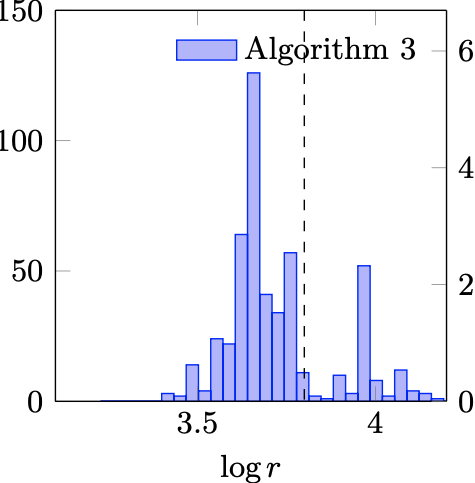}}
\scalebox{0.3}{\includegraphics{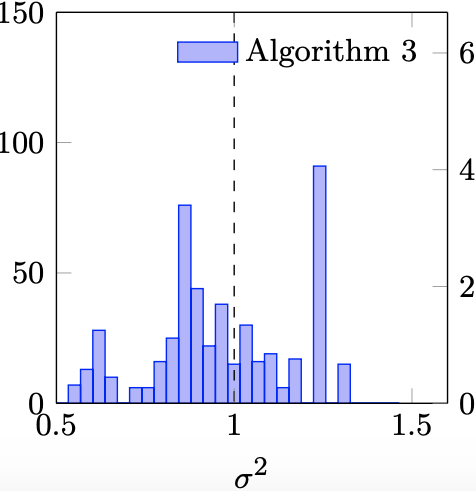}}
\scalebox{0.3}{\includegraphics{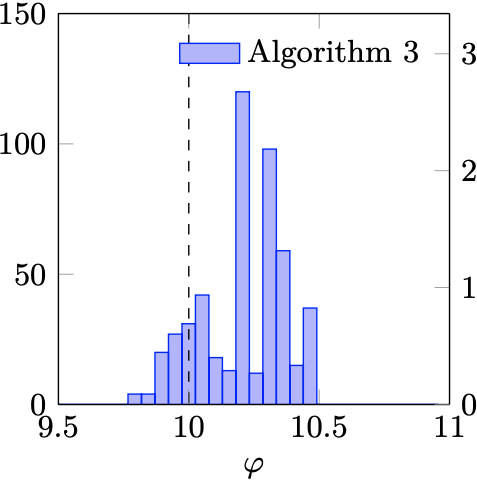}}
\caption{\small  MHC samples for the Ricker model (Algorithm 3)}
\label{fig:ricker:hist2}
\end{figure}

Figure \ref{fig:ricker:hist2}  and  \ref{fig:ricker:hist3} compare our method with the MCWM pseudo\hyp{}marginal Metropolis\hyp{}Hastings algorithm \cite{andrieu2009pseudo}.
We have implemented the default pseudo\hyp{}marginal method which deploys an average of conditional likelihoods for $X_i$, given $N_i$,
\[
	\hat{p}(X_i)=\frac{1}{K}\sum_{k=1}^{K}\prod_{t=1}^T p(X_{i,t}\mid N_{i,t,k})=\frac{1}{K}\sum_{k=1}^{K}\prod_{t=1}^{T}\frac{(\varphi N_{i,t,k})^{X_{i,t}}e^{-\varphi N_{i,t,k}}}{X_{i,t}!}
\]
as the likelihood approximation, where $K$ is some positive integer and  where $N_{i,t,k}$ are independently drawn across $k=1,\dots,K$.  
In our comparisons, we let $K=20 n$.
Figure \ref{fig:ricker:hist2} shows that the two methods produce posterior draws that are located at similar places, and the widths of the histograms are also comparable.
We would like to point out, again, that our method does not require that a tractable conditional likelihood is available nor that a user\hyp{}specified summary statistic is supplied.

\begin{figure}[!t]
\scalebox{0.3}{\includegraphics{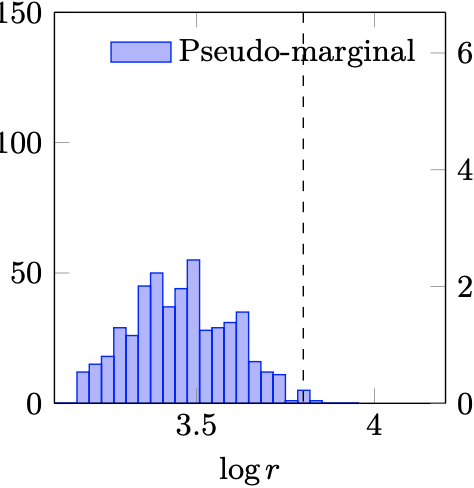}}
\scalebox{0.3}{\includegraphics{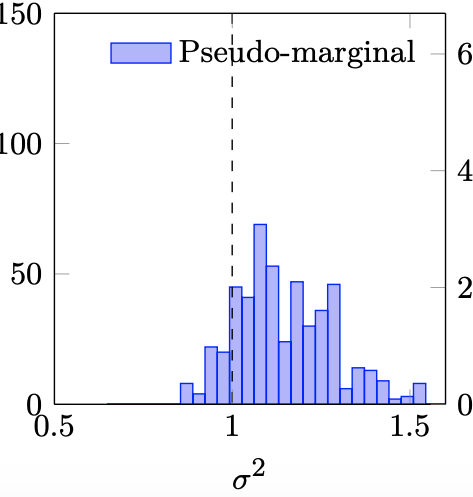}}
\scalebox{0.3}{\includegraphics{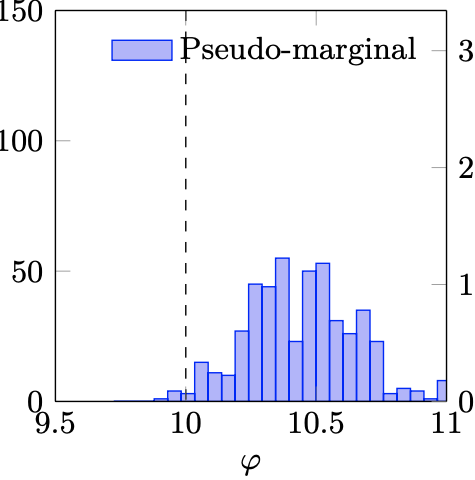}}
\caption{\small  Posterior samples for the Ricker model using the pseudo-marginal   MCWM method}
\label{fig:ricker:hist3}
\end{figure}

\section{Bayesian Model Selection}\label{sec:bf}
The performance of summary statistic\hyp{}based methods is ultimately sensitive to the quality of summary statistics whose selection can be a  delicate matter.
One such instance is  model selection, where it is known  that when ABC may fail even when the summary statistic is sufficient for {\em each} of the models considered \citep{robert2011lack}.
Our method {\em does not} require a summary statistic but a sieve of discriminators that can adapt to the oracle discriminator in the limit. This creates hope that our method can tackle model selection problems.
To illustrate this point we consider a toy model choice problem considered in \cite{robert2011lack}.
The actual data follows $X_i\sim N(0,1)$ for $i=1,\dots, n=500$. We have two candidate models $P_{1,\mu}=N(\mu,1)$ and $P_{2,\mu}=N(\mu,1+3/\sqrt{n})$ to choose from.
We let the parameters be $\theta\vcentcolon=(m,\mu)$, where $m\in\{1,2\}$ is the model indicator and $\mu$ is unknown mean with a prior $N(0,1)$. The model is assigned a uniform prior, i.e. $P(m=1)=P(m=2)=0.5$.  Following the traditional Bayesian model selection formalism, we collect evidence for model $m=1$  with a Bayes factor 
\[
	B_{12}\vcentcolon=\frac{\pi_n(m=1\mid X)}{\pi_n(m=2\mid X)}.
\]
The Bayes factor is the ratio of the marginal likelihoods (or posterior probabilities) of $m=1$ over $m=2$.   The actual Bayes factor value is $B_{12}=9$, indicating strong evidence in favor of $m=1$.
The Bayes factor will be estimated by the ratio of the frequencies of the posterior samples given by ABC or our method.
Since our parameter of interest $m$ is discrete, there is no de-biasing for this example.
\cite{robert2011lack} in their Lemma 2 show that when the summary statistic is $\sum_i X_i$, the Bayes factor estimated by ABC asymptotes to $1$.
This is equivalent to choosing the model with a coin toss.
For our method, we use the logistic regression on regressors $(1,X_i,X_i^2)$, which can mimic  the oracle discriminator.
\begin{figure}
{\includegraphics[width=5cm,height=4.5cm]{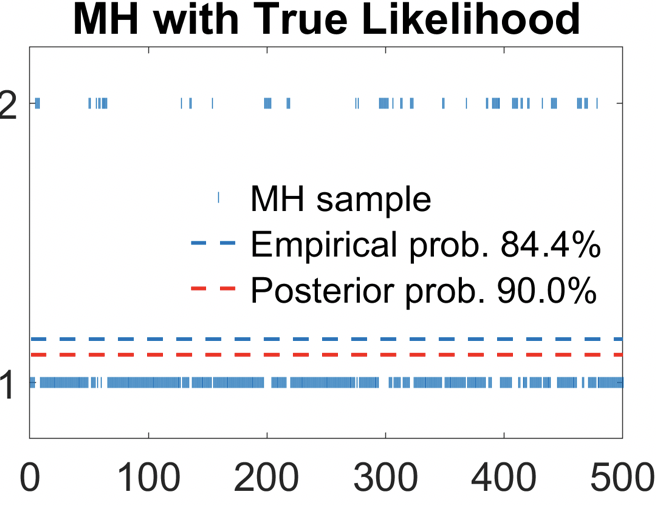}}
{\includegraphics[width=5cm,height=4.5cm]{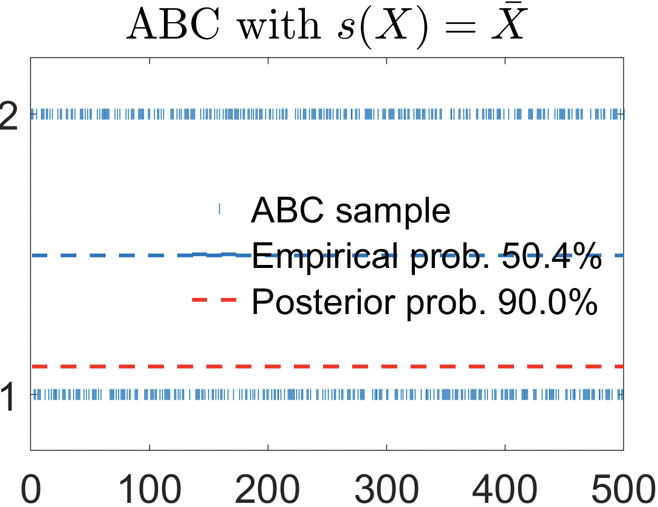}}
{\includegraphics[width=5cm,height=4.5cm]{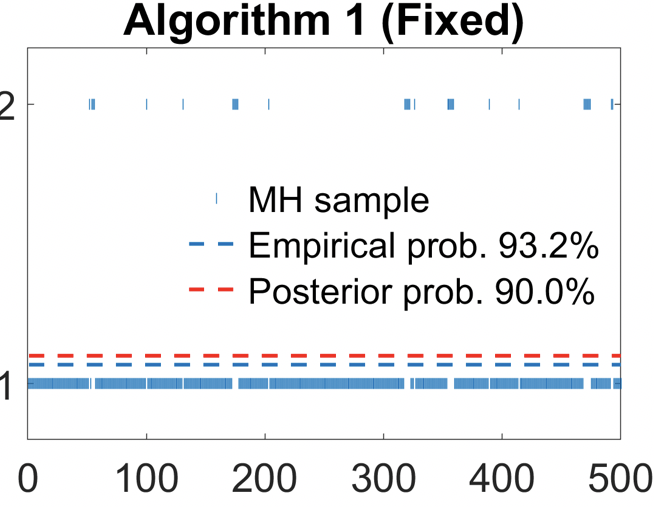}}
\caption{\small Trace plots of sampled models using: (Left) MH with the true likelihood ratio, (Middle) ABC with $s(\Xn)=\bar X_n$ and (Right) fixed generator MHC.}\label{fig:ricker:trace}
\end{figure}
The trace plots of sampled models for exact MH, MHC and ABC are provided in Figure \ref{fig:ricker:trace}.  Table \ref{tbl:selection} summarizes the  posterior model frequencies.
The true posterior probabilities are $\pi_n(m=1\mid X)\approx 0.9$ and $\pi_n(m=2\mid X)\approx 0.1$, so the Bayes factor is 9.
The ``Oracle MH'' is the Metropolis\hyp{}Hastings with the true likelihood, in which 84.4\% of the posterior draws choose model 1.
Algorithms 1 and 2 choose model 1 respectively 93.2\% and 70\% of the times.
ABC based on the sum, on the other hand, chooses the model randomly. Finally, \Cref{fig:selection:quad} in Appendix  gives the estimated log-likelihood ratio for each model.
In terms of $\mu$, we again see that Algorithm 1 is slightly biased with the correct shape and Algorithm 2 is less biased but more dispersed on average.

\begin{table}[t!]
\centering
\small
\begin{tabular}{lccccc}
\hline\hline
& Posterior & Oracle MH & Algorithm 1 & Algorithm 2 & ABC \\
\hline
Model 1 & 90\% & 422 & 466 & 350 & 252 \\
Model 2 & 10\% & \hphantom{0}78 & \hphantom{0}34 & 150 & 248 \\
\hline
Bayes factor & 9.00 & 5.41 & 13.71 & 2.33 & 1.02 \\
\hline
\end{tabular}
\caption{``Posterior'' column gives the posterior probability of each model, $\pi_n(m=j\mid X)$. Other columns give the frequencies of the corresponding sample of size 500. ``Oracle MH'' refers to the Metropolis\hyp{}Hastings algorithm with the true likelihood. ``ABC'' is based on the summary statistics $s(X)=\bar{X}_n$.}
\label{tbl:selection}
\end{table}

Figure \ref{fig:selection:quad} shows true likelihood ratio and  and classification-based estimates for fixed and random designs for the Bayesian model selection example from Section \ref{sec:bf}.
Under the fixed design, the curve is smooth and slightly biased with a similar shape to the true log-likelihood. For the random design, there is no smoothness (due to the fake data refreshing aspect).

\begin{figure}
\centering
\scalebox{0.3}{\includegraphics{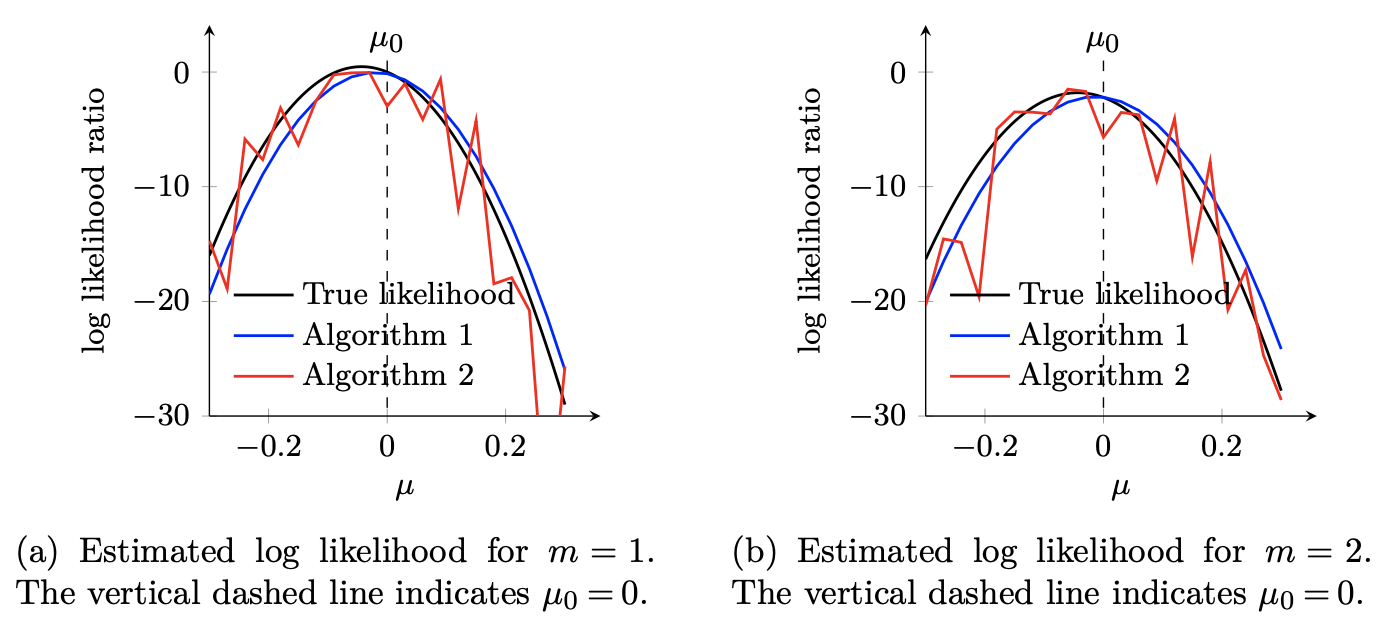}}
\caption{Estimated log likelihood for models 1 and 2. The figures indicate that it is smooth in $\mu$ and have the same curvature as the true log likelihood.}
\label{fig:selection:quad}
\end{figure}

\vspace{-0.5cm}
\section{The CIR Model: Further Details}
This section presents additional plots for the CIR analysis from Section \ref{sec:cir}.
Figure \ref{fig:MHC_sup} shows smoothed posterior samples for MHC (fixed generator) and $nrep\in\{1,5\}$. These plots look qualitatively similar to the random generator results presented in Figure 2 in the main manuscript.
Next, Figure \ref{fig:MHC_trace_CIR1} and \ref{fig:MHC_trace_CIR2} show trace-plots of the MHC samples. We can see that (1) using larger $nrep$ reduces variance, (2) random generators have smaller acceptance rates for the same proposal distribution. {Trace-plots for the MCWM method (Figure \ref{fig:MCWM_trace}) show bias in estimation of $\sigma$.  }
Table \ref{tab:CIR} shows posterior summaries for the various algorithms we tried,  {including acceptance rates and effective sample size (computed using the \texttt{coda} R package).
Since MHC (random generator) resembles GIMH \citep{beaumont2003estimation} in that it recycles the fake data, one would expect the effective sample of  MHC to be smaller than for MCWM. However, making the MCWM likelihood estimator more accurate (increasing $M$ and $N$) made the effective sample size (ESS) smaller even though the acceptance rate was still around $10\%$. Interestingly, the random generator MHC also showed a decreased ESS (as well
as the acceptance rate) once we used ``better" log-likelihood estimator (i.e. averaging over $nrep = 5$ estimators using different fake data). }
Lastly, histograms of the posterior samples together with demarkations of the $95\%$ credible intervals are in Figure \ref{fig:MHC_hist_CIR1} and \ref{fig:MHC_hist_CIR2}.

{
\begin{table}[!h]
\scalebox{0.75}{\begin{tabular}{c | ccc | ccc | ccc | c  c c}
\hline\hline
Method              & \multicolumn{3}{c}{$\alpha$} &  \multicolumn{3}{c}{$\beta$} &  \multicolumn{3}{c}{$\sigma$} & AR & Time & ESS\\
\cline{2-10}
	                  & $\bar\alpha$ & l &u & $\bar\beta$ & l &u & $\bar\sigma$ & l &u &  &  &\\
\hline
MH Exact		 		&0.0693  &0.683 &0.703    &0.1558 &0.1507 &0.1608 &0.07 &0.696 &0.704 & 9.1  & 3.3           &255\\
Alg1 ($nrep=1$)        	&0.0691  &0.0644 &0.0735 &0.1505 &0.1374 &0.1636 &0.0703 &0.0669 &0.0734 &  16.8& 4.6 &191\\
Alg2 ($nrep=1$)	        &0.0691  &0.0644 &0.0741 &0.1476 &0.1353 &0.1632 &0.693 &0.0667 &0.0725 &10.7  & 4.9   &155\\
Alg1  ($nrep=5$)	        &0.0698  &0.0667 &0.0725 &0.1468 &0.1377 &0.1574 &0.0699 &0.676 &0.725 & 7.8 &13.9      &104\\
Alg2  ($nrep=5$)	        &0.0691  &0.0665 &0.0715 &0.1468 &0.1366 &0.1571 &0.0691 &0.0674 &0.0714 & 5.6 &13.9  &112\\
MCWM ($M=2$)		&0.0693  &0.0658 &0.0733 &0.1469 &0.1287 &0.1632 &0.067 &0.0657 &0.0684 & 13.1 &15.9  &316\\
MCWM ($M=5$)		&0.0694  &0.0662 &0.723 &0.1538 &0.1423 &0.1634 &0.0689 &0.0676 &0.0698& 10.1&238.6  &63\\ 
\hline\hline
\end{tabular}}
\caption{Posterior means and $95\%$ credible interval boundaries (lower (l) and upper (u)). $AR$ is the acceptance rate and  \texttt{Time} is computing time (in hours) for $10\,000$ iterations. \texttt{ESS} is the average effective sample size for the three chains computed using the R package \texttt{coda}.}\label{tab:CIR}
\end{table}}

\begin{figure}[!h]
\includegraphics[height=5.3cm,width=5.3cm]{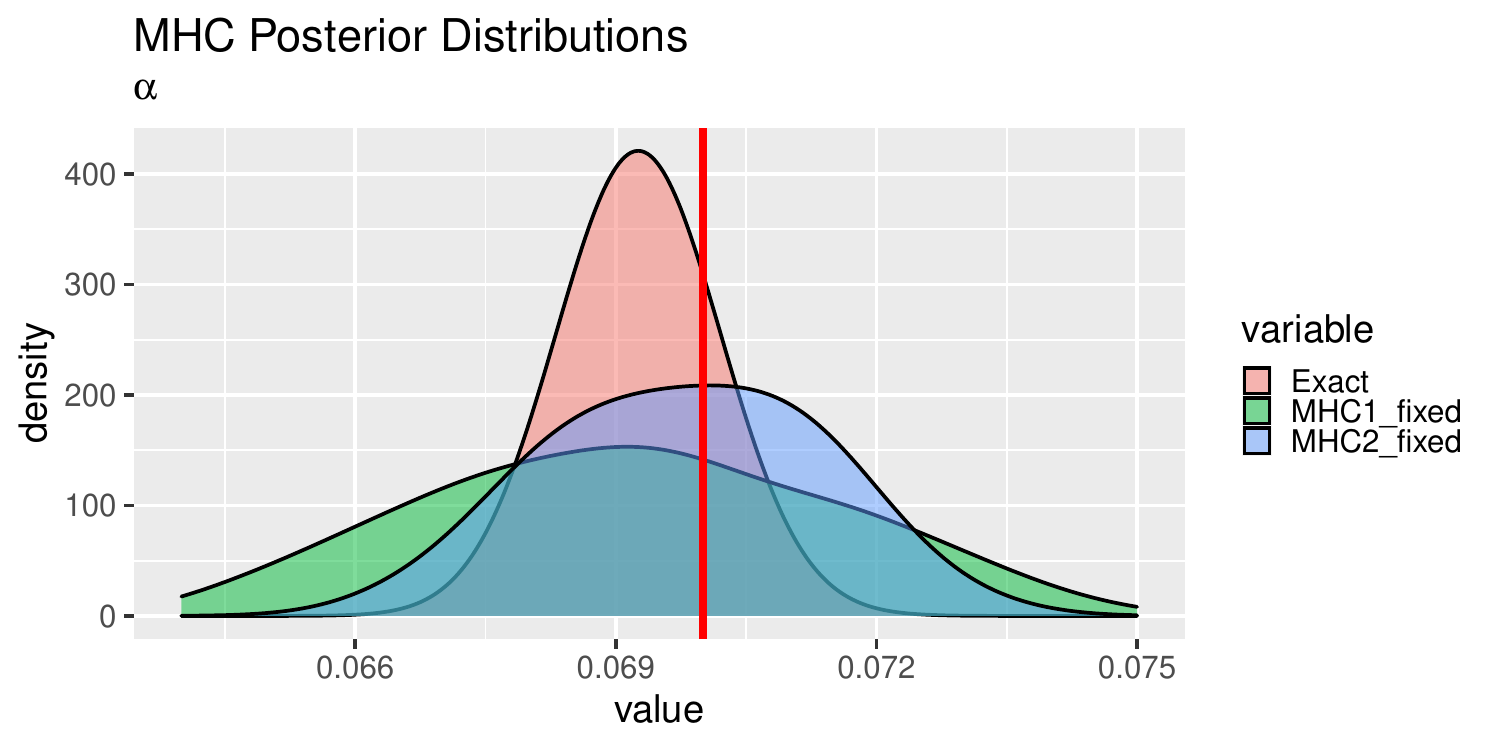} \includegraphics[height=5.3cm,width=5.3cm]{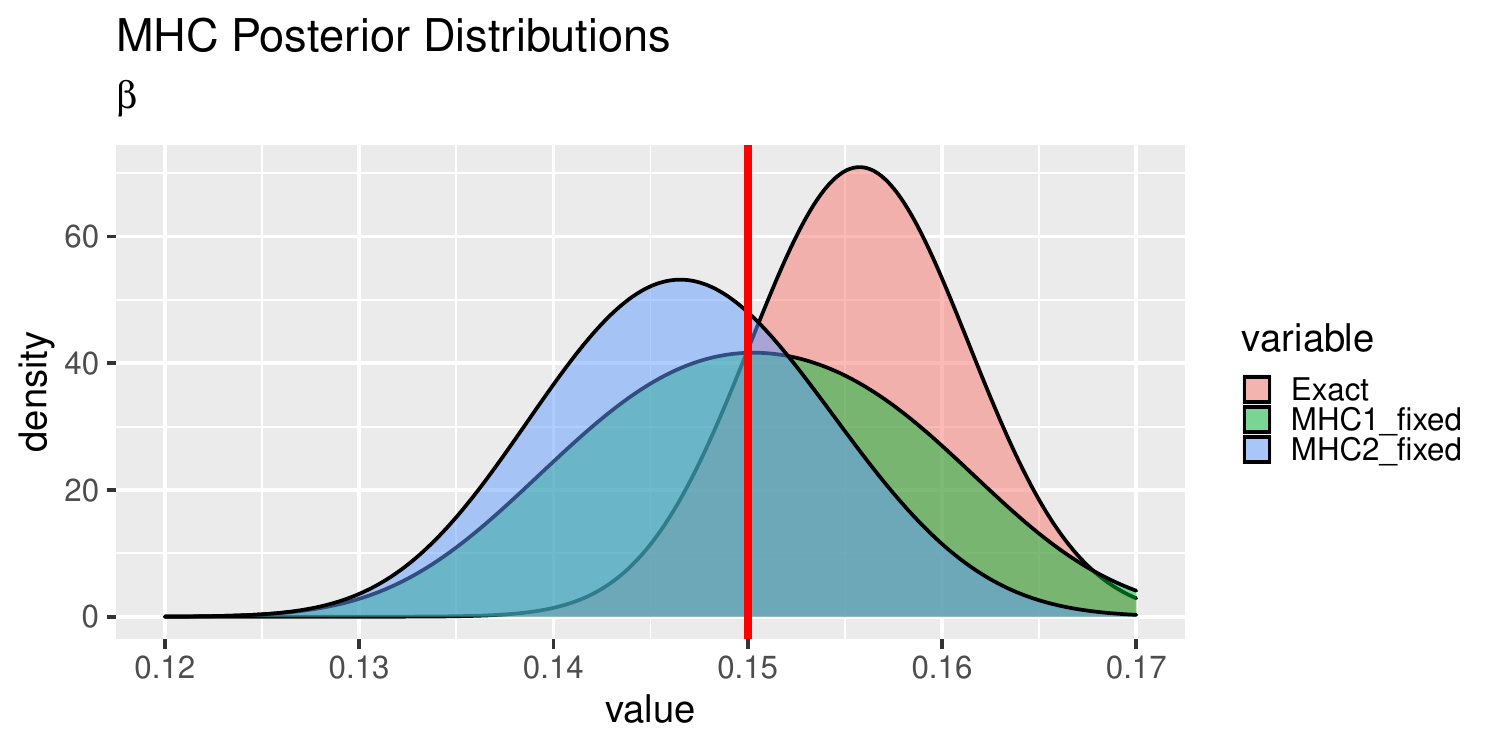} 
\includegraphics[height=5.3cm,width=5.3cm]{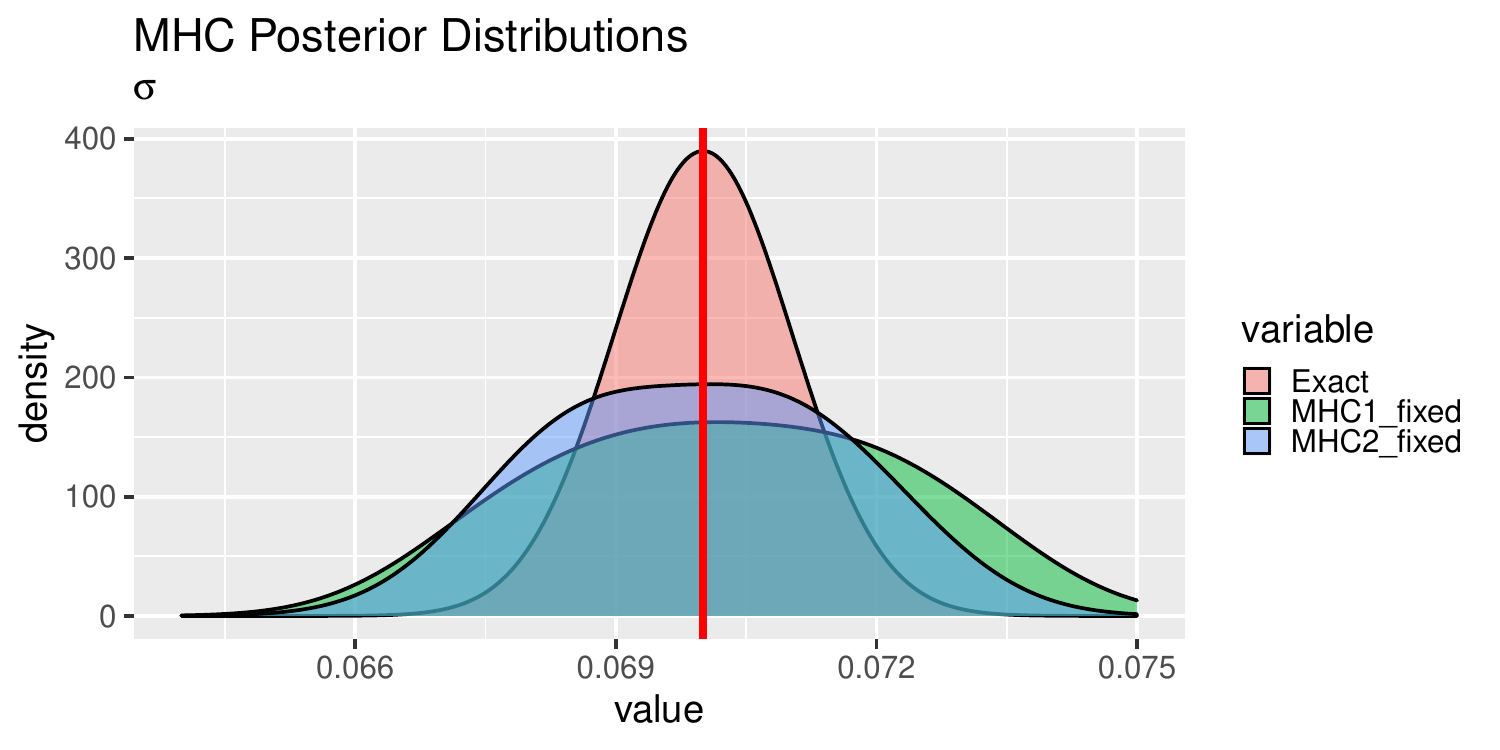} 
\caption{\small Smoothed posterior densities obtained by simulation using the exact MH and  MHC fixed generator using $nrep\in\{1,5\}$}\label{fig:MHC_sup}
\end{figure}

\begin{figure}
\scalebox{0.8}{\includegraphics{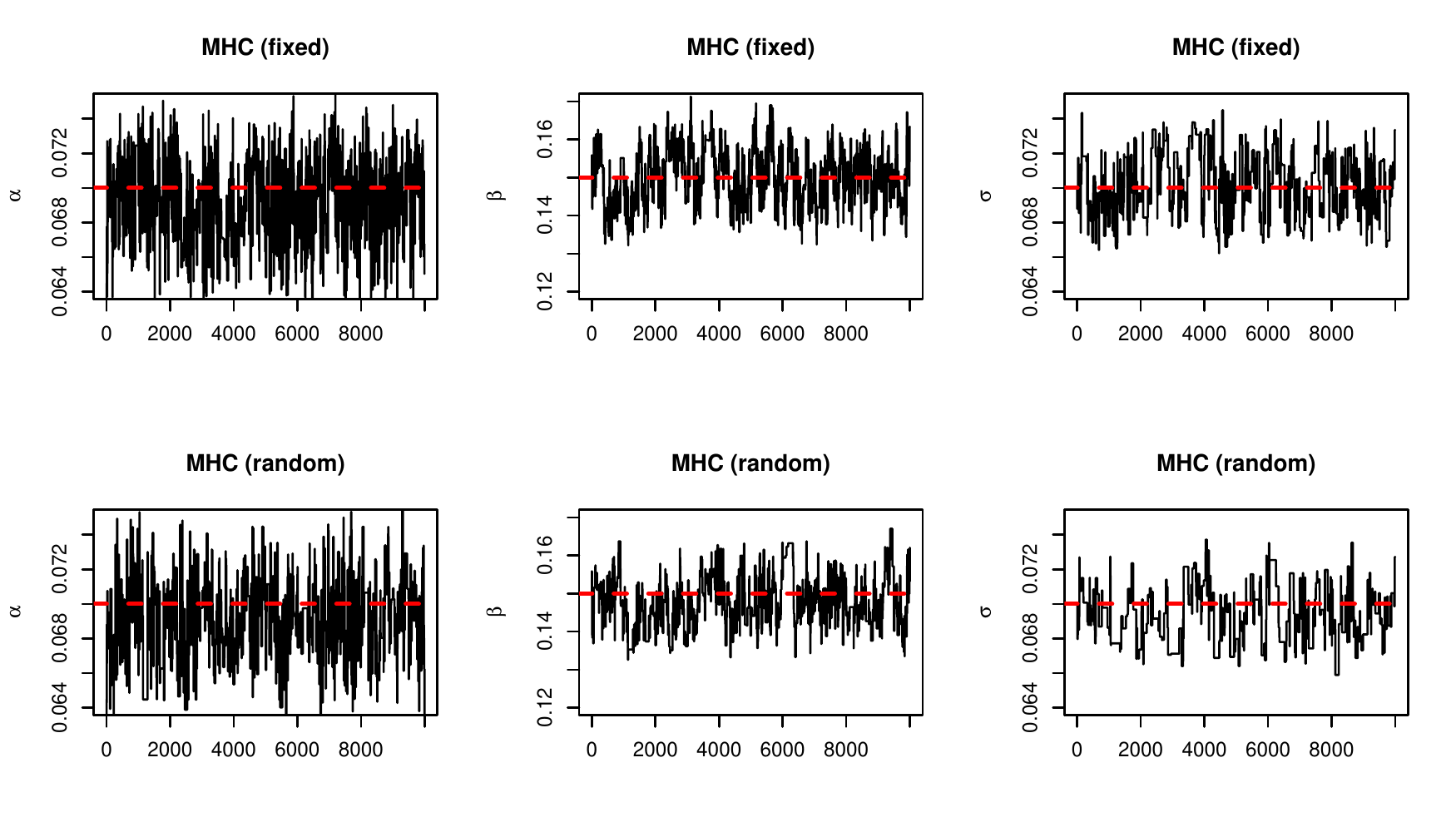}}
\caption{\small Trace-plots of $10\,000$ MHC iterations with $nrep=1$}\label{fig:MHC_trace_CIR1}
\end{figure}

\begin{figure}
\scalebox{0.8}{\includegraphics{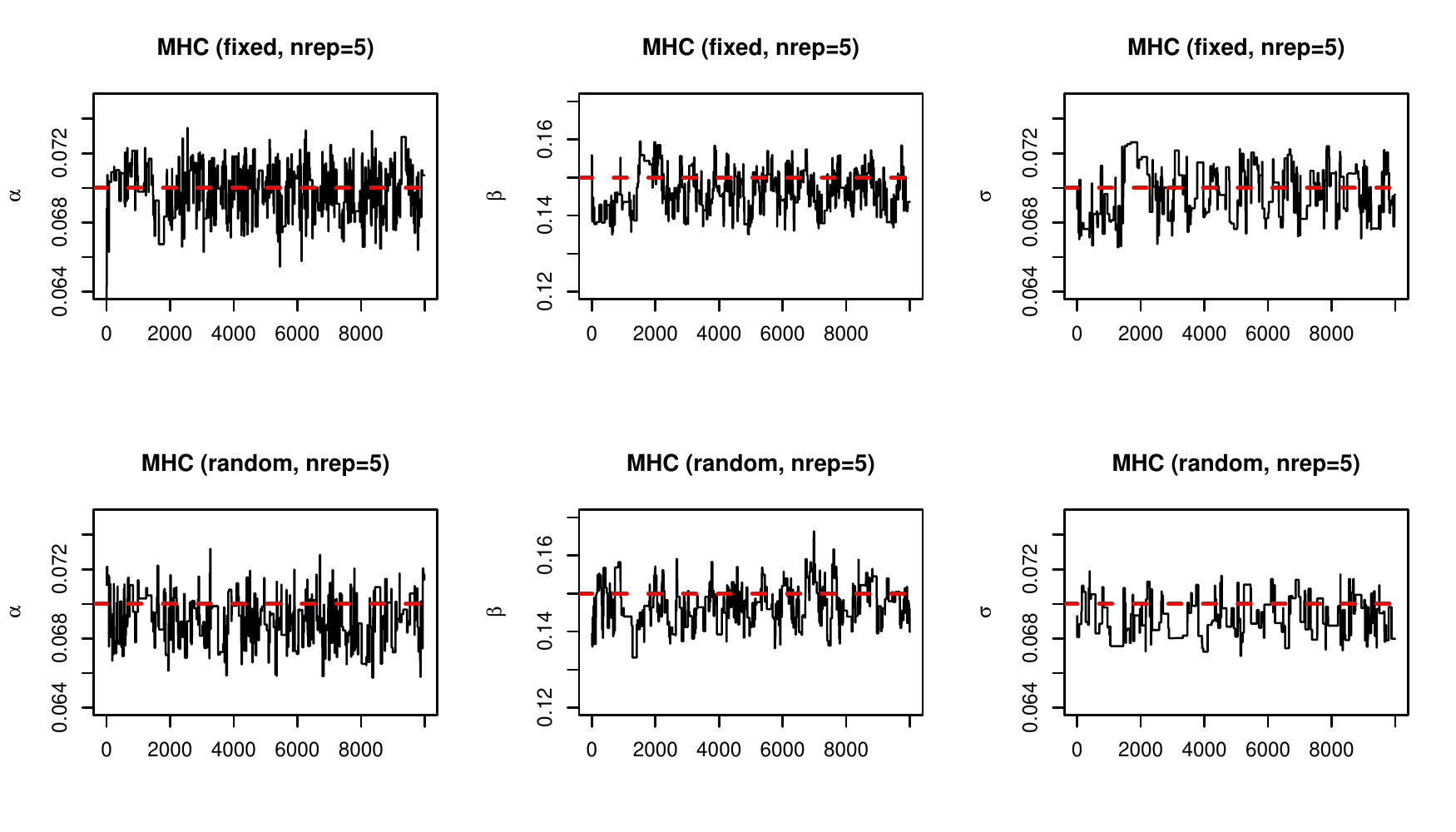}}
\caption{\small Trace-plots of $10\,000$ MHC iterations with $nrep=5$}\label{fig:MHC_trace_CIR2}
\end{figure}

\begin{figure}
\scalebox{0.8}{\includegraphics{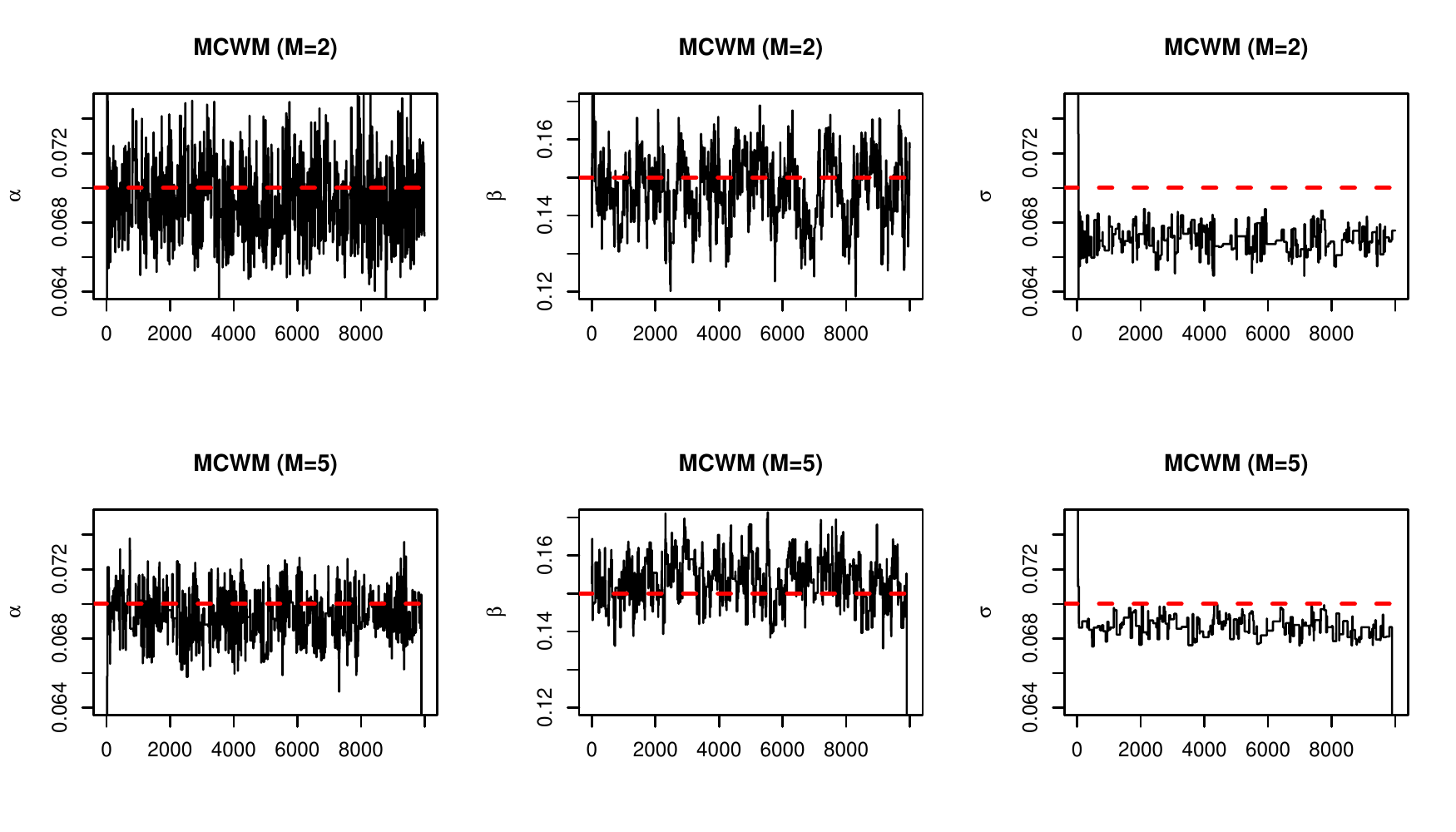}}
\caption{\small Trace-plots of $10\,000$ MCWM iterations with $M\in\{2,5\}$}\label{fig:MCWM_trace}
\end{figure}

\begin{figure}
\scalebox{0.8}{\includegraphics{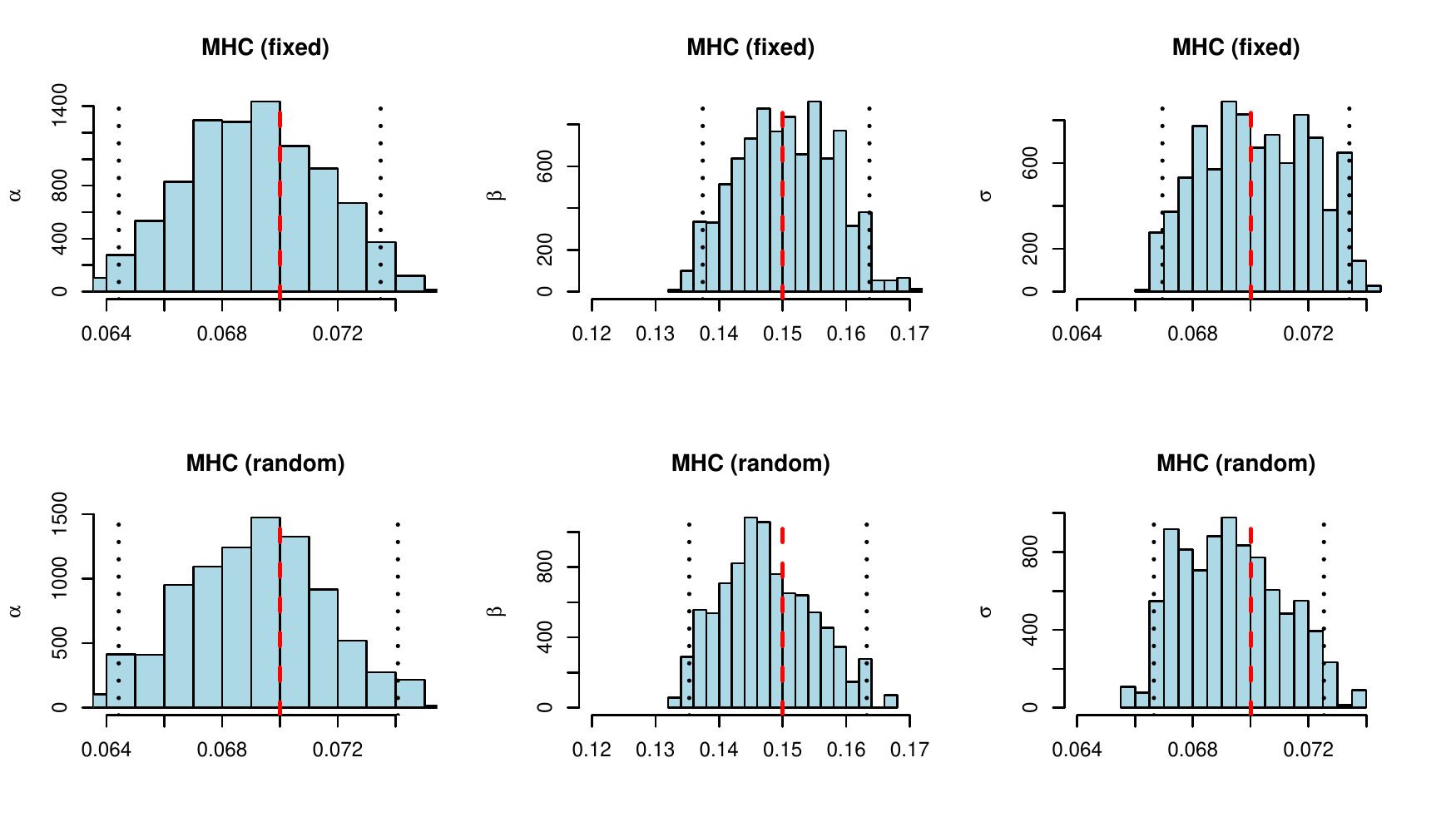}}
\caption{\small Histogram of $9\,000$ MHC iterations (after $1\,000$ burnin) with $nrep=1$}
\label{fig:MHC_hist_CIR1}
\end{figure}

\begin{figure}
\scalebox{0.8}{\includegraphics{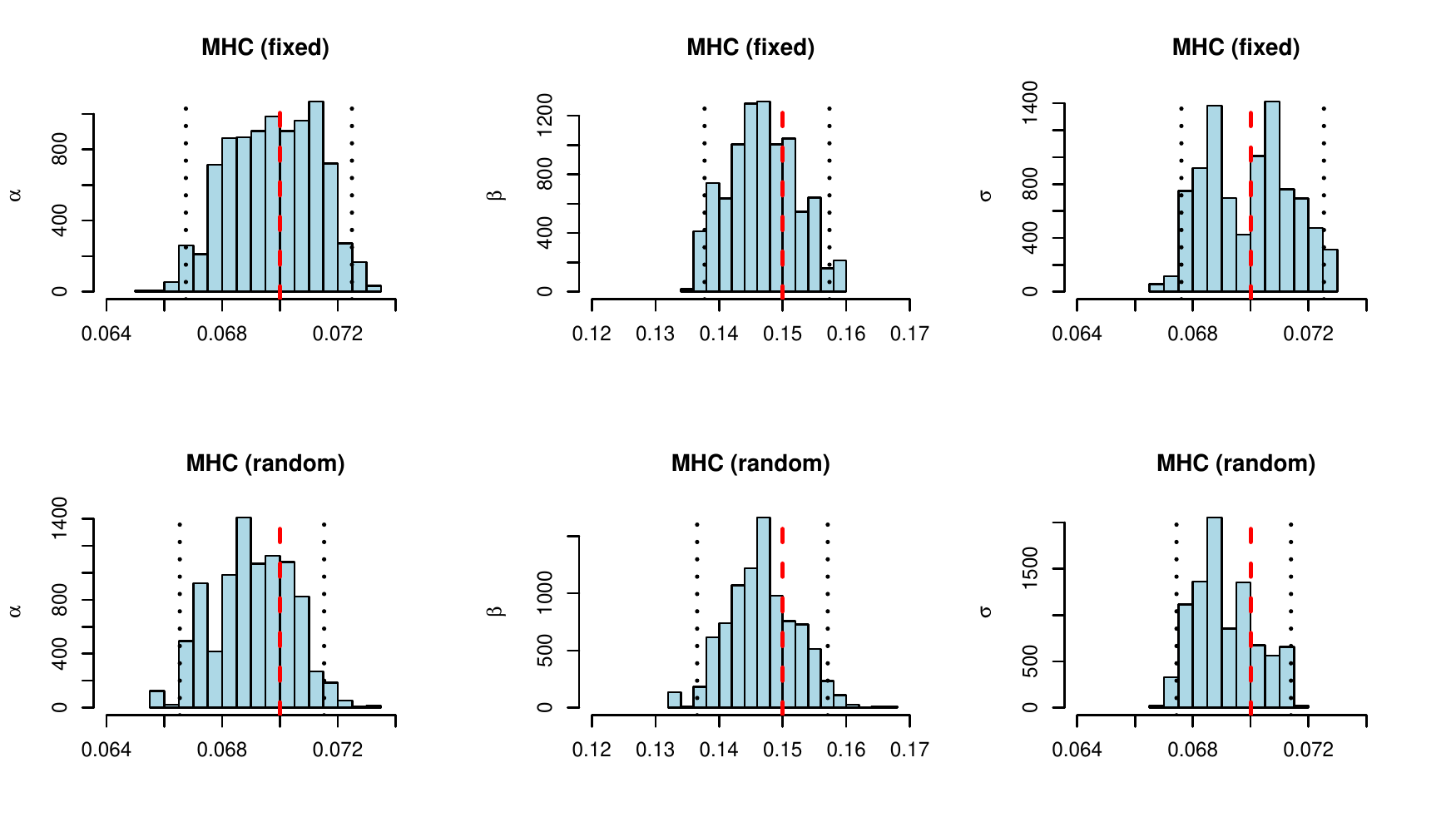}}
\caption{\small Histogram of $9\,000$ MHC iterations (after $1\,000$ burnin) with $nrep=5$}\label{fig:MHC_hist_CIR2}
\end{figure}

\clearpage

\section{The Lotka-Volterra Model: Further Details}
{\subsection{Timing Comparisons}\label{sec:lv_timing}
The complexity of MHC depends on the complexity of the classifier as well as on how costly it is to simulate fake data.
This will be problem-specific. For example, for the Lotka-Volterra model, we have used the Gillespie algorithm \cite{gillespie} which can be quite costly.  This will have some implication for the algorithm of   \cite{pham} which generates two  (not just one)  fake data sets at each step. This will be slower than our approach (which generates just one fake dataset and uses observed data for contrasting)  even when $m=n$.
In order to get a more concrete idea about the dependence on $n,m$ and $p$ (which depends on the length of the time series), we have measured the cost of a single iteration of MHC for various $m, n$ and $p$ for the default implementation of \texttt{cv.glmnet} (10 fold cross-validation) and \texttt{randomForests} (500 trees). The computing times are in Figure \ref{fig:times}. The default implementation of \texttt{glmnet} appears to scale less favorably with $n$ compared to random forests and the complexity, of course, increases with $m$. The method of \cite{pham} requires simulating two (as opposed to one) fake dataset at each step and is, thereby, slower.  This seemingly minor timing gap can aggregate in long Monte Carlo simulations. For example,  $10\,000$ iterations of MHC with default random forests took $2.5$ hours for $n=m=20$, where \cite{pham} takes more than $6$ hours with the same classifier and $n=m=20$.
  This gap is particularly prominent when $p$ (i.e. the length of the time series) is large.   Random forests scale less favorably with $p$, compared to \texttt{glmnet} logistic regression.

\begin{figure}[!t]
\scalebox{0.45}{\includegraphics{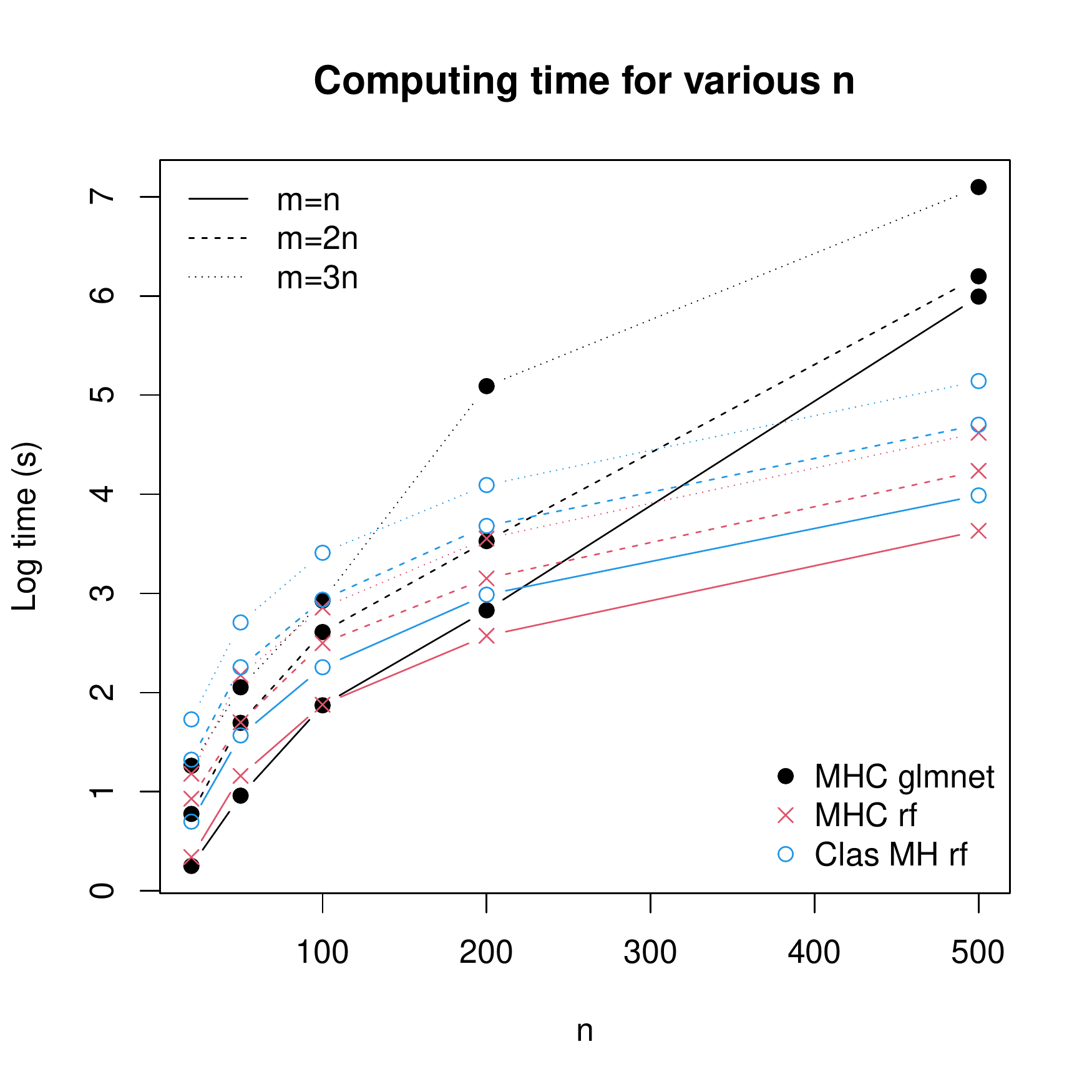}}\scalebox{0.45}{\includegraphics{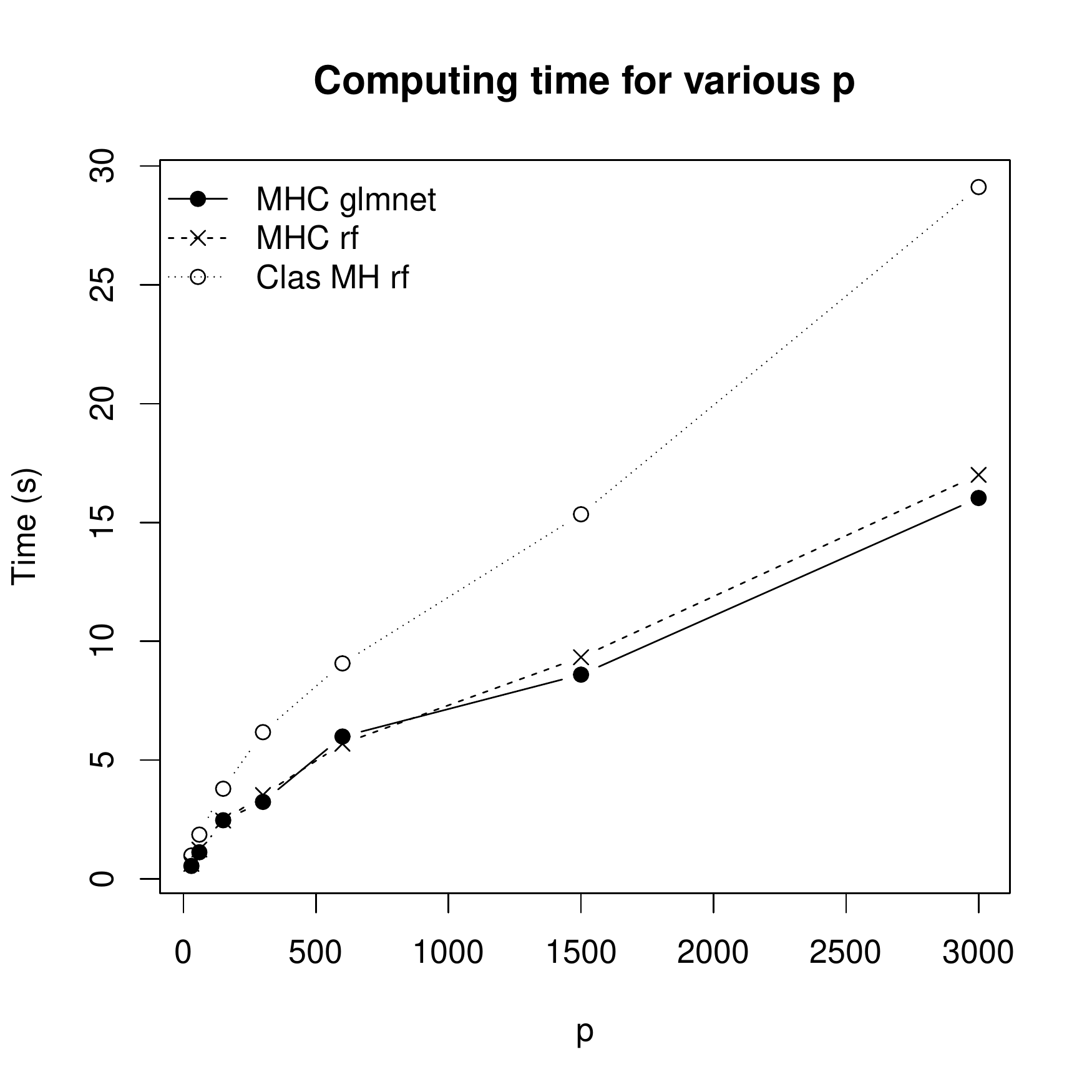}}
\caption{Computing times for one iteration of MHC and Clas MH (Classification MH of \cite{pham}) in the Lotka-Volterra example. (Left) Fixed $p=603$ and various $n$ and $m$ (fake data sample size). (Right) Fixed $n=20$ and various $p$ (depending on the length of the time series).}\label{fig:times}
\end{figure}
  
Our LASSO implementation uses \texttt{glmnet} \citep{glmnet} where the complexity depends on the number of penalty parameters and the number of iterations of the inner coordinate ascent algorithm. As shown in Section 3 of \citep{glmnet}, the \texttt{glmnet} algorithm for logistic regression has three nested loops. For each penalty parameter, one performs a penalized variant of iterated reweighted least squares. Because the weights are changing throughout the iterations, one cannot use faster covariance updates (Section 2.2 in \cite{glmnet}) and each inner cycle thereby costs $O(np)$. The complexity (without cross-validation) thus depends on the number of re-weighting steps, the number of inner  iteration cycles and the length of the regularization path.  

\subsection{The effect of $m$ and $nrep$}\label{sec:effect_m_nrep}

We found that  computing the classification estimator separately for $nrep$ many fake data (using observed data as a reference) and averaging them  out stabilizes estimation. Since our MHC approach uses real observed data for contrasting, it will have the limitation that the choice of $m$ cannot be much larger than $n$ in order for the classification to yield good results. Indeed, we found that for small $n$, increasing $m$ does help as long as it is not overly large to make the classification problem too imbalanced. This can be seen from Figure \ref{fig:varying} below where using $n=20$ and $m=1\,000$ yielded unstable classification (using cross-validation and the \texttt{glmnet} classifier). Averaging over $nrep$ log-likelihood estimators is a heuristic for stabilizing estimation when $n$ is small and, thereby, $m$ cannot be chosen overly large. In addition, while increasing $m$ may result in estimators which concentrate more sharply around the truth, averaging out   $nrep$  estimators will result in a smoother final estimator.

\begin{figure}[!t]
\scalebox{0.45}{\includegraphics{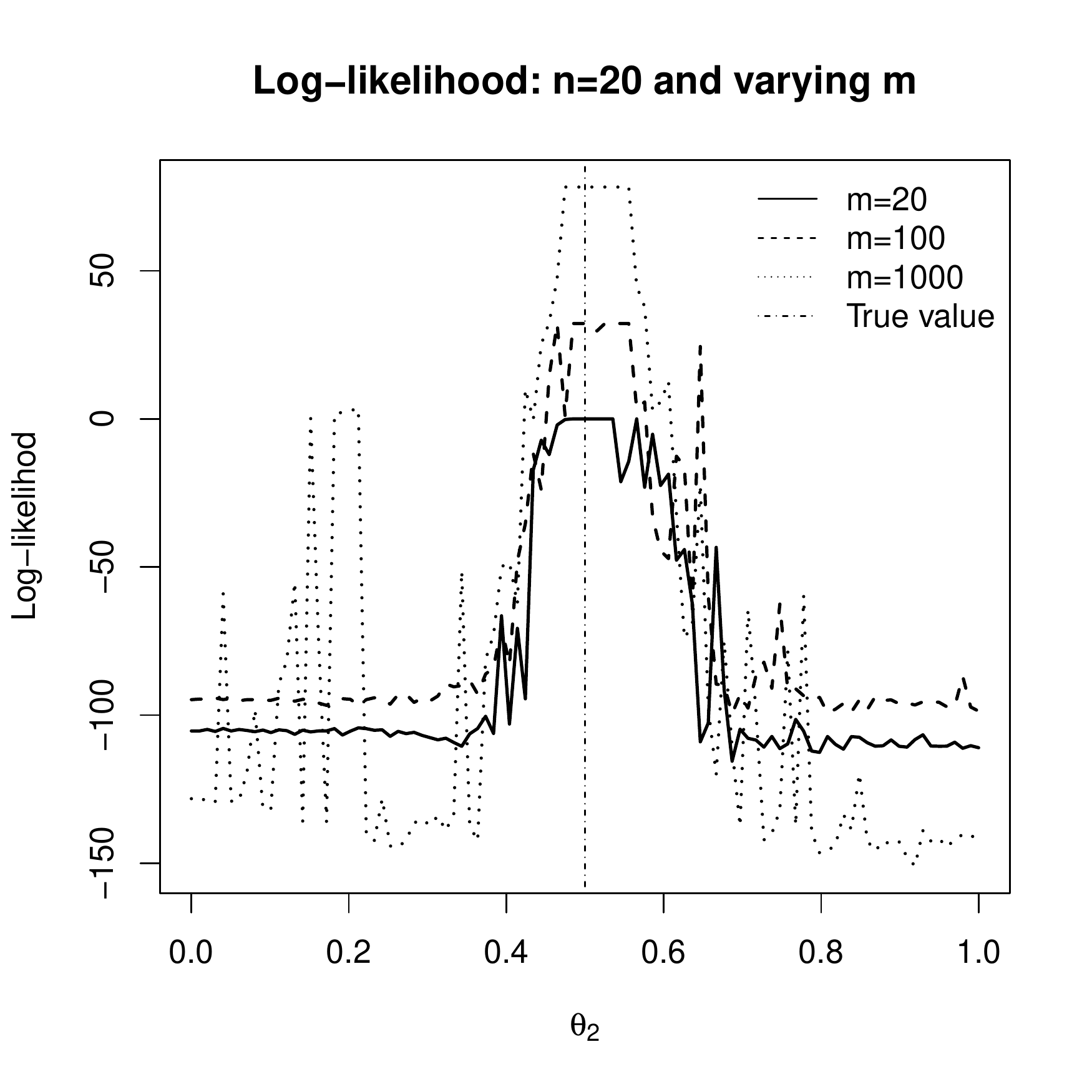}}\scalebox{0.45}{\includegraphics{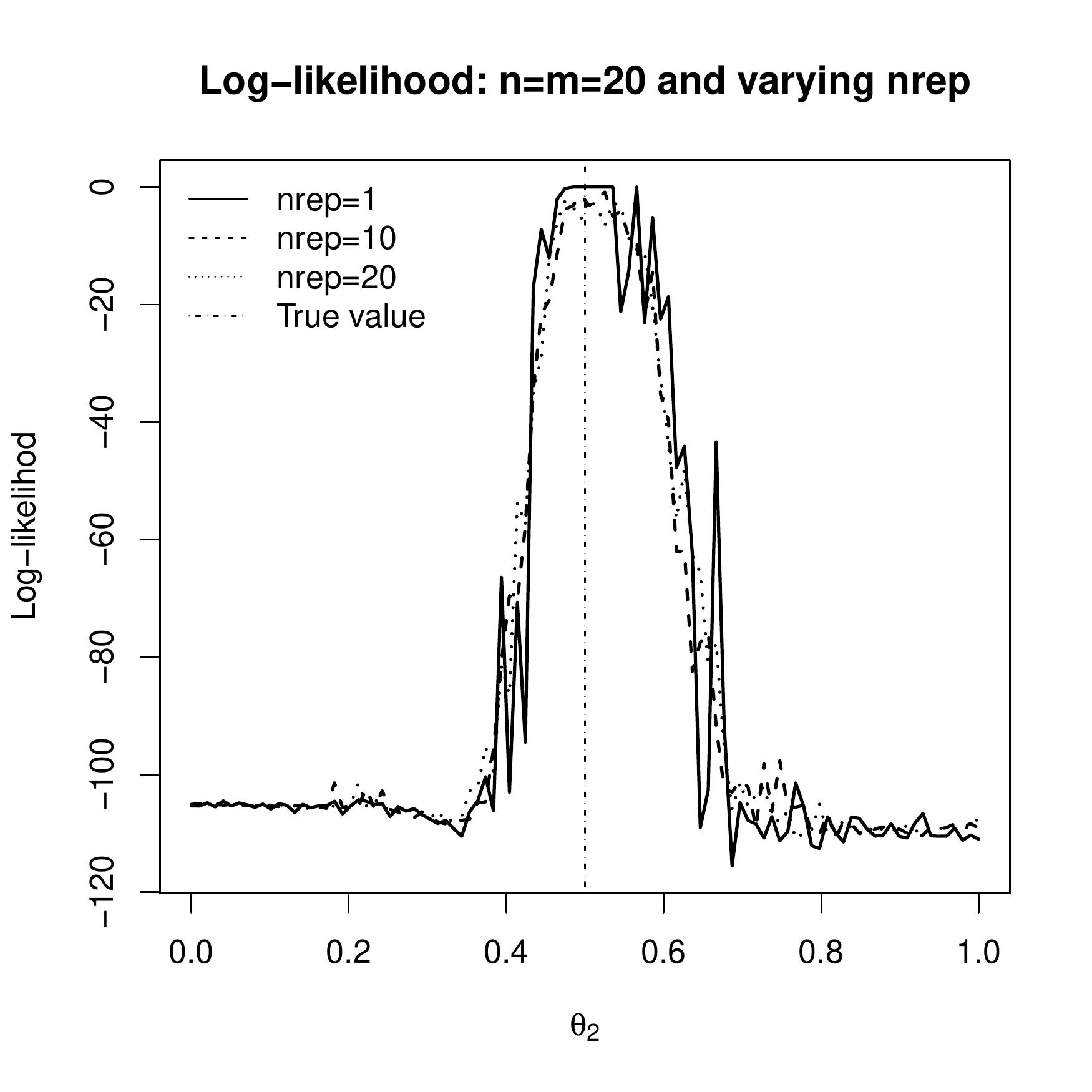}}
\caption{Log-likelihood estimators for varying $m$ and $nrep$ and fixed $n=20$.}\label{fig:varying}
\end{figure}

\begin{figure}[!t]
\scalebox{0.45}{\includegraphics{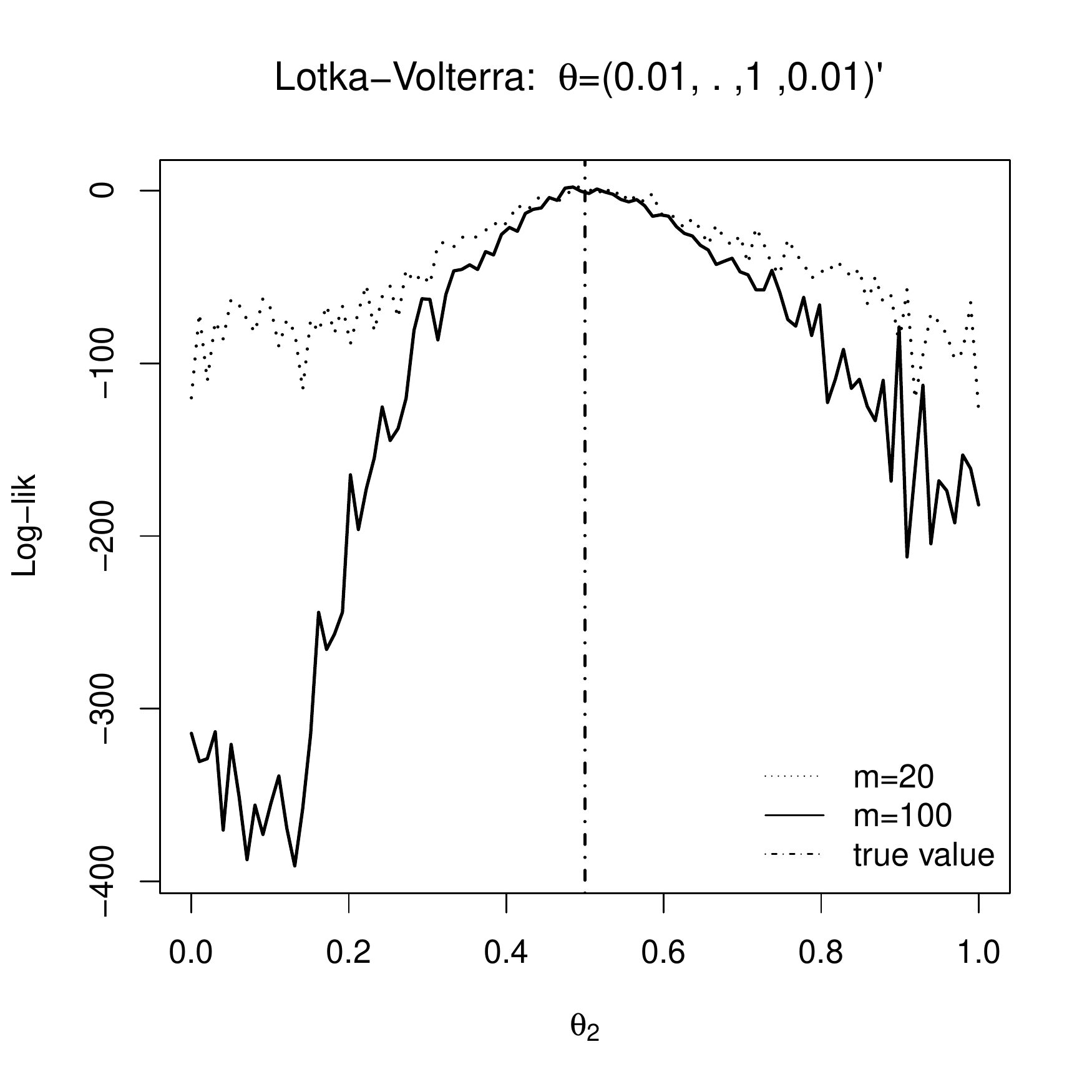}}
\scalebox{0.45}{\includegraphics{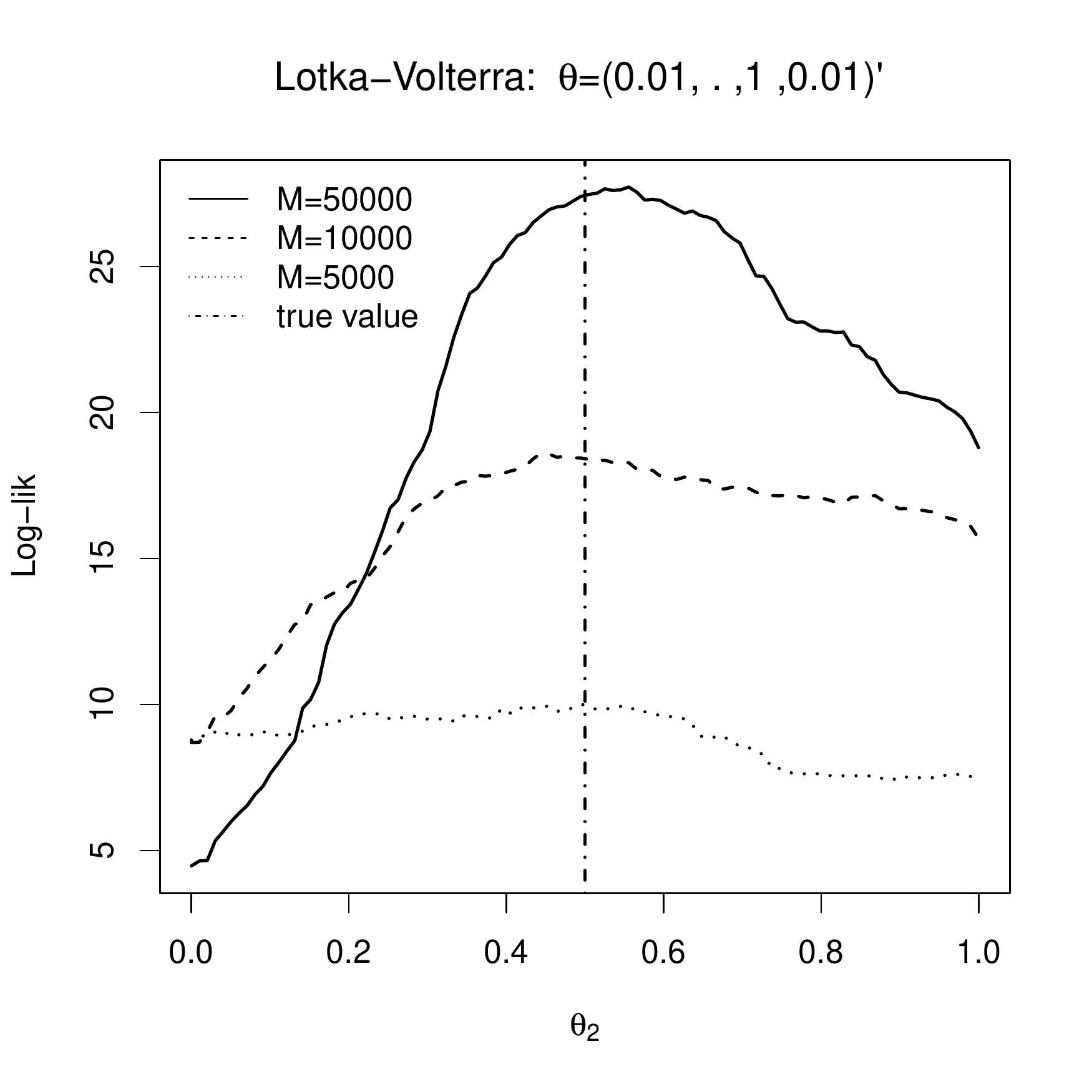}}
\caption{({\bf Conditional versus Marginal Reference}) Plot of estimated log-likelihood as a function of $\theta_2$, keeping all the other parameters at the truth. (Left) The conditional approach of \cite{pham} using various $m$ and using the default random forest classifier (R package \texttt{randomForest}), (Right) the marginal approach of \cite{hermans} using various $m$ and the random forest classifier.} \label{fig:lliks_new}
\end{figure}
\subsection{Comparisons}\label{sec:lv_comparisons}
Referees suggested comparisons with other classification MCMC approaches which use the conditional likelihood with fake data as a reference \cite{pham} or the marginal likelihood as a reference \cite{hermans}. 
We explore the extent to which using the conditional fixed reference (i.e. the observed data) in our MHC approach  is beneficial. 
\cite{hermans}  point out that using a fixed reference point might be problematic if there is not enough overlap between the conditional densities.
MHC uses the truth (i.e. the real data) as the  fixed reference, tacitly assuming that  if the Markov chain is initialized in the vicinity of the truth, the lack of overlap between the two likelihood densities would not be  a practical concern.
We anticipated that using other fixed reference point (i.e. not contrasting agains observed data) might increase variance in the random generator design since the fake reference data would introduce extra randomness. This is indeed the case when looking at the width of the $95\%$ credible interval in Table 3 (comparing MHC with random forests and  Classif MH of \cite{pham} with $n=m=20$). The only difference between these two methods is that  \cite{pham} generates another set of fake data as a reference.

In particular, the method of \cite{pham} directly computes the likelihood ratio of the new versus old proposed value by contrasting two fake datasets without any fixed reference. We have implemented their approach which uses a random forest discriminator (the default \texttt{randomForest} setting in R). The plot of  the estimated log-likelihod (a variant of Figure \ref{fig:varying} on the right) is depicted  in the  left panel of Figure \ref{fig:lliks_new}. We tried fake datasets of size $m=n=20$ and $m=100$.  Learning, of course, improves with increased $m$ but at much increased computational cost (see Section \ref{sec:lv_timing}).

\cite{hermans}  suggest a marginal model  trained ahead of the Monte Carlo simulation which compares dependent and independent data-parameter pairs. A related marginal technique is in \cite{gutman2}. We applied the technique of \cite{hermans} using, again, the default  random forest classifier.
Due to the compact support of the parameters (a rather small subset of the cube $[0,1]^4$), we can learn the likelihood surface quite well. If the parameters had an unbounded support, very many observations-parameter pairs would  need to be generated and this would drastically increase the learning time. For example, \cite{hermans} use  $1$ million training samples. However, performing  random forests on such a large dataset would not be practical. For our Lotka-Volterra model, we trained the classifier using $m=10\, 000$   and  $m=50\,000$  (which took roughly $2.7$ hours). Additional time is needed for the actual MCMC sampling.

To see the effect of the fake data-set size $m$ on the estimator of the (log)-likelihood, we  plot the estimator obtained using  the marginal reference  \cite{hermans} in Figure \ref{fig:lliks_new} in the right panel. 
We found that for the marginal approach, the interaction terms between parameters and data are essential for obtaining good prediction. This is why we did not choose the LASSO but a non-linear random forest classifier.
For the marginal approach, we try $m\in\{5\,000,10\,000,50\,000\}$ using the default implementation of random forests (R package \texttt{randomForests}). With enough training samples, the estimator is quite smooth. However, as will be seen from histograms and traceplots (Figure \ref{fig:hists} and Figure \ref{fig:trace1} below) there is certain bias in the posterior reconstruction.
 Choosing $m=10\,000$, the estimator still peaks around the truth but is wigglier. The conditional approach of \cite{pham} also yields estimators peaked around the truth. The shape is similar to our fixed reference approach using random forests (Figure \ref{fig:varying} on the right).  However, both of these plots  yield curves that are not nearly as peaked as with the \texttt{glmnet} classifier. This has at least two implications: (1) the Metropolis-Hastings with the \texttt{glmnet} classifier will be far more sensitive to initializations where we need to perhaps run ABC or other pilot run to obtain a satisfactory guess (see Figure \ref{fig:trace3}), (2) if initialized properly and if the chain mixes well, the \texttt{glmnet} classifier might provide tighter credible intervals. The choice of the proposal distribution will be also important and it should reflect the curvature of these likelihood shapes.
 

\begin{figure}[!t]
\vspace{-1cm}\centering
\begin{subfigure}[b]{0.45\textwidth}
\scalebox{0.4}{\includegraphics{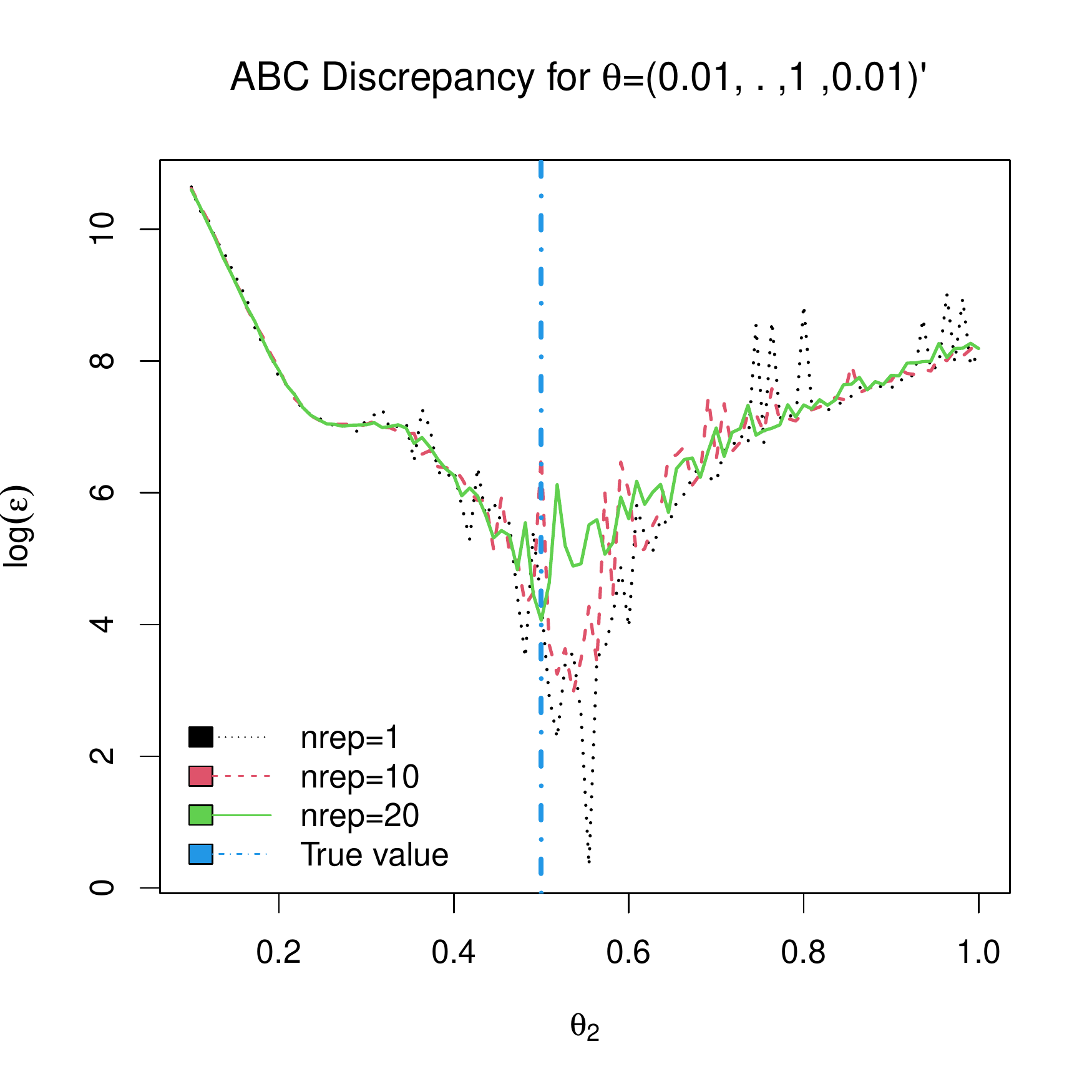}}
\caption{ABC tolerance $\varepsilon$}\label{fig:predator_abc}
\end{subfigure}
\begin{subfigure}[b]{0.45\textwidth}
\scalebox{0.4}{\includegraphics{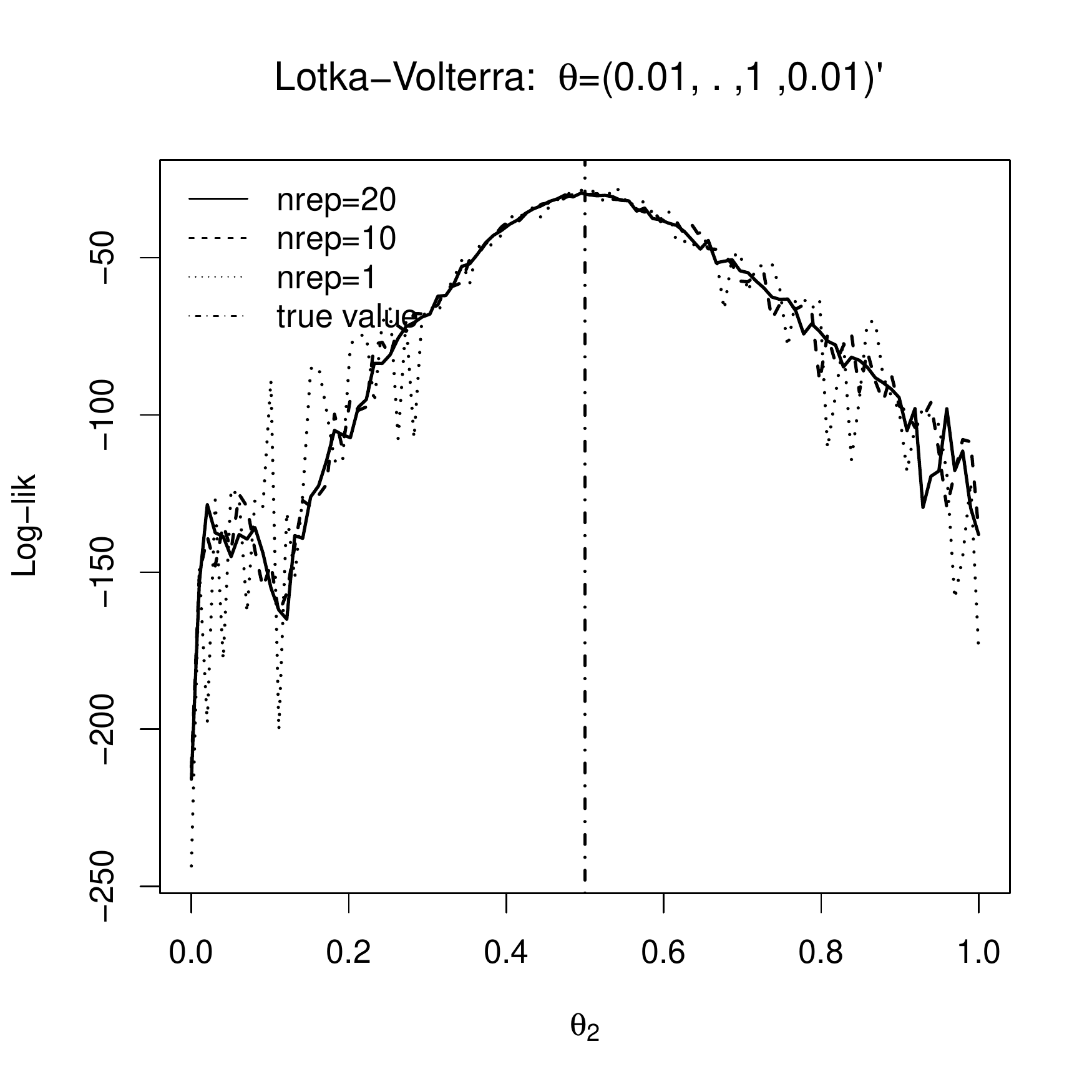}}
\caption{Log-lik  Estimator {\scriptsize (\texttt{rf})}}\label{fig:predator_liks}
\end{subfigure}
\caption{\small  Lotka-Volterra model. ABC discrepancy $\epsilon$ and the classification-based log-likelihood `estimator' $\eta$ using observed data as a reference with  the \texttt{randomForest} classifier.}\label{fig:lotka_summaries}
\end{figure}

 To see whether our ABC summary statistics are able to capture the oscillatory behavior (at different frequencies) and distinguish it from exploding population growth, we have plotted the squared $\|\cdot\|_2$ distance of the summary statistics\footnote{Out of curiosity, we have considered a single fake dataset as well as the average tolerance over $nrep$ fake data replications. } (i.e. the ABC tolerance threshold $\epsilon$) relative to the real data for a grid of values $\theta_2$, fixing the rest at the true values $\theta_1^0=0.01,\theta_3^0=1,\theta_4^0=0.01$ (see Figure \ref{fig:predator_abc}). We can see a V-shaped evolution of $\epsilon$ reaching a minimum near the true value $\theta_2^0=0.5$, especially for $nrep=20$. This creates hope that ABC based on these summary statistics has the capacity to provide a reliable posterior reconstruction. Contrastingly,  we have plotted the estimated log-likelihood  $\eta\equiv \sum_{i=1}^n\log[(1-\hat D(\bm x_i))/\hat D(\bm x_i)]$ (as a function of $\theta_2$) where $\bm x_i=(X^i_1,\dots, X^i_T, Y^i_1,\dots, Y^i_T)'$ after training the LASSO-penalized logistic regression classifier (Figure \ref{fig:varying} on the right) on $m=n$ fake data observations $\wt{\bm x}_i=
    (\wt X^i_1,\dots, \wt X^i_T, \wt Y^i_1,\dots, \wt Y^i_T)'$ for $1\leq i\leq m$ using the cross-validated penalty $\lambda$ (using the R package \texttt{glmnet}). {We also use the default implementation of random forests using the R package \texttt{randomForest} (Figure \ref{fig:lotka_summaries} on the right). We can see that random forests provide estimators which are not as sharply peaked, suggesting less sensitivity to Markov chain initialization. }

The trace-plots of MHC and the approaches of \cite{pham} and \cite{hermans}  are in Figures \ref{fig:trace1}, \ref{fig:trace2} and \ref{fig:trace3}).  
Figure \ref{fig:ABC1} portrays histograms of ABC samples (top $r=100$ out of $M=10\,000$ in the upper panel and top $r=1\,000$ out of $M=100\,000$ in the lower panel).
Finally, Figure \ref{fig:hists} shows histograms of MH samples (MHC, Classification MCMC of \cite{pham} and ALR MH approach of \cite{hermans}).
 }

 \begin{figure}[!h]
{\includegraphics[width=15cm,height=5cm]{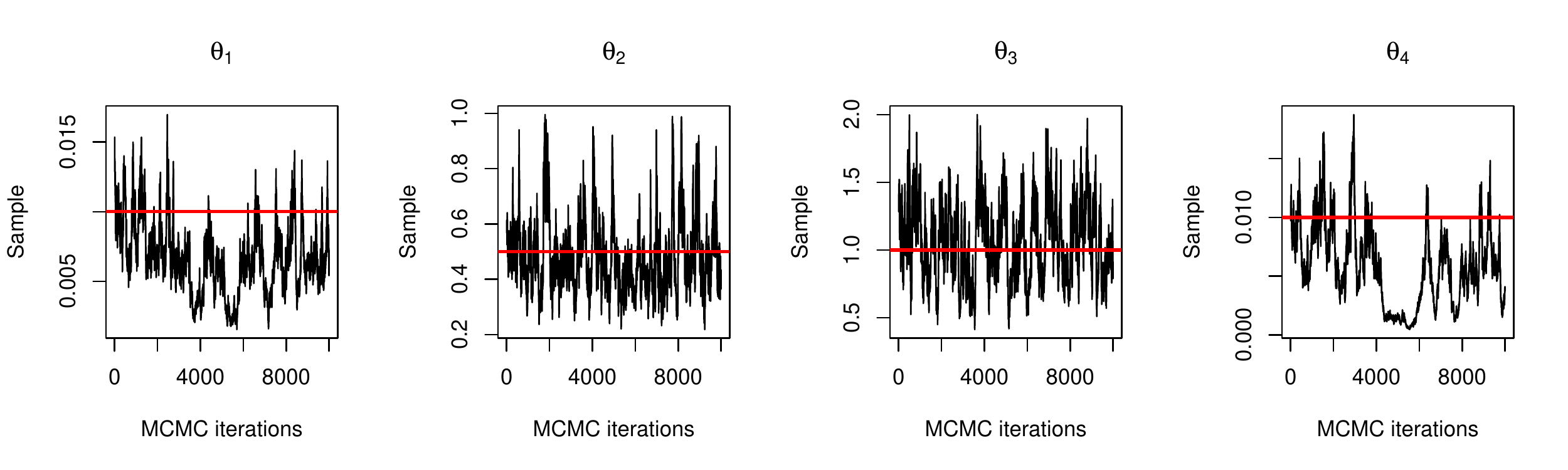}}\\
{\includegraphics[width=15cm,height=5cm]{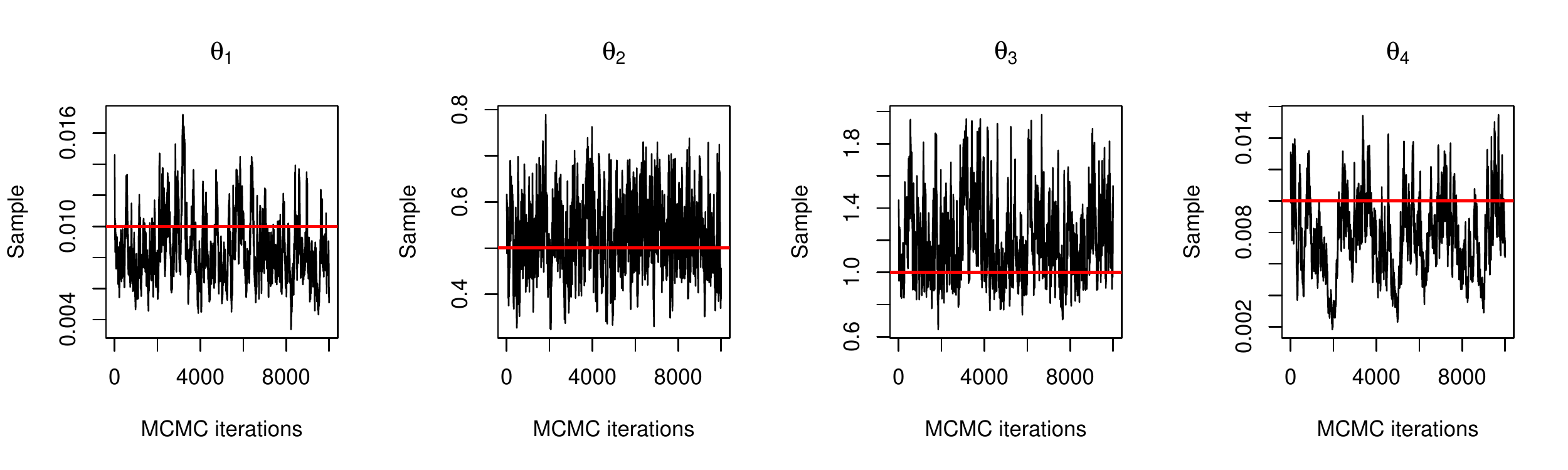}}\\
\caption{Traceplots of ALR MH of \cite{hermans} with $m=10\,000$ (top) and $m=50\,000$ (bottom)}\label{fig:trace1}
{\includegraphics[width=15cm,height=5cm]{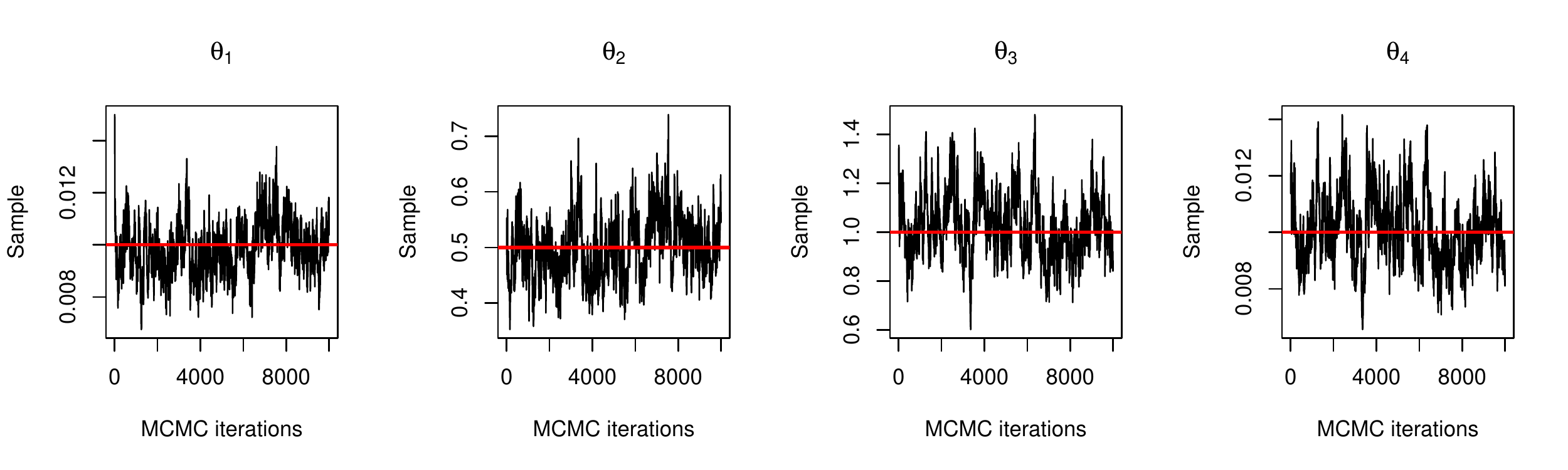}}\\
{\includegraphics[width=15cm,height=5cm]{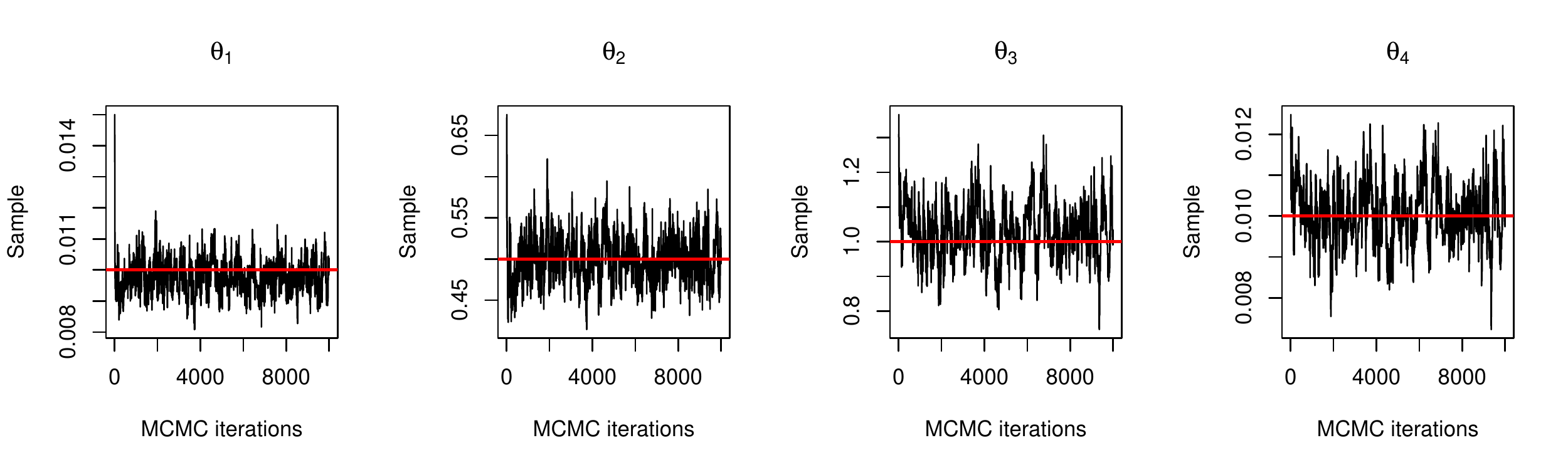}}\\
\caption{Traceplots of Classif MH of \cite{pham} with $m=20$ (top) and $m=100$ (bottom)}\label{fig:trace2}
\end{figure}
\begin{figure}[!h]
{\includegraphics[width=15cm,height=5cm]{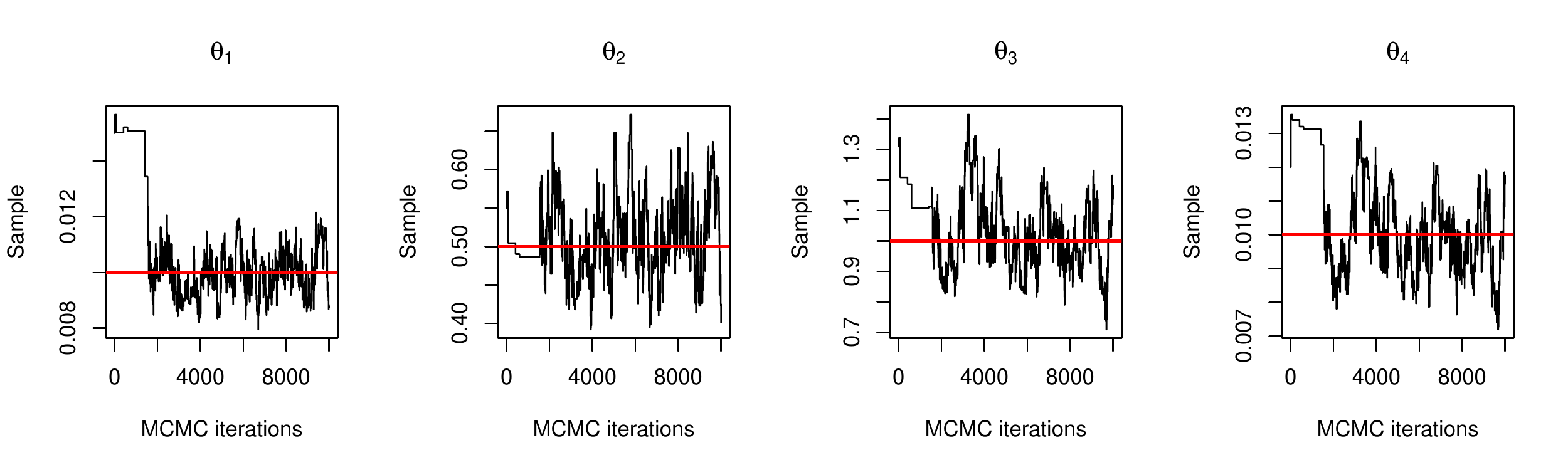}}\\
{\includegraphics[width=15cm,height=5cm]{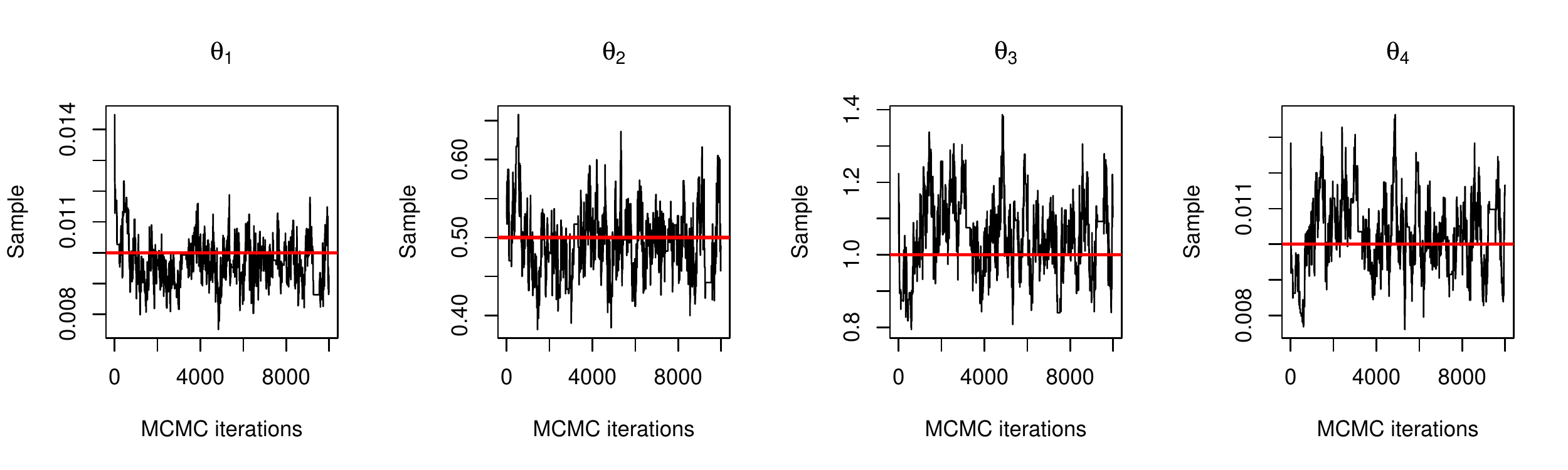}}\\
\caption{Traceplots of MHC with \texttt{glmnet} (top) and random forests (bottom)}\label{fig:trace3}
 \end{figure}

\begin{figure}[!h]
\scalebox{0.9}{\includegraphics{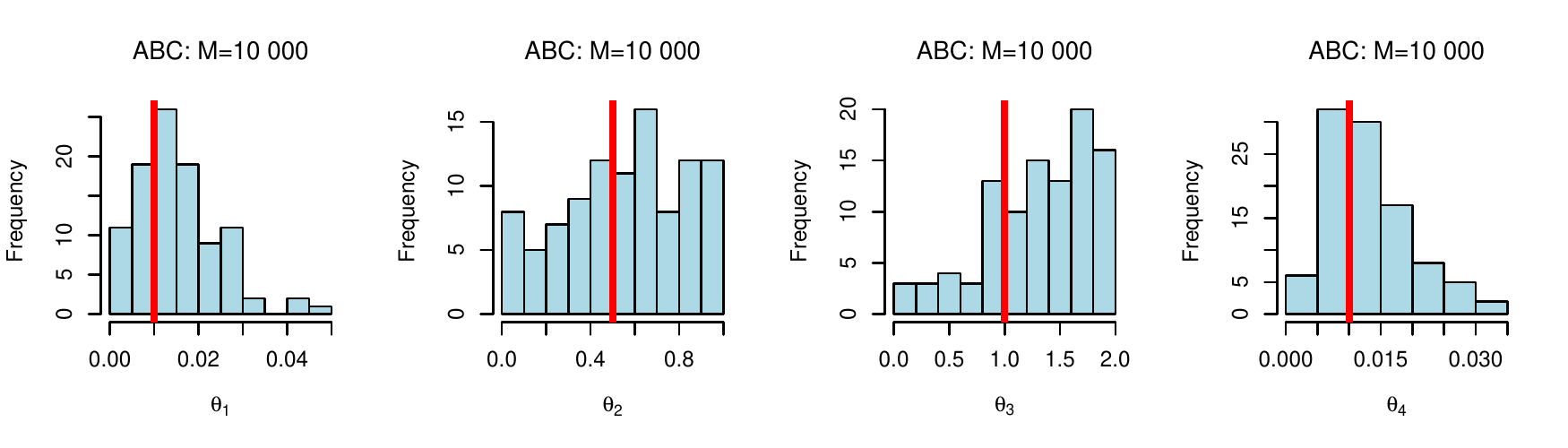}}
\scalebox{0.9}{\includegraphics{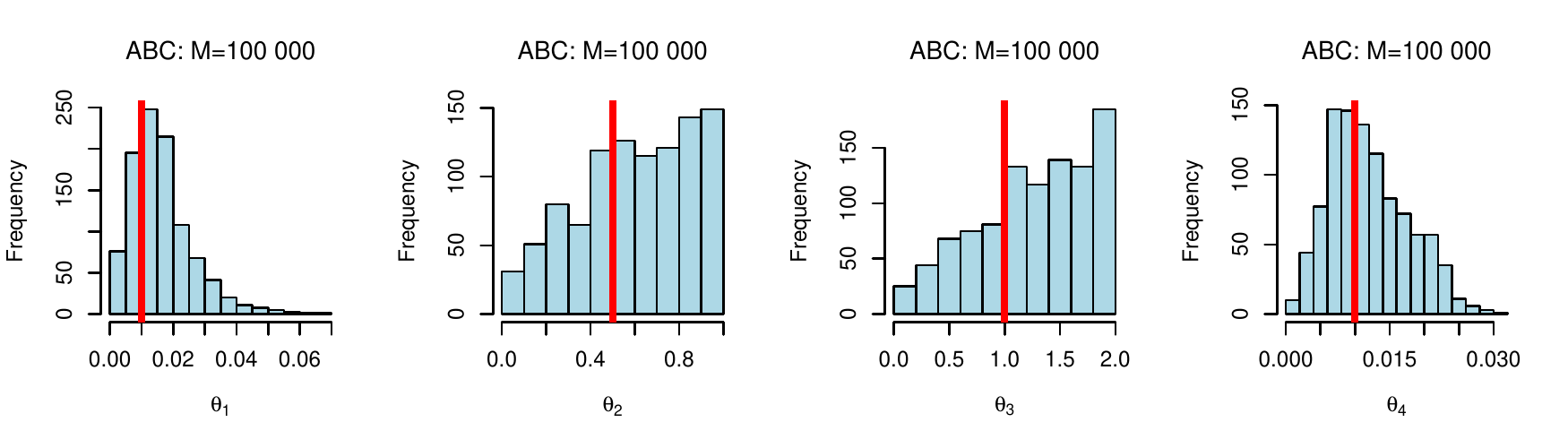}}
\caption{ \small ABC analysis of  the Lotka-Volterra model. Upper panel uses $M=10\,000$ and $r=100$ whereas the lower panel uses $M=100\,000$ and $r=1\,000$. Vertical red lines mark the true values.}\label{fig:ABC1}
\end{figure}

\begin{figure}[!h]
\scalebox{0.6}{\includegraphics{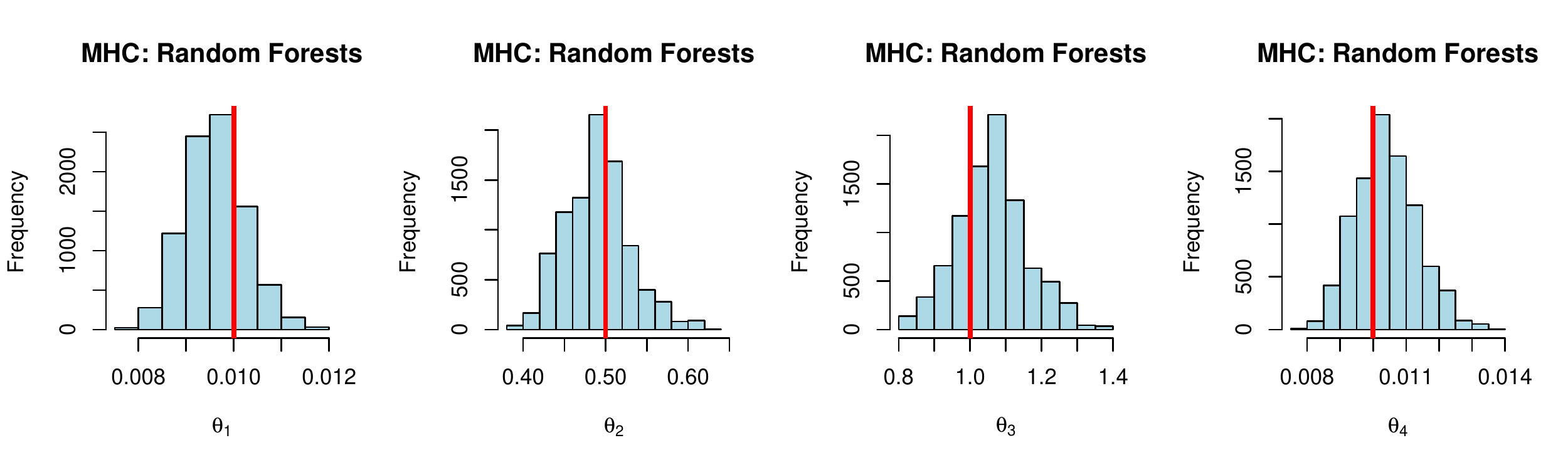}}
\scalebox{0.6}{\includegraphics{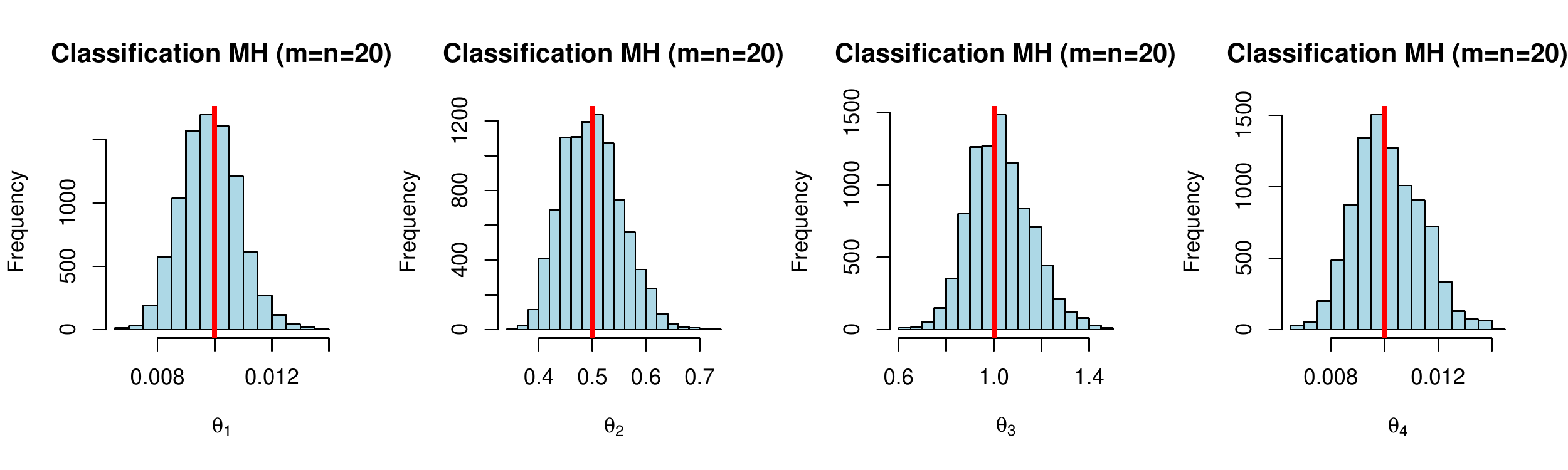}}
\scalebox{0.6}{\includegraphics{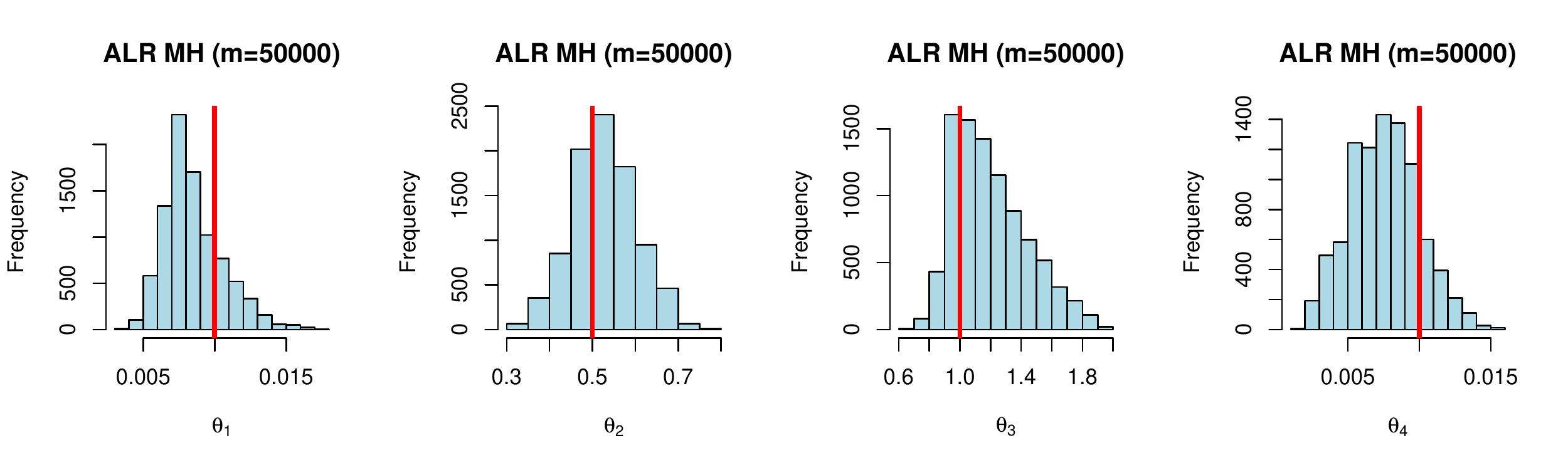}}
\caption{ \small MH analysis of  the Lotka-Volterra model ($9\,000$ MCMC iterations after $1\,000$ burnin). Upper panel shows results for MHC with random forests,  the middle panel uses the classification MCMC approach of \cite{pham} (using $m=n=20$) and the lower panel is the ALR MH approach of \cite{hermans}.  Vertical red lines mark the true values.}\label{fig:hists}
\end{figure}

\clearpage



\end{document}